\documentclass[a4paper, 10pt]{amsart}

\usepackage{arydshln}
\usepackage[latin1]{inputenc}
\usepackage[shortlabels]{enumitem}
\usepackage{nameref}
\usepackage[colorlinks=true]{hyperref}
\usepackage[nospace,noadjust]{cite}
 \usepackage[T1]{fontenc}
\usepackage{lastpage, setspace}
\usepackage{graphicx, pst-all}
\usepackage[all]{xy}
\usepackage{amsmath, amsthm, amssymb, tensor}
\usepackage{listings}
\usepackage{fancyhdr}
\usepackage{verbatim}
\usepackage{pdflscape}
\usepackage{color}
\usepackage{amsmath,bm}

\makeatletter
\newcommand{\displaybump}{\hbox to \@totalleftmargin{\hfil}}
\makeatother

\setlist[enumerate]{leftmargin=1cm}

\psset{linewidth=0.4pt}
\psset{arrowsize=3pt}

\SelectTips{cm}{}
\newdir{ >}{{}*!/-5pt/\dir{>}}

\newcommand{\bigast}{\operatorname{\scalebox{1.7}{\raisebox{-0.1ex}{$\ast$}}}}%

\newcommand{\lhdsim}{\ooalign{\kern0.5ex\raisebox{0.5ex}{$\lhd$}\cr\kern0.5ex\scalebox{1}{\raisebox{-0.65ex}{$\sim$}}\kern0.5ex}}%

\DeclareMathOperator\Aut{Aut}

\DeclareMathOperator\Sym{Sym}

\DeclareMathOperator\PGL{PGL}

\DeclareMathOperator\GL{GL}
\DeclareMathOperator\AGL{AGL}
\DeclareMathOperator\SL{SL}

\DeclareMathOperator\Id{Id}
\DeclareMathOperator\id{id}

\DeclareMathOperator\image{im}

\DeclareMathOperator\pr{pr}

\DeclareMathOperator\rank{rank}
\DeclareMathOperator\Alt{Alt}

\DeclareMathOperator\bbF{\mathbb{F}}

\DeclareMathOperator\bbN{\mathbb{N}}

\DeclareMathOperator\bbQ{\mathbb{Q}}

\DeclareMathOperator\bbZ{\mathbb{Z}}

\DeclareMathOperator\frakg{\mathfrak{g}}

\DeclareMathOperator\calH{\mathcal{H}}

\DeclareMathOperator\calL{\mathcal{L}}

\DeclareMathOperator\calM{\mathcal{M}}
\DeclareMathOperator\calN{\mathcal{N}}
\DeclareMathOperator\calO{\mathcal{O}}
\DeclareMathOperator\calP{\mathcal{P}}

\DeclareMathOperator\calU{\mathcal{U}}

\theoremstyle{definition}
\newtheorem{theorem}{Theorem}[section]
\newtheorem*{theorem*}{Theorem}
\newtheorem{lemma}[theorem]{Lemma}
\newtheorem*{lemma*}{Lemma}
\newtheorem{corollary}[theorem]{Corollary}
\newtheorem{definition}[theorem]{Definition}
\newtheorem*{definition*}{Definition}
\newtheorem{remark}[theorem]{Remark}
\newtheorem*{remark*}{Remark}
\newtheorem{proposition}[theorem]{Proposition}
\newtheorem*{proposition*}{Proposition}
\newtheorem{example}[theorem]{Example}
\newtheorem*{example*}{Example}
\newtheorem*{sketch of proof}{Sketch of Proof}
\newtheorem*{idea of proof}{Idea of Proof}
\newtheorem{question}[theorem]{Question}

\newtheorem{conjecture}[theorem]{Conjecture}

\makeatletter
\newtheorem*{rep@theorem}{\rep@title}
\newcommand{\newreptheorem}[2]{%
\newenvironment{rep#1}[1]{%
 \def\rep@title{#2 \scshape \ref{##1}}%
 \begin{rep@theorem}}%
 {\end{rep@theorem}}}
\makeatother

\newreptheorem{theorem}{Theorem}
\newreptheorem{proposition}{Proposition}
\newreptheorem{corollary}{Corollary}
\newreptheorem{lemma}{Lemma}
\newreptheorem{conjecture}{Conjecture}
\newreptheorem{definition}{Definition}

\title[Groups Acting On Trees With Prescribed Local Action]{Groups Acting on Trees \\ With Prescribed Local Action}
\author{Stephan Tornier}
\date{\today}

\begin{document}

\begin{abstract}
We extend Burger--Mozes theory of closed, non-discrete, locally quasiprimitive automorphism groups of locally finite, connected graphs to the semiprimitive case, and develop a generalization of Burger--Mozes universal groups acting on the regular tree $T_{d}$ of degree $d\in\mathbb{N}_{\ge 3}$. Three applications are given: First, we characterize the automorphism types which the quasi-center of a non-discrete subgroup of $\Aut(T_{d})$ may feature in terms of the group's local~action. In doing so, we explicitly construct closed, non-discrete, compactly generated subgroups of $\Aut(T_{d})$ with non-trivial quasi-center, and see that Burger--Mozes theory does not extend 
further to the transitive case. We then characterize the $(P_{k})$-closures of locally transitive subgroups of $\Aut(T_{d})$ containing an involutive inversion, and thereby partially answer two questions by Banks--Elder--Willis. Finally, we offer a new view on the Weiss conjecture.
\end{abstract}

\maketitle

\section*{Introduction}

In the structure theory of locally compact (l.c.) groups, totally disconnected (t.d.) ones are in the focus because any locally compact group $G$ is an extension of its connected component $G_{0}$ by the totally disconnected quotient $G/G_{0}$,
\begin{displaymath}
 \xymatrix{
  1 \ar[r] & G_{0} \ar[r] & G \ar[r] & G/G_{0} \ar[r] & 1,
 }
\end{displaymath}
and connected l.c. groups have been identified as inverse limits of Lie groups in seminal work by Gleason \cite{Gle52}, Montgomery-Zippin \cite{MZ52} and Yamabe \cite{Yam53}.

Every t.d.l.c. group can be viewed as a directed union of compactly generated open subgroups. Among the latter, groups acting on regular graphs and trees stand out due to the Cayley-Abels graph construction: Every compactly generated t.d.l.c. group $G$ acts vertex-transitively on a connected regular graph $\Gamma$ of finite degree $d$ with compact kernel~$K$. In particular, the universal cover of $\Gamma$ is the $d$-regular tree $T_{d}$ and we obtain a cocompact subgroup $\smash{\widetilde{G}}$ of its automorphism group $\Aut(T_{d})$,
\begin{displaymath}
 \xymatrix{
  1 \ar[r] & \pi_{1}(\Gamma) \ar[r] & \text{$\widetilde{G}$} \ar[r] & G/K \ar[r] & 1,
 }
\end{displaymath}
as an extension of $\pi_{1}(\Gamma)$ by $G/K$, see \cite[Section 11.3]{Mon01} and \cite{KM08} for details.

In studying the automorphism group $\Aut(\Gamma)$ of a locally finite, connected graph $\Gamma=(V,E)$, we follow the notation of Serre \cite{Ser03}. The group $\Aut(\Gamma)$ is t.d.l.c. when equipped with the permutation topology for its action on $V\cup E$, see Section~\ref{sec:permutation_groups}. Given a subgroup $H\le\Aut(\Gamma)$ and a vertex $x\in V$, the stabilizer $H_{x}$ of $x$ in $H$ induces a permutation group on the set $E(x):=\{e\in E\mid o(e)=x\}$ of edges issuing from $x$. We say that $H$ is locally ``X'' if for every $x\in V$ said permutation group satisfies property ``X'', e.g. being transitive, semiprimitive or quasiprimitive. 

\vspace{0.2cm}
In \cite{BM00a}, Burger--Mozes develop a remarkable structure theory of closed, non-discrete, locally quasiprimitive subgroups of $\Aut(\Gamma)$, which resembles the theory of semisimple Lie groups, see Theorem~\ref{thm:burger_mozes_structure_quasiprimitive}. In Section \ref{sec:bm_theory_semiprimitive}, specifically Theorem \ref{thm:burger_mozes_structure_semiprimitive}, we show that this theory readily carries over to the semiprimitive case.

\vspace{0.2cm}
Let $\Omega$ be a set of cardinality $d\!\in\!\bbN_{\ge 3}$ and let $T_{d}\!=\!(V,E)$ be the $d$-regular tree. Burger--Mozes complement their structure theory with a particularly accessible class of subgroups of $\Aut(T_{d})$ with prescribed local action: Given $F\le\Sym(\Omega)$, their universal group $\mathrm{U}(F)$ is closed in $\Aut(T_{d})$, vertex-transitive, compactly generated and locally permutation isomorphic to $F$. It is discrete if and only if $F$ is semiregular. When $F$ is transitive, $\mathrm{U}(F)$ is maximal up to conjugation among vertex-transitive subgroups of $\Aut(T_{d})$ that are locally permutation isomorphic to $F$, hence \emph{universal}.

We generalize the universal groups by prescribing the local action on balls of a given radius $k\in\bbN$, the Burger--Mozes construction corresponding to the case $k\!=\!1$. Equip $T_{d}$ with a labelling, i.e. a map $l:E\to\Omega$ such that for every $x\in V$ the map $l_{x}\!:\! E(x)\!\to\!\Omega, e\!\mapsto\! l(e)$ is a bijection, and $l(e)\!=\!l(\overline{e})$ for all $e\!\in\! E$. Also, fix a tree $B_{d,k}$ which is isomorphic to a ball of radius $k$ around a vertex in the labelled tree $T_{d}$ and let $l_{x}^{k}:B(x,k)\to B_{d,k}$ ($x\in V$) be the unique label-respecting isomorphism. Then
\begin{displaymath}
 \sigma_{k}:\Aut(T_{d})\times V\to\Aut(B_{d,k}),\ (g,x)\to l_{gx}^{k}\circ g\circ (l_{x}^{k})^{-1}
\end{displaymath}
captures the \emph{$k$-local action} of $g$ at the vertex $x\in V$.

\begin{repdefinition}{def:ukf}
Let $F\le\Aut(B_{d,k})$. Define
\begin{displaymath}
 \mathrm{U}_{k}(F):=\{g\in\Aut(T_{d})\mid \forall x\in V:\ \sigma_{k}(g,x)\in F\}.
\end{displaymath}
\end{repdefinition}

While $\mathrm{U}_{k}(F)$ is always closed, vertex-transitive and compactly generated, other properties of $\mathrm{U}(F)$ need not carry over. Foremost, the group $\mathrm{U}_{k}(F)$ need not be locally action isomorphic to $F$; we say that $F\le\Aut(B_{d,k})$ satisfies condition (C) if it is. This can be viewed as an interchangeability condition on neighbouring local actions, see Section \ref{sec:ukf_examples}. There also is a discreteness condition (D) on $F\le\Aut(B_{d,k})$ in terms of certain stabilizers in $F$ under which $\mathrm{U}_{k}(F)$ is discrete, see Section \ref{sec:ukf_discreteness}. Finally, the groups $\mathrm{U}_{k}(\! F)$ are universal in a sense akin to the above by Theorem~\ref{thm:ukf_universal}.

For $\smash{\widetilde{F}\le\Aut(B_{d,k})}$, let $\smash{F:=\pi\widetilde{F}\le\Sym(\Omega)}$ denote its projection to $\Aut(B_{d,1})$, which is naturally permutation isomorphic to $\Sym(\Omega)$ via the labelling of $B_{d,1}$.
The following rigidity theorem is inspired by \cite[Proposition 3.3.1]{BM00a}.

\begin{reptheorem}{thm:ukf_rigid}
\hspace{-0.1cm}Let $F\!\le\!\Sym(\Omega)$ be $2$-transitive and $F_{\omega}$ $(\omega\!\in\!\Omega)$ simple non-abelian. Further, let $\smash{\widetilde{F}\le\Aut(B_{d,k})}$ with $\smash{\pi\widetilde{F}=F}$ satisfy \eqref{eq:C}. Then $\smash{\mathrm{U}_{k}(\widetilde{F})}$ equals either
\begin{displaymath}
 \mathrm{U}_{2}(\Gamma(F)),\quad \mathrm{U}_{2}(\Delta(F))\quad\text{or}\quad\mathrm{U}_{1}(F).
\end{displaymath}
\end{reptheorem}

Here, the groups $\Gamma(F),\Delta(F)\!\le\!\Aut(B_{d,2})$ of Section~\ref{sec:ukf_examples} satisfy both (C) and~(D) and therefore yield discrete universal groups. Illustrating the necessity of the assumptions in Theorem~\ref{thm:ukf_rigid}, we construct further universal groups in the case when either point stabilizers in $F$ are not simple, $F$ is not primitive, or $F$ is not perfect, see e.g. $\Phi(F,N), \Phi(F,\calP),\Pi(F,\rho,X)\le\Aut(B_{d,2})$ in Section~\ref{sec:ukf_examples}.

\vspace{0.2cm}
In Section \ref{sec:applications}, we present three applications of the framework of universal groups. First, we study the quasi-center of subgroups of $\Aut(T_{d})$. The quasi-center $\mathrm{QZ}(G)$ of a topological group $G$ consists of those elements whose centralizer in $G$ is open. It plays a major role in the Burger--Mozes Structure Theorem \ref{thm:burger_mozes_structure_quasiprimitive}: A non-discrete, locally quasiprimitive subgroup of $\Aut(T_{d})$ does not feature any non-trivial quasi-central elliptic elements. We complete this fact to the following local-to-global-type characterization of the automorphism types which the quasi-center of a non-discrete subgroup of $\Aut(T_{d})$ may feature in terms of the group's local action.

\begin{reptheorem}{thm:local_action_qz}
Let $H\le\Aut(T_{d})$ be non-discrete. If $H$ is locally
\begin{enumerate}[(i)]
 \item transitive then $\mathrm{QZ}(H)$ contains no inversion.
 \item semiprimitive then $\mathrm{QZ}(H)$ contains no non-trivial edge-fixating element.
 \item quasiprimitive then $\mathrm{QZ}(H)$ contains no non-trivial elliptic element.
 \item $k$-transitive $(k\in\bbN)$ then $\mathrm{QZ}(H)$ contains no hyperbolic element of length~$k$.
\end{enumerate}
\end{reptheorem}

More importantly, the proof of the above theorem suggests to use groups of the form $\bigcap_{k\in\bbN}\mathrm{U}_{k}(F^{(k)})$ for appropriate local actions $F^{(k)}\le\Aut(B_{d,k})$ in order to \emph{explicitly} construct non-discrete subgroups of $\Aut(T_{d})$ whose quasi-centers contain certain types of automorphisms. This leads to the following sharpness result.

\begin{reptheorem}{thm:local_action_qz_sharp}
There is $d\in\bbN_{\ge 3}$ and a closed, non-discrete, compactly generated subgroup of $\Aut(T_{d})$ which is locally
\begin{enumerate}[(i)]
 \item intransitive and contains a quasi-central inversion.
 \item transitive and contains a non-trivial quasi-central edge-fixating element.
 \item semiprimitive and contains a non-trivial quasi-central elliptic element.
 \item\begin{enumerate}[(a),itemindent=-0.4cm]
        \item \hspace{-0.15cm} intransitive and contains a quasi-central hyperbolic element of length~$1$.
        \item \hspace{-0.15cm} quasiprimitive and contains a quasi-central hyperbolic element of length~$2$.
       \end{enumerate}
\end{enumerate}
\end{reptheorem}

Part (ii) of this theorem can be strengthened to the following result which shows that Burger--Mozes theory does not extend further to locally transitive groups.

\begin{reptheorem}{thm:qz_open}
There is $d\in\bbN_{\ge 3}$ and a closed, non-discrete, compactly generated, locally transitive subgroup of $\Aut(T_{d})$ with open, hence non-discrete, quasi-center.
\end{reptheorem}

We also give an algebraic characterization of the $(P_{k})$-closures of locally transitive subgroups of $\Aut(T_{d})$ which contain an involutive inversion, and thereby partially answer two questions by Banks--Elder--Willis \cite[p. 259]{BEW15}. Recall (Section \ref{sec:independence_simplicity}) that the $(P_{k})$-closure\footnote{The $(P_{k})$-closure was introduced in \cite{BEW15} and called $k$-closure; however the term $k$-closure has an established meaning for permutation groups due to Wielandt, so we use $(P_{k})$-closure here.} $(k\in\bbN)$
of a subgroup $H\le\Aut(T_{d})$ is given by
\begin{displaymath}
 H^{(P_{k})}=\{g\in\Aut(T_{d})\mid \forall x\in V(T_{d})\ \exists h\in H:\ g|_{B(x,k)}=h|_{B(x,k)}\}.
\end{displaymath}

\begin{reptheorem}{thm:k_closure_char}
Let $H\le\Aut(T_{d})$ be locally transitive and contain an involutive inversion. Then $\smash{H^{(P_{k})}=\mathrm{U}_{k}(F^{(k)})}$ for some labelling $l$ of $T_{d}$ and $F^{(k)}\le\Aut(B_{d,k})$.
\end{reptheorem}

Combined with the independence properties $(\!P_{k}\!)$ $(k\!\in\!\bbN)$ (Section \ref{sec:independence_simplicity}), introduced by Banks--Elder--Willis \cite{BEW15} as generalizations of Tits' Independence Property, Theorem \ref{thm:k_closure_char} entails the following characterization of universal groups.

\begin{repcorollary}{cor:ukf_char}
Let $H\le\Aut(T_{d})$ be closed, locally transitive and contain an involutive inversion. Then $H=\mathrm{U}_{k}(F^{(k)})$ if and only if $H$ satisfies Property $(P_{k})$.
\end{repcorollary}

Banks--Elder--Willis use subgroups of $\Aut(T_{d})$ with pairwise distinct $(P_{k})$-closures to construct infinitely many, pairwise non-conjugate, non-discrete simple subgroups of $\Aut(T_{d})$ via Theorem \ref{thm:bew_simplicity} and ask whether they are also pairwise non-isomorphic as topological groups. We partially answer this question in the following theorem.

\begin{reptheorem}{thm:bew_non_isomorphism}
Let $H\le\Aut(T_{d})$ be non-discrete, locally permutation isomorphic to $F\le\Sym(\Omega)$ and contain an involutive inversion. Suppose that $F$ is transitive and that every non-trivial subnormal subgroup of $F_{\omega}$ $(\omega\!\in\!\Omega)$ is transitive on $\Omega\backslash\{\omega\}$. If $H^{(P_{k})}\neq H^{(P_{l})}$ for some $k,l\in\bbN$ then $(H^{(P_{k})})^{+_{k}}$ and $(H^{(P_{l})})^{+_{l}}$ are non-isomorphic.
\end{reptheorem}

Infinitely many families of pairwise non-isomorphic simple groups of this type, each sharing a certain transitive local action, are constructed in Example \ref{ex:bew_example}.

\vspace{0.2cm}
Finally, Section~\ref{sec:view_weiss} offers a new view on the Weiss conjecture \cite{Wei78} which states that there are only finitely many conjugacy classes of discrete, locally primitive and vertex-transitive subgroups of $\Aut(T_{d})$ for a given $d\in\bbN_{\ge 3}$. This conjecture was extended by Poto{\v{c}}nik--Spiga--Verret in \cite{PSV12} to semiprimitive local actions and impressive partial results have been obtained by the same authors as well as Giudici--Morgan \cite{GM14}. We show that under the additional assumption that each group contains an involutive inversion, it suffices to show that for every semiprimitive $F\le\Sym(\Omega)$ there are only finitely many $\smash{\widetilde{F}\le\Aut(B_{d,k})}$ $(k\in\bbN)$ with $\smash{\pi\widetilde{F}=F}$ that satisfy conditions (C) and (D) in a minimal fashion, see Definition~\ref{def:cd_dimension}.

\subsection*{Acknowledgements}

The author is indebted to Marc Burger and George Willis for their support and the suggestion to define generalized universal groups. Thanks are also due to Luke Morgan and Michael Giudici for sharing their insight on permutation groups, and Michael Giudici, for providing a proof of Lemma \ref{lem:2tran_ext}. Two anonymous referees' comments were also much appreciated.

A good part of this research was carried out during visits to The University of Newcastle, Australia, for the hospitality of which the author is thankful. Finally, part of this research was supported by the SNSF Doc.Mobility fellowship 172120 as well as the ARC grants DP120100996, FL170100032 and DE210100180.

\section{Preliminaries}

This section collects preliminaries on permutation groups, graph theory and Burger--Mozes theory. References are given in the respective subsection.

\subsection{Permutation Groups}\label{sec:permutation_groups}

Let $\Omega$ be a set. In this section, we collect definitions and results concerning the group $\Sym(\Omega)$ of bijections of $\Omega$. Refer to \cite{DM96}, \cite{Pra97}, \cite{GM18} and \cite[Section 1.2]{KM08} for details beyond the following.

\vspace{0.2cm}
Let $F\le\Sym(\Omega)$. The \emph{degree} of $F$ is $|\Omega|$. For $\omega\in\Omega$, the \emph{stabilizer} of $\omega$ in $F$ is $F_{\omega}:=\{\sigma\in F\mid \sigma\omega=\omega\}$. The subgroup of $F$ generated by its point stabilizers is denoted by $F^{+}:=\langle\{F_{\omega}\mid\omega\in\Omega\}\rangle$. The permutation group $F$ is \emph{semiregular}, or \emph{free}, if $F_{\omega}=\{\id\}$ for all $\omega\in\Omega$; equivalently, if $F^{+}$ is trivial. It is \emph{transitive} if its action on $\Omega$ is transitive, and \emph{regular} if it is both semiregular and transitive.

Let $F\le\Sym(\Omega)$ be transitive. The \emph{rank} of $F$ is the number $\rank(F):=|F\backslash\Omega^{2}|$ of orbits of the diagonal action $\sigma\cdot(\omega,\omega'):=(\sigma\omega,\sigma\omega')$ of $F$ on $\Omega^{2}$. Equivalently, $\rank(F)=|F_{\omega}\backslash\Omega|$ for all $\omega\in\Omega$. Note that the diagonal $\Delta(\Omega):=\{(\omega,\omega)\mid\omega\in\Omega\}$ is always an orbit of the diagonal action $F\curvearrowright\Omega^{2}$. The permutation group $F$ is \emph{$2$-transitive} if it acts transitively on $\Omega^{2}\backslash\Delta(\Omega)$. In other words, $\rank(F)=2$.

We now define several classes of permutation groups lying in between the classes of transitive and $2$-transitive permutation groups. Let $F\le\Sym(\Omega)$. A partition $\calP:\Omega=\bigsqcup_{i\in I}\Omega_{i}$ of $\Omega$ is \emph{preserved} by $F$, or \emph{$F$-invariant}, if for all $\sigma\in F$ we have $\{\sigma\Omega_{i}\mid i\in I\}=\{\Omega_{i}\mid i\in I\}$. The partitions $\Omega=\Omega$ and $\Omega=\bigsqcup_{\omega\in\Omega}\{\omega\}$ are \emph{trivial}. A map $a:\Omega\to F$ is \emph{constant with respect to $\calP$} if $a(\omega)=a(\omega')$ whenever $\omega,\omega'\in\Omega_{i}$ for some $i\in I$. The permutation group $F$ is \emph{primitive} if it is transitive and preserves no non-trivial partition of $\Omega$. Equivalently, $F$ is transitive and its point stabilizers are maximal subgroups. Given a normal subgroup $N$ of $F$, the partition of $\Omega$ into $N$-orbits is $F$-invariant. Consequently, every non-trivial normal subgroup of a primitive group is transitive. The permutation group $F$ is \emph{quasiprimitive} if it is transitive and all its non-trivial normal subgroups are transitive. Finally, $F$ is \emph{semiprimitive} if it is transitive and all its normal subgroups are either transitive or semiregular. The following implications among the above properties follow from the definitions. We list examples illustrating that each implication is strict. 
\begin{displaymath}
 \text{$2$-transitive}\ \Rightarrow\ \underset{\text{{\normalsize $A_{5}\curvearrowright A_{5}/D_{5}$}}}{\text{primitive}}\ \Rightarrow \underset{\text{{\normalsize $A_{5}\curvearrowright A_{5}/C_{5}$}}}{\text{quasiprimitive}} \Rightarrow\ \underset{\text{{\normalsize $C_{4}\unrhd C_{2}$}}}{\text{semiprimitive}}\ \Rightarrow\underset{\text{{\normalsize $D_{4}\unrhd C_{2}\!\times\! C_{2}$}}}{\text{transitive\phantom{p}}}
\end{displaymath}
Note that $A_{5}$ is simple and that $C_{5}\lneq D_{5}\lneq A_{5}$ is a non-maximal subgroup of $A_{5}$.

\subsubsection*{Permutation Topology}
Let $X$ be a set and $H\le\Sym(X)$. The basic open sets of the \emph{permutation topology} on $H$ are $U_{x,y}:=\{h\in H\mid \forall i\in\{1,\ldots,n\}:\ h(x_{i})=y_{i}\}$, where $n\in\bbN$ and $x=(x_{1},\ldots,x_{n}), y=(y_{1},\ldots,y_{n})\in X^{n}$. This turns $H$ into a Hausdorff, t.d. group and makes the action map $H\times X\to X$ continuous for the discrete topology on $X$. The group $H$ is discrete if and only if the stabilizer in $H$ of a finite subset of $X$ is trivial. It is compact if and only if it is closed and all its orbits are finite. Finally, $\Sym(X)$ is second-countable if and only if $X$ is countable. 

\newpage
\subsection{Graph Theory}

We first recall Serre's \cite{Ser03} notation and definitions in the context of graphs and trees, and then collect generalities about automorphisms of trees. We conclude with an important simplicity criterion.

\subsubsection*{Definitions and Notation}
A \emph{graph} $\Gamma$ is a tuple $(V,E)$ consisting of a \emph{vertex set} $V$ and an \emph{edge set} $E$, together with a fixed-point-free involution of $E$, denoted by $e\mapsto\overline{e}$, and maps $o,t:E\to V$, providing the \emph{origin} and \emph{terminus} of an edge, such that $o(\overline{e})=t(e)$ and $t(\overline{e})=o(e)$ for all $e\in E$. Given $e\in E$, the pair $\{e,\overline{e}\}$ is a \emph{geometric edge}. For $x\in V$, we let $E(x):=o^{-1}(x)=\{e\in E\mid o(e)=x\}$ be the set of edges issuing from $x$. The \emph{valency} of $x\in V$ is $|E(x)|$. A vertex of valency $1$ is a \emph{leaf}. A \emph{morphism} between graphs $\Gamma_{1}=(V_{1},E_{1})$ and $\Gamma_{2}=(V_{2},E_{2})$ is a pair $(\alpha_{V},\alpha_{E})$ of maps $\alpha_{V}:V_{1}\to V_{2}$ and $\alpha_{E}:E_{1}\to E_{2}$ preserving the graph structure, i.e. $\alpha_{V}(o(e))=o(\alpha_{E}(e))$ and $\alpha_{V}(t(e))=t(\alpha_{E}(e))$ for all $e\in E$.

For $n\in\bbN$, let $\mathrm{Path}_{n}$ denote the graph with vertex set $\{0,\ldots,n\}$ and edge set $\smash{\{(k,k+1),\overline{(k,k+1)}\mid k\in\{0,\ldots,n-1\}\}}$. A \emph{path} of length $n$ in a graph $\Gamma$ is a morphism $\gamma$ from $\mathrm{Path}_{n}$ to $\Gamma$. It can be identified with $\smash{(e_{1},\ldots,e_{n})\in E(\Gamma)^{n}}$, where $e_{k}=\gamma((k-1,k))$ for $k\in\{1,\ldots,n\}$. In this case, $\gamma$ is a path \emph{from} $o(e_{1})$ \emph{to} $t(e_{n})$.

Similarly, let $\mathrm{Path}_{\bbN_{0}}$ and $\mathrm{Path}_{\bbZ}$ be the graphs with vertex sets $\bbN_{0}$ and $\bbZ$, and edge sets $\smash{\{(k,k+1),\overline{(k,k+1)}\mid k\in\bbN_{0}\}}$ and $\smash{\{(k,k+1),\overline{(k,k+1)}\mid k\in\bbZ\}}$ respectively. A \emph{half-infinite} path, or \emph{ray}, in a graph $\Gamma$ is a morphism $\gamma$ from $\mathrm{Path}_{\bbN_{0}}$ to $\Gamma$. It can be identified with $(e_{k})_{k\in\bbN}\in E(\Gamma)^{\bbN}$ where $e_{k}=\gamma((k-1,k))$ for $k\in\bbN$. In this case, $\gamma$ \emph{originates at}, or \emph{issues from}, $o(e_{1})$. An \emph{infinite path}, or \emph{line}, in a graph $\Gamma$ is a morphism from $\mathrm{Path}_{\bbZ}$ to $\Gamma$. A pair $\smash{(e_{k},e_{k+1})=(e_{k},\overline{e_{k}})}$ of edges in a path is a \emph{backtracking}. A graph is \emph{connected} if any two of its vertices can be joined by a path. The maximal connected subgraphs of a graph are its \emph{connected components}.

A \emph{forest} is a graph in which there are no non-backtracking paths $(e_{1},\ldots,e_{n})$ with $o(e_{1})=t(e_{n})$ $(n\in\bbN)$. Consequently, a morphism of forests is determined by the underlying vertex map. In particular, a path of length $n\in\bbN$ in a forest is determined by the images of the vertices of $\mathrm{Path}_{n}$.

A \emph{tree} is a connected forest. As a consequence of the above, the vertex set $V$ of a tree $T$ admits a natural metric: Given $x,y\in V$, define $d(x,y)$ as the minimal length of a path from $x$ to $y$. A tree in which every vertex has valency $d\in\bbN$ is \emph{$d$-regular}. It is unique up to isomorphism and denoted by $T_{d}$.

Let $T=(V,E)$ be a tree. For $S\subseteq V\cup E$, the \emph{subtree spanned by $S$} is the unique minimal subtree of $T$ containing $S$. For $x\in V$ and $n\in\bbN_{0}$, the subtree spanned by $\{y\in V\mid d(y,x)\le n\}$ is the \emph{ball} of radius $n$ around $x$, denoted by $B(x,n)$. Similarly, $S(x,n)=\{y\in V\mid d(y,x)\!=\! n\}$ is the \emph{sphere} of radius $n$ around $x$, and $E(x,n):=\{e\in E\mid d(o(e),x),d(t(e),x)\le n\}$. For a subtree $T'\subseteq T$, let $\pi:V\to V(T')$ denote the closest point projection, i.e. $\pi(x)=y$ whenever $d(x,y)=\min_{z\in V(T')}\{d(x,z)\}$. In the case of an edge $e=(x,y)\in E$, the \emph{half-trees} $T_{x}$ and $T_{y}$ are the subtrees spanned by $\pi^{-1}(x)$ and $\pi^{-1}(y)$ respectively.

Two non-backtracking rays $\gamma_{1},\gamma_{2}:\mathrm{Path}_{\bbN}\to T$ in $T$ are \emph{equivalent}, $\gamma_{1}\sim\gamma_{2}$, if there exist $N,d\in\bbN$ such that $\gamma_{1}(n)=\gamma_{2}(n+d)$ for all $n\ge N$. The \emph{boundary}, or \emph{set of ends}, of $T$ is the set $\partial T$ of equivalence classes of non-backtracking rays in $T$.

\vspace{-0.09cm}
\subsubsection*{Automorphism Groups of Graphs}\label{sec:automorphism_groups}
Let $\Gamma=(V,E)$ be a graph. We equipt the group $\Aut(\Gamma)$ of automorphims of $\Gamma$ with the permutation topology for its action on $V\cup E$.

\vspace{0.2cm}
\paragraph*{Notation}
Let $H\le\Aut(\Gamma)$. Given a subgraph $\Gamma'\subseteq\Gamma$, the \emph{pointwise stabilizer} of $\Gamma'$ in $H$ is denoted by $H_{\Gamma'}$\index{$H_{\Gamma'}$}. Similary, the \emph{setwise stabilizer} of $\Gamma'$ in $H$ is denoted by $H_{\{\Gamma'\}}$\index{$H_{\{\Gamma'\}}$}. In the case where $\Gamma'$ is a single vertex $x$, the permutation group that $H_{x}$ induces on $E(x)$ is denoted by $\smash{H_{x}^{(1)}\le\Sym(E(x))}$\index{$H_{x}^{1}$}. Given a property ``X'' of permutation groups, the group $H$ is \emph{locally} ``X''\index{locally ``X''} if for every $x\in V$ the permutation group $\smash{H_{x}^{(1)}}$ has ``X''; with the exception that $H$ is \emph{locally $k$-transitive} $(k\in\bbN_{\ge 3})$ if $H_{x}$ acts transitively on the set of non-backtracking paths of length $k$ issuing from~$x$. It is \emph{locally $\infty$-transitive} if it is locally $k$-transitive for all $k\in\bbN$.

Let $d\in\bbN_{\ge 3}$ and $T_{d}=(V,E)$ the $d$-regular tree. Then $\Aut(T_{d})$ acts on $\partial T_{d}$ by $g\cdot[\gamma]:=[g\circ\gamma]$. Given $[\gamma]\in\partial T_{d}$, the \emph{stabilizer} of $[\gamma]$ in $H$ is $H_{[\gamma]}=\{h\!\in\! H\mid h\circ\gamma\sim\gamma\}$\index{$H_{\omega}$}.

We let $\tensor[^{+}]{H}{}\!=\!\langle \{H_{x}\!\mid\! x\in V\}\rangle$\index{$\tensor[^{+}]{H}{}$} denote the subgroup of $H$ generated by vertex-stabilizers and $H^{+}\!=\!\langle\{H_{e}\!\mid\! e\in E\}\rangle$\index{$H^{+}$} the subgroup generated by edge-stabilizers. For a subtree $T\subseteq T_{d}$ and $k\in\bbN$, let $T^{k}$\index{$T^{k}$} denote the subtree of $T_{d}$ spanned by $\{x\in V\mid d(x,T)\le k\}$. We set $H^{+_{k}}=\langle\{H_{e^{k-1}}\!\mid\! e\in E\}\rangle$\index{$H^{+_{k}}$}. Then $H^{+_{1}}=H^{+}$ and 
\begin{displaymath}
\smash{H^{+_{k}}\unlhd H^{+}\unlhd\tensor[^{+}]{H}{}\unlhd H}.
\end{displaymath}

\paragraph*{Classification of Automorphisms} Automorphisms of $T_{d}$ can be distinguished into three distinct types. Refer to \cite[Section 6.2.2]{GGT18} for details.

For $g\!\in\!\Aut(T_{d})$, set $l(g)\!:=\!\min_{x\in V}d(x,gx)$ and $V(g)\!:=\!\{x\in V\!\mid\! d(x,gx)=l(g)\}$. If $l(g)=0$ then $g$ fixes a vertex. An automorphism of this kind is \emph{elliptic}. Suppose now that $l(g)>0$. If $V(g)$ is infinite then $g$ is \emph{hyperbolic}. Geometrically, it is a translation of \emph{length} $l(g)$ along the line in $T_{d}$ defined by $V(g)$. If $V(g)$ is finite then $l(g)=1$ and $g$ maps some edge $e\in E$ to $\overline{e}$, and is termed an \emph{inversion}.

\vspace{0.2cm}
\paragraph*{Independence and Simplicity}\label{sec:independence_simplicity}
The base case of the simplicity criterion presented below is due to Tits \cite{Tit70} and applies to sufficiently rich subgroups of $\Aut(T_{d})$. The generalized version is due to Banks--Elder--Willis \cite{BEW15}, see also \cite{GGT18}.

Let $C$ denote a path in $T_{d}$ (finite, half-infinite or infinite). For every $x\in V(C)$ and $k\in\bbN_{0}$, the pointwise stabilizer $H_{C^{k}}$ of $C^{k}$\index{$C^{k}$|seeonly{$T^{k}$}} induces an action $\smash{H_{C^{k}}^{(x)}\le\Aut(\pi^{-1}(x))}$ on $\pi^{-1}(x)$, the subtree spanned by those vertices of $T$ whose closest vertex in $C$ is~$x$. We therefore obtain an injective homomorphism
\begin{displaymath}
 \varphi_{C}^{(k)}:H_{C^{k}}\to\prod\nolimits_{x\in V(C)}H_{C^{k}}^{(x)}.
\end{displaymath}

\vspace{-0.1cm}
\noindent
A subgroup $H\le\Aut(T_{d})$ satisfies \emph{Property $(P_{k})$}\index{Property $(P_{k})$}\index{$(P_{k})$} $(k\in\bbN)$ if $\varphi_{C}^{(k-1)}$ is an isomorphism for every path $C$ in $T_{d}$. If $H\le\Aut(T_{d})$ is closed, it suffices to check the above properties in the case where $C$ is a single edge. For example, given a closed subgroup $H\le\Aut(T_{d})$, Property $(P_{k})$ is satisfied by its \emph{$(P_{k})$-closure}\index{$(P_{k})$-closure}
\begin{displaymath}
 H^{(P_{k})}=\{g\in\Aut(T_{d})\mid \forall x\in V(T_{d})\ \exists h\in H:\ g|_{B(x,k)}=h|_{B(x,k)}\}.
\end{displaymath}

\begin{theorem}[{\cite[Theorem 7.3]{BEW15}}]\label{thm:bew_simplicity}
Let $H\le\Aut(T_{d})$. Suppose $H$ neither fixes an end nor stabilizes a proper subtree of $T_{d}$ setwise, and that $H$ satisfies Property $(P_{k})$. Then the group $H^{+_{k}}$ is either trivial or simple.
\end{theorem}

\subsection*{Burger--Mozes Theory}\label{sec:bm_theory}

In \cite{BM00a}, Burger--Mozes develop a structure theory of certain locally quasiprimitive automorphism groups of graphs which resembles the theory of semisimple Lie groups. Their fundamental definitions are meaningful in the setting of t.d.l.c. groups. Let $H$ be a t.d.l.c. group. Define
\begin{displaymath}
 H^{(\infty)}:=\bigcap\{N\unlhd H\mid N \text{ is closed and cocompact in $H$}\},
\end{displaymath}
alternatively the intersection of all open finite-index subgroups of $H$, and
\begin{displaymath}
 \mathrm{QZ}(H):=\{h\in H\mid Z_{H}(h)\le H \text{ is open}\},
\end{displaymath}
the \emph{quasi-center} of $H$. Both $H^{(\infty)}$ and $\mathrm{QZ}(H)$ are topologically characteristic subgroups of $H$, i.e. they are preserved by continuous automorphisms of $H$. Whereas $H^{(\infty)}\le H$ is closed, the quasi-center need not be so.


Whereas for a general t.d.l.c. group $H$ nothing much can be said about the size of $H^{(\infty)}$ and $\mathrm{QZ}(H)$, Burger--Mozes show that good control can be obtained in the case of certain locally quasiprimitive automorphism groups of graphs. The following result summarizes their structure theory. It is a combination of Proposition 1.2.1, Corollary 1.5.1, Theorem 1.7.1 and Corollary 1.7.2 in \cite{BM00a}.

\begin{theorem}\label{thm:burger_mozes_structure_quasiprimitive}
Let $\Gamma$ be a locally finite, connected graph. Further, let $H\le\Aut(\Gamma)$ be closed, non-discrete and locally quasiprimitive. Then
\begin{enumerate}[(i)]
 \item $H^{(\infty)}$ is minimal closed normal cocompact in $H$,
 \item $\mathrm{QZ}(H)$ is maximal discrete normal, and non-cocompact in $H$, and
 \item $H^{(\infty)}\!/\mathrm{QZ}(H^{(\infty)})\!=\! H^{(\infty)}\!/(\mathrm{QZ}(H)\cap H^{(\infty)})$ admits minimal, non-trivial closed normal subgroups; finite in number, $H$-conjugate and topologically simple.
\end{enumerate}
If $\Gamma$ is a tree, and, in addition, $H$ is locally primitive then
\begin{enumerate}[(iv)]
 \item[(iv)] $H^{(\infty)}\!/\mathrm{QZ}(H^{(\infty)})$ is a direct product of topologically simple groups.
\end{enumerate}
\end{theorem}

\subsubsection*{Burger--Mozes Universal Groups}\label{sec:burger_mozes_groups}
The first introduction of Burger--Mozes universal groups in \cite[Section 3.2]{BM00a} was expanded in the introductory article \cite{GGT18}, which we follow closely. Most results are generalized in Section~\ref{sec:universal_groups}.

Let $\Omega$ be a set of cardinality $d\in\bbN_{\ge 3}$ and let $T_{d}=(V,E)$ denote the $d$-regular tree. A \emph{labelling} $l$ of $T_{d}$ is a map $l:E\to\Omega$ such that for every $x\in V$ the map $l_{x}:E(x)\to\Omega,\ e\mapsto l(e)$ is a bijection, and $l(e)\!=\!l(\overline{e})$ for all $e\in E$. The \emph{local action} $\sigma(g,x)\in\Sym(\Omega)$ of an automorphism $g\in\Aut(T_{d})$ at a vertex $x\in V$ is defined via
\begin{displaymath}
 \sigma:\Aut(T_{d})\times X\to\Sym(\Omega),\ (g,x)\mapsto\sigma(g,x):=l_{gx}\circ g\circ l_{x}^{-1}.
\end{displaymath}

\begin{definition}\label{def:u}
Let $F\le\Sym(\Omega)$ and $l$ a labelling of $T_{d}$. Define
\begin{displaymath}
 \mathrm{U}^{(l)}(F)\!:=\!\{g\in\Aut(T_{d})\mid\forall x\in V:\ \sigma(g,x)\in F\}.
\end{displaymath}
\end{definition}

The map $\sigma$ satisfies a \emph{cocycle identity}: For all $g,h\in\Aut(T_{d})$ and $x\in V$ we have $\sigma(gh,x)=\sigma(g,hx)\sigma(h,x)$.
As a consequence, $\mathrm{U}^{(l)}(F)$ is a subgroup of $\Aut(T_{d})$.

Passing to a different labelling amounts to passing to a conjugate of $\mathrm{U}^{(l)}(F)$ inside $\Aut(T_{d})$. We therefore omit the reference to an explicit labelling from here~on.

The following proposition collects several basic properties of Burger--Mozes groups. We refer the reader to \cite[Section 6.4]{GGT18} for proofs.

\begin{proposition}\label{prop:uf_basic_properties}
Let $F\le\Sym(\Omega)$. The group $\mathrm{U}(F)$ is
\begin{enumerate}[(i)]
 \item closed in $\mathrm{Aut}(T_{d})$,
 \item vertex-transitive,
 \item\label{item:uf_comp_gen} compactly generated,
 \item\label{item:uf_local_f} locally permutation isomorphic to $F$,
 \item edge-transitive if and only if $F$ is transitive, and
 \item\label{item:uf_discrete} discrete if and only if $F$ is semiregular.
\end{enumerate}
\end{proposition}

Part \ref{item:uf_comp_gen} of Proposition \ref{prop:uf_basic_properties} relies on the following result which we include for future reference. Given $x\in V$ and $\omega\in\Omega$, let $\smash{\iota_{\omega}^{(x)}\in\mathrm{U}(\{\id\})}$ denote the unique label-respecting inversion of the edge $e_{\omega}\in E$ with $o(e_{\omega})=x$ and $l(e_{\omega})=\omega$.

\begin{lemma}\label{lem:uid_fin_gen}
Let $x\in V$. Then $\mathrm{U}(\{\id\})=\langle\{\iota_{\omega}^{(x)}\mid\omega\in\Omega\}\rangle\cong\underset{\omega\in\Omega}{\bigast}\langle\iota_{\omega}^{(x)}\rangle\cong\underset{\omega\in\Omega}{\bigast}\bbZ\!/2\bbZ$.
\end{lemma}

\begin{proof}
Every element of $\mathrm{U}(\{\id\})$ is determined by its image on~$x$. Hence it suffices to show that $\smash{\langle\{\iota_{\omega}^{(x)}\mid\omega\in\Omega\}\rangle}$ is vertex-transitive and has the asserted structure. Indeed, let $y\in V\backslash\{x\}$, and let $\omega_{1},\ldots,\omega_{n}\in\Omega$ be the labels of the shortest path from $x$ to $y$. Then $\smash{\iota_{\omega_{1}}^{(x)}\circ\cdots\circ\iota_{\omega_{n}}^{(x)}}$ maps $x$ to $y$ as every $\smash{\iota_{\omega}^{(x)}}$ $(\omega\in\Omega)$ is label-respecting. Setting $X_{\omega}:=T_{t(e_{\omega})}$ we have $\iota_{\omega}(X_{\omega'})\subseteq X_{\omega}$ for all distinct $\omega,\omega'\in\Omega$. Hence the assertion follows from the ping-pong lemma.
\end{proof}

The name \emph{universal group} is due to the following maximality statement. Its proof, see \cite[Proposition 3.2.2]{BM00a}, should be compared with the proof of Theorem \ref{thm:ukf_universal}.

\begin{proposition}\label{prop:uf_universal}
Let $H\le\Aut(T_{d})$ be locally transitive and vertex-transitive. Then there is a labelling $l$ of $T_{d}$ such that $H\le\mathrm{U}^{(l)}(F)$ where $F\le\Sym(\Omega)$ is action isomorphic to the local action of $H$.
\end{proposition}

\newpage
\section{Structure Theory of locally semiprimitive groups}\label{sec:bm_theory_semiprimitive}

We generalize the Burger--Mozes theory of locally quasiprimitive automorphism groups of graphs to the semiprimitive case. While this adjustment of Sections 1.1 to 1.5 in \cite{BM00a} is straightforward and has been initiated in \cite[Section II.7]{Tor18} and \cite[Section 6.2]{CB19} we provide a full account for the reader's convenience.

\subsection{General Facts}

Let $\Gamma=(V,E)$ be a connected graph. We first collect a few general facts about several classes of subgroups of $\Aut(\Gamma)$ for future reference.

\begin{lemma}\label{lem:bm_1.3.0}
Let $H\le\Aut(\Gamma)$ be locally transitive. Then $\tensor[^{+}]{H}{}$ is geometric edge transitive and of index at most $2$ in $H$.
\end{lemma}

\begin{proof}
Since $H$ is locally transitive, so is $\tensor[^{+}]{H}{}$ given that $\tensor[^{+}]{H}{}_{x}=H_{x}$ for all $x\in V$. Hence it is geometric edge transitive. In particular it has at most two vertex orbits which implies the second assertion.
\end{proof}

\begin{lemma}\label{lem:bm_1.3.1}
Let $H\!\le\!\Aut(\Gamma)$ and let $\Gamma'=(V',E')$ be a connected subgraph of $\Gamma$. Suppose $R\subseteq H$ is such that for every $x'\in V'$ and $e\in E(x')$ there is $r\in R$ such that $re\in E'$. Then $\Lambda:=\langle R\rangle$ satisfies $\bigcup_{\lambda\in\Lambda}\lambda\Gamma'=\Gamma$.
\end{lemma}

\begin{proof}
By assumption, $B(\Gamma',1)\subseteq\bigcup_{\lambda\in\Lambda}\lambda\Gamma'$. Now suppose $B(\Gamma',n)\subseteq\bigcup_{\lambda\in\Lambda}\lambda\Gamma'$ for some $n\in\bbN$. Let $x'\in V(B(\Gamma',n))$. Pick $\lambda\in\Lambda$ such that $\lambda(x')\!\in\! V'$. Since $\lambda$ induces a bijection between $E(x')$ and $E(\lambda(x'))$ we conclude that $B(\Gamma',n+1)\subseteq\bigcup_{\lambda\in\Lambda}\lambda\Gamma'$.
\end{proof}

Assume from now on that $\Gamma$ is a locally finite, connected graph.

\begin{lemma}\label{lem:bm_1.3.2}
Let $H\le\Aut(\Gamma)$. If $H\backslash\Gamma$ is finite then there is a finitely generated subgroup $\Lambda\le H$ such that $\Lambda\backslash\Gamma$ is finite.
\end{lemma}

\begin{proof}
Let $\Gamma'=(V',E')\subseteq\Gamma$ be a connected subgraph which projects onto $H\backslash\Gamma$. For every $x'\in V'$ and $e\in E(x')$, pick $\lambda_{x',e}\in H$ such that $\lambda_{x',e}(e)\in E'$. Then $\Lambda:=\langle\{\lambda_{x',e}\mid x'\in X,\ e\in E(x')\}\rangle$ satisfies the conclusion by Lemma \ref{lem:bm_1.3.1}.
\end{proof}

\begin{lemma}\label{lem:bm_1.3.3}
Let $\Lambda\le\Aut(\Gamma)$. If $\Lambda\backslash\Gamma$ is finite then $Z_{\Aut(\Gamma)}(\Lambda)$ is discrete.
\end{lemma}

\begin{proof}
Let $F\subseteq E$ be finite such that $\bigcup_{\lambda\in\Lambda}\lambda F=E$ and $U:=\Lambda_{F}\cap Z_{\Aut(\Gamma)}(\Lambda)$, which is open in $Z_{\Aut(\Gamma)}(\Lambda)$. Given that $U$ and $\Lambda$ commute, $U$ acts trivially on $E=\bigcup_{\lambda\in\Lambda}\lambda F$. Hence $U=\{\id\}$ and $Z_{\Aut(\Gamma)}(\Lambda)$ is discrete.
\end{proof}

\begin{lemma}\label{lem:bm_1.3.4}
Let $\Lambda_{1},\Lambda_{2}\le\Aut(\Gamma)$. If $\Lambda_{1}\backslash\Gamma$ is finite and $[\Lambda_{1},\Lambda_{2}]\le\Aut(\Gamma)$ is discrete then $\Lambda_{2}\le\Aut(\Gamma)$ is discrete.
\end{lemma}

\begin{proof}
Using Lemma \ref{lem:bm_1.3.2} pick $R\subseteq\Lambda_{1}$ such that $\langle R\rangle\backslash\Gamma$ is finite. As $[\Lambda_{1},\Lambda_{2}]\!\le\!\Aut(\Gamma)$ is discrete, there is an open subgroup $U\le\Lambda_{2}$ such that $[r,U]=\{e\}$ for all $r\in R$. That is, $U\le Z_{\Aut(\Gamma)}(\langle R\rangle)$. Hence $U$ is discrete by Lemma \ref{lem:bm_1.3.3}, and so is $\Lambda_{2}$.
\end{proof}

\begin{lemma}\label{lem:bm_1.3.5}
Let $H\le\Aut(\Gamma)$ be non-discrete. Then $\mathrm{QZ}(H)\backslash\Gamma$ is infinite.
\end{lemma}

\begin{proof}
If $\mathrm{QZ}(H)\backslash\Gamma$ is finite, there is a finitely generated subgroup $\Lambda\le\mathrm{QZ}(H)$ such that $\Lambda\backslash\Gamma$ is finite as well by Lemma \ref{lem:bm_1.3.2}. Hence there is an open subgroup $U\le H$ with $U\le Z_{\Aut(\Gamma)}(\Lambda)$. Hence $U$ and thereby $H$ is discrete by Lemma~\ref{lem:bm_1.3.3}.
\end{proof}

\begin{lemma}\label{lem:bm_1.3.6}
Let $\Lambda\!\le\!\Aut(\Gamma)$ be discrete. If $\Lambda\backslash\Gamma$ is finite then $N_{\Aut(\Gamma)}(\Lambda)$ is discrete.
\end{lemma}

\begin{proof}
Apply Lemma \ref{lem:bm_1.3.4} to $\Lambda_{1}:=\Lambda$ and $\Lambda_{2}:=N_{\Aut(\Gamma)}(\Lambda)$.
\end{proof}

\subsection{Normal Subgroups}

Let $\Gamma=(V,E)$ denote a locally finite, connected graph. For closed subgroups $\Lambda\unlhd H$ of $\Aut(\Gamma)$ we define
\begin{displaymath}
 \calN_{\mathrm{nf}}(H,\Lambda)=\{N\unlhd H\mid \Lambda\le N\unlhd H,\ N \text{ is closed and does not act freely on } E\},
\end{displaymath}
the set of closed normal subgroups of $H$ which contain $\Lambda$ and do not act freely on $E$. The set $\calN_{\mathrm{nf}}(H,\Lambda)$ is partially ordered by inclusion. We let $\calM_{\mathrm{nf}}(H,\Lambda)\subseteq\calN_{\mathrm{nf}}(H,\Lambda)$ denote the set of minimal elements in $\calN_{\mathrm{nf}}(H,\Lambda)$.

\begin{lemma}\label{lem:bm_1.4.1}
Let $\Gamma\!=\!(V,E)$ be a locally finite, connected graph and $\Lambda\unlhd H\!\le\!\Aut(\Gamma)$. If $H\backslash\Gamma$ is finite and $H$ does not act freely on $E$ then $\calM_{\mathrm{nf}}(H,\Lambda)\neq\emptyset$.
\end{lemma}

\begin{proof}
We argue using Zorn's Lemma. First note that $\calN_{\mathrm{nf}}(H,\Lambda)$ is non-empty as it contains $H$. Let $C\subseteq\calN_{\mathrm{nf}}(H,\Lambda)$ be a chain. Pick a finite set $F\subseteq E$ of representatives of $H\backslash E$. For every $N\in C$, the set $F_{N}:=\{e\in F\mid N|_{e^{1}}\le\Aut(e^{1}) \text{ is non-trivial}\}$ is non-empty. Since $F$ is finite and $C$ is a chain it follows that $\bigcap_{N\in C}F_{N}$ is non-empty, i.e. there exists $e\in F$ such that $N|_{e^{1}}$ is non-trivial for every $N\in C$. As before, we conclude that $M:=\bigcap_{N\in C}N|_{e^{1}}$ is non-trivial. Now, for $\alpha\in M\backslash\{\id\}$ and $N\in C$, the set $N^{\alpha}:=\{g\in N_{e}\mid g|_{e^{1}}=\alpha\}$ is a non-empty compact subset of $H_{e}$, and since $C$ is a chain every finite subset of $\{N^{\alpha}\mid N\in C\}$ has non-empty intersection. Hence $\bigcap_{N\in C}N^{\alpha}$ is non-empty and therefore $N_{C}:=\bigcap_{N\in C}N$ is a closed normal subgroup of $H$ containing $\Lambda$ that does not act freely on $E$. Overall, $N_{C}\in\calM_{\mathrm{nf}}(H,\Lambda)$.
\end{proof}

The following lemma is contained in the author's PhD thesis \cite[Section II.7]{Tor18} and, independently, in Caprace-Le Boudec \cite[Section 6.2]{CB19}.

\begin{lemma}\label{lem:bm_1.4.2}
Let $\Gamma=(V,E)$ be a locally finite, connected graph. Further, let $H\le\Aut(\Gamma)$ be locally semiprimitive and $N\unlhd H$. Define
\begin{enumerate}
 \item[] $V_{1}:=\{x\in V\mid N_{x}\curvearrowright S(x,1) \text{ is transitive and not semiregular}\}$,
 \item[] $V_{2}:=\{x\in V\mid N_{x}\curvearrowright S(x,1) \text{ is semiregular}\}$.
\end{enumerate} 
Then one of the following holds.
\begin{enumerate}[(i)]
 \item\label{item:semiprimitive_normal_free} $V=V_{2}$ and $N$ acts freely on $E$.
 \item\label{item:semiprimitive_normal_tran} $V=V_{1}$ and $N$ is geometric edge transitive.
 \item $V=V_{1}\sqcup V_{2}$ is an $H$-invariant partition of $V$ and $B(x,1)$ is a fundamental domain for the action of $N$ on $\Gamma$ for any $x\in V_{2}$.
\end{enumerate}
\end{lemma}

\begin{proof}
Since $H$ is locally semiprimitive and $N$ is normal in $H$, we have $V=V_{1}\sqcup V_{2}$. If $N$ does not act freely on $E$ then there is an edge $e\in E$ with $N_{e}\neq\{\id\}$ and an $N_{e}$-fixed vertex $x\in V$ for which $N_{x}\curvearrowright S(x,1)$ is not semiregular, hence transitive. That is, $V_{1}\neq\emptyset$. Now, either $V_{2}(N)=\emptyset$ in which case $N$ is locally transitive and we are in case \ref{item:semiprimitive_normal_tran}, or $V_{2}(N)\neq\emptyset$. Being locally transitive, $H$ acts transitively on the set of geometric edges and therefore has at most two vertex orbits. Given that both $V_{1}$ and $V_{2}$ are non-empty and $H$-invariant, they constitute exactly said orbits. Since any pair of adjacent vertices $(x,y)$ is a fundamental domain for the $H$-action on $V$, we conclude that if $y\in V_{2}$ then $x\in V_{1}$. Thus every leaf of $B(y,1)$ is in $V_{1}$ and we are in case (iii) by Lemma \ref{lem:bm_1.3.1}.
\end{proof}

\subsection{The Subquotient $H^{(\infty)}/\mathrm{QZ}(H^{(\infty)}$}

In this section, we achieve control over $H^{(\infty)}$ and $\mathrm{QZ}(H)$ as well as the normal subgroups of $H$ in the semiprimitive case. We then describe the structure of the subquotient $H^{(\infty)}/\mathrm{QZ}(H^{(\infty)})$. First, recall the following lemma from topological group theory.

\begin{lemma}\label{lem:discrete_normal_sub}
Let $G$ be a topological group. If $H\unlhd G$ is discrete then $H\subseteq\mathrm{QZ}(G)$.
\end{lemma}

\begin{proof}
For $h\in H$, the map $\varphi_{h}:G\to H$, $g\mapsto ghg^{-1}$ is well-defined because $H\unlhd G$, and continuous. Hence there is an open set $U\subseteq G$ containing $1\in G$ and such that $\varphi_{h}(U)\subseteq\{h\}$, i.e. $U\subseteq Z_{G}(h)$.
\end{proof}

\begin{proposition}\label{prop:bm_1.2.1}
Let $\Gamma=(V,E)$ be a locally finite, connected graph. Further, let $H\le\Aut(\Gamma)$ be closed, non-discrete and locally semiprimitive. Then
\begin{enumerate}[(i)]
 \item\label{item:bm_1.2.1_1} $H/H^{(\infty)}$ is compact,
 \item\label{item:bm_1.2.1_2} $\mathrm{QZ}(H)$ acts freely on $E$, and is discrete non-cocompact in $H$,
 \item\label{item:bm_1.2.1_3} for any closed normal subgroup $N\unlhd H$, either $N$ is non-discrete cocompact and $N\unrhd H^{(\infty)}$, or $N$ is discrete and $N\unlhd\mathrm{QZ}(H)$,
 \item\label{item:bm_1.2.1_4} $\mathrm{QZ}(H^{(\infty)})=\mathrm{QZ}(H)\cap H^{(\infty)}$ acts freely on $E$ without inversions,
 \item\label{item:bm_1.2.1_5} for any open normal subgroup $N\unlhd H^{(\infty)}$ we have $N=H^{(\infty)}$, and
 \item\label{item:bm_1.2.1_6} $H^{(\infty)}$ is topologically perfect, i.e. $H^{(\infty)}=[H^{(\infty)},H^{(\infty)}]$.
\end{enumerate}
\end{proposition}

\begin{proof}
For \ref{item:bm_1.2.1_1}, let $N\unlhd H$ be closed and cocompact. Since $H$ is non-discrete, so is $N$ in view of Lemma \ref{lem:bm_1.3.6}. Hence $N\in\calN_{\mathrm{nf}}(H,\{\id\})$. Conversely, if $N\in\calN_{\mathrm{nf}}(H,\{\id\})$ then $N$ is cocompact in $H$ by Lemma \ref{lem:bm_1.4.2}. We conclude that $H^{(\infty)}=\bigcap\calN_{\mathrm{nf}}(H,\{\id\})$. This intersection is in fact given by a single minimal element of $\calN_{\mathrm{nf}}(H,\{\id\})$: Using Lemma \ref{lem:bm_1.4.1}, pick $M\in\calM_{\mathrm{nf}}(H,\{\id\})$, and let $N\in\calN_{\mathrm{nf}}(H,\{\id\})$. Suppose $N\not\supseteq M$. Because $M$ is minimal, $N\cap M$ acts freely on $E$. In particular, $N\cap M$ is discrete. Since both $N$ and $M$ are normal in $H$, we also have $N\cap M\supseteq[N,M]$ and hence $N$ and $M$ are discrete by Lemma \ref{lem:bm_1.3.4}. Then so is $H\subseteq N_{\Aut(\frakg)}(H)$ by Lemma \ref{lem:bm_1.3.6}. Overall, $\smash{H^{(\infty)}\!=\!M\!\in\!\calM_{\mathrm{nf}}(H,\{\id\})}$ and assertion now follows from Lemma \ref{lem:bm_1.4.2}.
 
As to \ref{item:bm_1.2.1_2}, the group $\mathrm{QZ}(H)$ is non-cocompact by Lemma \ref{lem:bm_1.3.5} and therefore acts freely on $E$ by Lemma \ref{lem:bm_1.4.2}. In particular, it is discrete.

For \ref{item:bm_1.2.1_3}, let $N\unlhd H$ be a closed normal subgroup. If $N$ acts freely on $E$, then $N$ is discrete and hence contained in $\mathrm{QZ}(H)$ by Lemma \ref{lem:discrete_normal_sub}. If $N$ does not act freely on $E$ then $N$ is cocompact in $H$ by Lemma \ref{lem:bm_1.4.2} and therefore contains $H^{(\infty)}$.

Concerning \ref{item:bm_1.2.1_4} the inclusion $\mathrm{QZ}(H)\cap H^{(\infty)}\subseteq\mathrm{QZ}(H^{(\infty)})$ is automatic. Further, $\mathrm{QZ}(H^{(\infty)})$ is normal in $H$ because it is topologically characteristic in $H^{(\infty)}\unlhd H$. Therefore, if $\mathrm{QZ}(H^{(\infty)})\not\subseteq\mathrm{QZ}(H)$, then $\mathrm{QZ}(H^{(\infty)})$ is non-discrete by part \ref{item:bm_1.2.1_3} and does not act freely on $E$. Then $\mathrm{QZ}(H^{(\infty)})\backslash\Gamma$ is finite by Lemma \ref{lem:bm_1.4.2}, contradicting Lemma \ref{lem:bm_1.3.5} applied to $H^{(\infty)}$ which is non-discrete because $\mathrm{QZ}(H^{(\infty)})\le H^{(\infty)}$ is. Consequently, $\mathrm{QZ}(H^{(\infty)})\le\mathrm{QZ}(H)$ which proves the assertion.

For part \ref{item:bm_1.2.1_5}, note that $\calM_{\mathrm{nf}}(H^{(\infty)},\{\id\})$ is non-empty by Lemma \ref{lem:bm_1.4.1} as $H^{(\infty)}$ is cocompact in $\Aut(\Gamma)$ by part \ref{item:bm_1.2.1_1} and non-discrete by part \ref{item:bm_1.2.1_3}. Further, since $\mathrm{QZ}(H^{(\infty)})$ acts freely on $E$, every $N\in\calN_{\mathrm{nf}}(H^{(\infty)},\{\id\})$ is non-discrete by part \ref{item:bm_1.2.1_3} as well. Given an open subgroup $U\unlhd H^{(\infty)}$ and $N\in\calM_{\mathrm{nf}}(H^{\infty},\{\id\})$, the group $U\cap N$ is normal in $H^{(\infty)}$ and non-discrete. In particular, $U\cap N$ does not act freely on $E$ and hence $U\cap N=N$. Thus $U$ contains the subgroup of $H^{(\infty)}$ generated by the elements of $\calM_{\mathrm{nf}}(H^{(\infty)},\{\id\})$, which is closed, normal and non-discrete. Hence $U= H^{(\infty)}$.

As to \ref{item:bm_1.2.1_6}, the group $[H^{(\infty)},H^{(\infty)}]$ is non-discrete by part \ref{item:bm_1.2.1_1} and Lemma \ref{lem:bm_1.3.4}. Hence so is $\smash{\overline{[H^{(\infty)},H^{(\infty)}]}\unlhd H^{(\infty)}}$. Now apply part \ref{item:bm_1.2.1_3}.
\end{proof}

\begin{proposition}\label{prop:bm_1.5.1}
Let $\Gamma=(V,E)$ be a locally finite, connected graph. Further, let $H\le\Aut(\Gamma)$ be a closed, non-discrete and locally semiprimitive. Finally, let $\Lambda\unlhd H$ such that $\Lambda\le\mathrm{QZ}(H^{(\infty)})$. Then the following hold.
\begin{enumerate}[(i)]
 \item\label{item:bm_1.5.1_1} \begin{enumerate}[(a)]
	  \item\label{item:bm_1.5.1_1_a} The group $H$ acts transitively on $\calM_{\mathrm{nf}}(H^{(\infty)},\Lambda)$.
	  \item\label{item:bm_1.5.1_1_b} The set $\calM_{\mathrm{nf}}(H^{(\infty)},\Lambda)$ is finite and non-empty.
       \end{enumerate}
 \item\label{item:bm_1.5.1_2} Let $M\in\calM_{\mathrm{nf}}(H^{(\infty)},\Lambda)$
 \begin{enumerate}[(a)]
  \item\label{item:bm_1.5.1_2_a} The group $M/\Lambda$ is topologically perfect.
  \item\label{item:bm_1.5.1_2_b} The group $\mathrm{QZ}(M)$ acts freely on $E$ and $\mathrm{QZ}(M)=\mathrm{QZ}(H^{(\infty)})\cap M$.
  \item\label{item:bm_1.5.1_2_c} The group $M/\mathrm{QZ}(M)$ is topologically simple.
 \end{enumerate}
 \item\label{item:bm_1.5.1_3} For every $N\in\calN_{\mathrm{nf}}(H^{(\infty)},\Lambda)$ there is $M\in\calM_{\mathrm{nf}}(H^{(\infty)},\Lambda)$ with $N\supseteq M$.
\end{enumerate}
\end{proposition}

\begin{proof}
Since every discrete normal subgroup of $H^{(\infty)}$ is contained in $\mathrm{QZ}(H^{(\infty)})$ by Lemma \ref{lem:discrete_normal_sub} \ref{item:bm_1.2.1_3} and the latter acts freely on $E$ by Proposition \ref{prop:bm_1.2.1} \ref{item:bm_1.2.1_3}, every element of $\mathcal{N}_{\mathrm{nf}}(H^{(\infty)},\Lambda)$ is non-discrete. We proceed with a number of claims.

\begin{enumerate}[(1),series=this_proof]
 \item\label{item:bm_1.5.1_proof_1} For every $N\in\mathcal{N}_{\mathrm{nf}}(H^{(\infty)},\Lambda)$ we have $[H^{(\infty)},N]\not\subseteq\mathrm{QZ}(H^{(\infty)})$. \newline
 This follows from the above combined with \ref{prop:bm_1.2.1} \ref{item:bm_1.2.1_1} and Lemma \ref{lem:bm_1.3.4}.
\end{enumerate}

In the following, given $S\subseteq\calM_{\mathrm{nf}}(H^{(\infty)},\Lambda)$, we let $M_{S}:=\langle M\mid M\in S\rangle\le H^{(\infty)}$ denote the subgroup of $H^{(\infty)}$ generated by $\bigcup_{M\in S}M$.

\begin{enumerate}[(1),resume=this_proof]
 \item\label{item:bm_1.5.1_proof_2} The group $H$ acts transitively on $\mathcal{M}_{\mathrm{nf}}(H^{(\infty)},\Lambda)$. \newline
 Let $ S$ be an orbit for the action of $H$ on $\calM_{\mathrm{nf}}(H^{(\infty)},\Lambda)$, and suppose there is an element $M\in\calM_{\mathrm{nf}}(H^{(\infty)},\Lambda)\backslash S$. For every $N\in S$, the subgroup $N\cap M$ is normal in $H^{(\infty)}$ and acts freely on $E$ by minimality of $M$, hence is discrete. The same therefore holds for $[N,M]\subseteq N\cap M$. Thus $[N,M]\subseteq\mathrm{QZ}(H^{(\infty)})$. As $\mathrm{QZ}(H^{(\infty)})$ is discrete by Proposition \ref{prop:bm_1.2.1} and therefore closed in $H^{(\infty)}$ we conclude $[\overline{M_{ S}},M]\subseteq\mathrm{QZ}(H^{(\infty)})$. On the other hand, $\overline{M_{S}}$ is normal in $H$ since $S$ is an $H$-orbit. It is also closed in $H$, and non-discrete by the above. Thus $\overline{M_{S}}=H^{(\infty)}$ by Proposition \ref{prop:bm_1.2.1} \ref{item:bm_1.2.1_3}, and $[H^{(\infty)},M]\subseteq\mathrm{QZ}(H^{(\infty)})$ which contradicts part \ref{item:bm_1.5.1_proof_1}.
 
 \item\label{item:bm_1.5.1_proof_3} For every $M\in\calM_{\mathrm{nf}}(H^{(\infty)},\Lambda)$ we have $\smash{\overline{[M,M]\cdot\Lambda}=M}$. \newline
 Note that $\overline{[M,M]\cdot\Lambda}$ is a group because $\Lambda$ is normal in $M$. Suppose there is an element $M_{0}\in\calM_{\mathrm{nf}}(H^{(\infty)},\Lambda)$ with $\overline{[M_{0},M_{0}]\cdot\Lambda}\lneq M_{0}$. Then $\overline{[M_{0},M_{0}]\cdot\Lambda}$ acts freely on $E$ by minimality of $M_{0}$ and is discrete. Being normal in $H$, we obtain $[M_{0},M_{0}]\subseteq\mathrm{QZ}(H^{(\infty})$. Part \ref{item:bm_1.5.1_proof_2} now implies that $[M,M]\subseteq\mathrm{QZ}(H^{(\infty)})$ for all $M\in\mathcal{M}_{\mathrm{nf}}(H^{(\infty)},\Lambda)$. Given that $[M,M']\subseteq\mathrm{QZ}(H^{(\infty)})$ for all distinct $M,M'$ in $\calM_{\mathrm{nf}}(H^{(\infty)},\Lambda)$ as well, we conclude that $[H^{(\infty)},H^{(\infty)}]\subseteq\mathrm{QZ}(H^{(\infty)})$ which contradicts part \ref{item:bm_1.5.1_proof_1}.
 
 \item\label{item:bm_1.5.1_proof_4} For every $N\in\calN_{\mathrm{nf}}(H^{(\infty)},\Lambda)$ there is $M\in\calM_{\mathrm{nf}}(H^{(\infty)},\Lambda)$ with $N\supseteq M$. \newline
 Let $ S\!:=\!\{M\!\in\!\calM_{\mathrm{nf}}(H^{(\infty)},\Lambda)\!\mid\! N\not\supseteq M\}$. Then $[\overline{M_{S}},N]\!\subseteq\!\mathrm{QZ}(H^{(\infty)})$ as above. On the other hand, for $T:=\calM_{\mathrm{nf}}(H^{(\infty)},\Lambda)$, the group $\overline{M_{T}}\subseteq H^{(\infty)}$ is closed, non-discrete and normal in $H$, thus $\overline{M_{T}}=H^{(\infty)}$. Using \ref{item:bm_1.5.1_proof_1}, we conclude that $ S\neq T$ which proves the assertion.
 
 \item\label{item:bm_1.5.1_proof_5} Let $S, S'$ be disjoint subsets of $\calM_{\mathrm{nf}}(H^{(\infty)},\Lambda)$. Then $\overline{M_{S}}\cap\overline{M_{S'}}\subseteq\mathrm{QZ}(H^{(\infty)})$. 
 If not, we have $\smash{\overline{M_{ S}}\cap\overline{M_{ S'}}\!\in\!\calM_{\mathrm{nf}}(H^{(\infty)},\Lambda)}$ and there is, by part \ref{item:bm_1.5.1_proof_4}, an element $M\in\calM_{\mathrm{nf}}(H^{(\infty)},\Lambda)$ with $\smash{M\subseteq\overline{M_{S}}\cap\overline{M_{ S'}}}$. However, this implies that $[M,M]\subseteq[\overline{M_{S}},\overline{M_{S'}}]\subseteq\mathrm{QZ}(H^{(\infty)})$ which contradicts part \ref{item:bm_1.5.1_proof_3}.
 
 \item\label{item:bm_1.5.1_proof_6} The set $\calM_{\mathrm{nf}}(H^{(\infty)},\Lambda)$ is finite and non-empty. \newline 
 The set $\calM_{\mathrm{nf}}(H^{(\infty)},\Lambda)$ is non-empty by Lemma~\ref{lem:bm_1.4.1}. Let $G=\bigcup\overline{M_{S}}$, where the union is taken over all finite subsets $S$ of the set $\calM_{\mathrm{nf}}(H^{(\infty)},\Lambda)$. Then $G$ is non-discrete and normal in $H$. Hence $\overline{G}=H^{(\infty)}$ by Proposition \ref{prop:bm_1.2.1} \ref{item:bm_1.2.1_3}. Since $H$ is second-countable and locally compact, it is metrizable. Hence $H^{(\infty)}$ is a separable metric space and the same holds for $G$. Let $L\subseteq G$ be a countable dense subgroup, and fix an exhaustion $F_{1}\subseteq F_{2}\subseteq\dots\subseteq F$ of $F$ by finite sets. Let $( S_{n})_{n\in\bbN}$ be an increasing sequence of finite subsets of $\calM_{\mathrm{nf}}(H^{(\infty)},\Lambda)$ such that $F_{n}\subseteq\overline{M_{S_{n}}}$. In particular
 \begin{displaymath}
  \displaybump L\subseteq\overline{M_{\bigcup_{n\in\bbN} S_{n}}} \quad\text{and thus}\quad \overline{M_{\bigcup_{n\in\bbN} S_{n}}}=H^{(\infty)}
 \end{displaymath}
 which by \ref{item:bm_1.5.1_proof_5} and \ref{item:bm_1.5.1_proof_1} implies $\calM_{\mathrm{nf}}(H^{(\infty)},\Lambda)=\bigcup_{n\in\bbN} S_{n}$. Thus $\calM_{\mathrm{nf}}(H^{(\infty)},\Lambda)$ is countable. Next, fix $M\in\calM_{\mathrm{nf}}(H^{(\infty)},\Lambda)$. Then $N_{H}(M)$ is closed and of countable index in $H$, and thus has non-empty interior as $H$ is a Baire space. Hence $N_{H}(M)$ is open in $H$. Given that $N_{H}(M)$ contains $H^{(\infty)}$ we conclude that $N_{H}(M)$ is of finite index in $H$ using Proposition \ref{prop:bm_1.2.1} \ref{item:bm_1.2.1_1}. Since $H$ acts transitively by on $\calM_{\mathrm{nf}}(H^{(\infty)},\Lambda)$ by \ref{item:bm_1.5.1_proof_2} we conclude that $\calM_{\mathrm{nf}}(H^{(\infty)},\Lambda)$ is finite by the orbit-stabilizer theorem.
\end{enumerate}

The above claims yield parts \ref{item:bm_1.5.1_1}\ref{item:bm_1.5.1_1_a}, \ref{item:bm_1.5.1_1}\ref{item:bm_1.5.1_1_b}, \ref{item:bm_1.5.1_2}\ref{item:bm_1.5.1_2_a} and \ref{item:bm_1.5.1_3} of Proposition \ref{prop:bm_1.5.1}. We now turn to parts \ref{item:bm_1.5.1_2}\ref{item:bm_1.5.1_2_b} and \ref{item:bm_1.5.1_2}\ref{item:bm_1.5.1_2_c}.

\begin{enumerate}
 \item[\ref{item:bm_1.5.1_2}\ref{item:bm_1.5.1_2_b}] Using part \ref{item:bm_1.5.1_proof_6}, let $\calM_{\mathrm{nf}}(H^{(\infty)},\Lambda)=\{M_{1},\ldots,M_{r}\}$ and define
 \begin{displaymath}
  \displaybump \Omega:=\mathrm{QZ}(M_{1})\cdot\ldots\cdot\mathrm{QZ}(M_{r}).
 \end{displaymath}
 Note that since $\mathrm{QZ}(M_{i})$ is characteristic in $M_{i}$, which is normal in $H^{(\infty)}$, the quasi-centers in the above definition normalize each other, so $\Omega$ is a group. It is then normal in $H$. If $\Omega$ does not act freely on $E$ then $\Omega\backslash\Gamma$ is finite by Lemma \ref{lem:bm_1.4.2} and there exist $\lambda_{1},\ldots,\lambda_{k}\in\Omega$ by Lemma \ref{lem:bm_1.3.2} such that for $\Omega':=\langle\lambda_{1},\ldots,\lambda_{k}\rangle$ the quotient $\Omega'\backslash\Gamma$ is finite. For every $i\in\{1,\ldots,k\}$, write $\lambda_{i}=a_{i}b_{i}$ where $a_{i}\in\mathrm{QZ}(M_{1})$ and $b_{i}\in\mathrm{QZ}(M_{2})\cdot\ldots\cdot\mathrm{QZ}(M_{r})$. Let $U_{1}\le M_{1}$ be an open subgroup such that $[a_{i},U_{1}]=\{e\}$ for all $i\in\{1,\ldots,k\}$. Since $[M_{2}\cdot\ldots M_{r},M_{1}]\subseteq\mathrm{QZ}(H^{(\infty)})$, there is an open subgroup $U_{2}\le M_{1}$ such that $[b_{i},U_{2}]=\{e\}$ for all $i\in\{1,\ldots,k\}$. Hence $U:=U_{1}\cap U_{2}\le M_{1}$ is contained in $Z_{\Aut(\Gamma)}(\Omega')$ which by Lemma \ref{lem:bm_1.3.3} implies that $U$ and hence $M_{1}$ is discrete, a contradiction. Thus $\Omega$ acts freely on $E$, is discrete and therefore $\Omega\subseteq\mathrm{QZ}(H^{(\infty)})$. That is, $\mathrm{QZ}(M_{i})\subseteq\mathrm{QZ}(H^{(\infty)})\cap M_{i}$. The opposite inclusion follows from the definitions.
 
 \item[\ref{item:bm_1.5.1_2}\ref{item:bm_1.5.1_2_c}] Let $M\in\calM_{\mathrm{nf}}(H^{(\infty)},\Lambda)$ and $N\unlhd M$ a closed subgroup containing $\mathrm{QZ}(M)$. For every $M'\in\calM_{\mathrm{nf}}(H^{(\infty)},\Lambda)$ with $M\neq M'$ we have
 \begin{displaymath}
  \displaybump [M',M]\subseteq M'\subseteq M\subseteq\mathrm{QZ}(H^{(\infty)})
 \end{displaymath}
 This implies $[M',N]\subseteq\mathrm{QZ}(H^{(\infty)})\cap M=\mathrm{QZ}(M)\subseteq N$, i.e. $M'$ normalizes $N$. Since $N\unlhd M$, this implies $N\unlhd H^{(\infty)}$ and hence, by minimality of $M$, we have either $N=M$ or $N$ acts freely on $E$ and $N\subseteq\mathrm{QZ}(H^{(\infty)})\cap M=\mathrm{QZ}(M)$. \qedhere
\end{enumerate}
\end{proof}

\begin{corollary}\label{cor:bm_1.5.2}
Let $\Gamma=(V,E)$ be a locally finite, connected graph. Further, let $H\le\Aut(\Gamma)$ be closed, non-discrete and locally semiprimitive. Minimal, non-trivial closed normal subgroups of $H^{(\infty)}\!/\mathrm{QZ}(H^{(\infty)})$ exist. They are all $H$-conjugate, finite in number and topologically simple.
\end{corollary}

\begin{proof}
Apply Proposition~\ref{prop:bm_1.5.1} to $\Lambda=\mathrm{QZ}(H^{(\infty)})$.
\end{proof}

We summarize the previous results in the following theorem, which is a verbatim copy of Burger--Mozes' Theorem~\ref{thm:burger_mozes_structure_quasiprimitive}, except that the local action need only be semiprimitive, not quasiprimitive.

\begin{theorem}\label{thm:burger_mozes_structure_semiprimitive}
Let $\Gamma$ be a locally finite, connected graph. Further, let $H\!\le\!\Aut(\Gamma)$ be closed, non-discrete and locally semiprimitive. Then
\begin{enumerate}[(i),series=bm_structure]
 \item\label{item:bm_structure_h_infty} $H^{(\infty)}$ is minimal closed normal cocompact in $H$,
 \item\label{item:bm_structure_qz} $\mathrm{QZ}(H)$ is maximal discrete normal, and non-cocompact in $H$, and
 \item\label{item:bm_structure_subquotient} $H^{(\infty)}\!/\mathrm{QZ}(H^{(\infty)})\!=\! H^{(\infty)}\!/(\mathrm{QZ}(H)\cap H^{(\infty)})$ admits minimal, non-trivial closed normal subgroups; finite in number, $H$-conjugate and topologically simple.
\end{enumerate}
If $\Gamma$ is a tree, and, in addition, $H$ is locally primitive then
\begin{enumerate}[(i),resume=bm_structure]
 \item\label{item:bm_structure_primitive} $H^{(\infty)}\!/\mathrm{QZ}(H^{(\infty)})$ is a direct product of topologically simple groups.
\end{enumerate}
\end{theorem}

\begin{proof}
Parts \ref{item:bm_structure_h_infty} and \ref{item:bm_structure_qz} stem from parts \ref{item:bm_1.2.1_1}, \ref{item:bm_1.2.1_2} and \ref{item:bm_1.2.1_3} of Proposition \ref{prop:bm_1.2.1} in combination with Section \ref{sec:bm_theory}. For part \ref{item:bm_structure_subquotient}, use part \ref{item:bm_1.2.1_4} of Proposition \ref{prop:bm_1.2.1} and Corollary \ref{cor:bm_1.5.2}. Finally, part \ref{item:bm_structure_primitive} is Corollary 1.7.2 in \cite{BM00a}. It follows from Theorem 1.7.1 in \cite{BM00a} as the commutator of any two distinct elements in $\calM_{\mathrm{nf}}(H^{(\infty)},\Lambda)$ is contained in $\mathrm{QZ}(H^{(\infty)})$.
\end{proof}

\section{Universal Groups}\label{sec:universal_groups}

In this section, we develop a generalization of Burger--Mozes universal groups that arises through prescribing the local action on balls of a given radius $k\in\bbN$ around vertices. The Burger--Mozes construction corresponds to the case $k=1$.

Whereas many properties of the original construction carry over to the new setup, others require adjustments. Notably, there are compatibility and discreteness conditions on the local action $F$ under which the associated universal group is locally action isomorphic to $F$ and discrete respectively.

We then exhibit examples and (non)-rigidity phenomena of our construction. Finally, a universality statement holds under an additional assumption.

\subsection{Definition and Basic Properties}

\subsubsection{Definition}

Let $\Omega$ be a set of cardinality $d\in\bbN_{\ge 3}$ and let $T_{d}=(V,E)$ denote the $d$-regular tree. A \emph{labelling} $l$ of $T_{d}$ is a map $l:E\to\Omega$ such that for every $x\in V$ the map $l_{x}\!:E(x)\to\Omega,\ e\mapsto l(e)$ is a bijection, and $l(e)=l(\overline{e})$ for all $e\in E$.

For every $k\in\bbN$, fix a tree $B_{d,k}$ which is isomorphic to a ball of radius $k$ around a vertex in $T_{d}$. Let $b$ denote its center and carry over the labelling of $T_{d}$ to $B_{d,k}$ via the chosen isomorphism. Then for every $x\in V$ there is a unique, label-respecting isomorphism $l_{x}^{k}:B(x,k)\to B_{d,k}$. We define the \emph{$k$-local action} $\sigma_{k}(g,x)\!\in\!\Aut(B_{d,k})$ of an automorphism $g\!\in\!\Aut(T_{d})$ at a vertex $x\in V$ via
\begin{displaymath}
\sigma_{k}:\Aut(T_{d})\times V\to\Aut(B_{d,k}),\ (g,x)\mapsto \sigma_{k}(g,x):=l_{gx}^{k}\circ g\circ (l_{x}^{k})^{-1}.
\end{displaymath}

\begin{definition}\label{def:ukf}
Let $F\le\Aut(B_{d,k})$ and $l$ be a labelling of $T_{d}$. Define
\begin{displaymath}
 \mathrm{U}_{k}^{(l)}(F):=\{g\in\Aut(T_{d})\mid\ \forall x\in V:\ \sigma_{k}(g,x)\in F\}.
\end{displaymath}
\end{definition}

The following lemma states that the maps $\sigma_{k}$ satisfy a cocycle identity which implies that $\smash{\mathrm{U}_{k}^{(l)}(F)}$ is a subgroup of $\Aut(T_{d})$ for every $F\le\Aut(B_{d,k})$.

\begin{lemma}\label{lem:cocycle}
Let $x\in V$ and $g,h\in\Aut(T_{d})$. Then $\sigma_{k}(gh,x)=\sigma_{k}(g,hx)\sigma_{k}(h,x)$.
\end{lemma}

\begin{proof}
We compute
\begin{align*}
  \sigma_{k}(gh,x)=&l_{(gh)x}^{k}\circ gh\circ (l_{x}^{k})^{-1}=l_{(gh)x}^{k}\circ g\circ h\circ(l_{x}^{k})^{-1}= \\
  &=l_{(gh)x}^{k}\circ g\circ (l_{hx}^{k})^{-1}\circ l_{hx}^{k}\circ h\circ(l_{x}^{k})^{-1}=\sigma_{k}(g,hx)\sigma_{k}(h,x).\qedhere
\end{align*}
\end{proof}

\subsubsection{Basic Properties}

Note that the group $\smash{\mathrm{U}_{1}^{(l)}(F)}$ of Definition \ref{def:ukf} coincides with the Burger--Mozes universal group $\smash{\mathrm{U}_{(l)}(F)}$ introduced in \cite[Section 3.2]{BM00a} under the natural isomorphism $\Aut(B_{d,1})\cong\Sym(\Omega)$. Several basic properties of the latter group carry over to the generalized setup. First of all, passing between different labellings of $T_{d}$ amounts to conjugating in $\Aut(T_{d})$. Subsequently, we shall therefore omit the reference to an explicit labelling.

\begin{lemma}\label{lem:ext}
For every quadruple $(l,l',x,x')$ of labellings $l,l'$ of $T_{d}$ and vertices $x,x'\in V$, there is a unique automorphism $g\in\Aut(T_{d})$ with $gx=x'$ and $l'=l\circ g$.
\end{lemma}

\begin{proof}
Set $gx:=x'$. Now assume inductively that $g$ is uniquely determined on $B(x,n)$ $(n\in\bbN_{0})$ and let $v\in S(x,n)$. Then $g$ is also uniquely determined on $E(v)$ by the requirement $l'=l\circ g$, namely $g|_{E(v)}:=l|_{E(gv)}^{-1}\circ l'|_{E(v)}$.
\end{proof}

\begin{proposition}\label{prop:ukf_labelling}
Let $F\le\Aut(B_{d,k})$. Further, let $l$ and $l'$ be labellings of $T_{d}$. Then the groups $\smash{\mathrm{U}_{k}^{(l)}(F)}$ and $\smash{\mathrm{U}_{k}^{(l')}(F)}$ are conjugate in $\Aut(T_{d})$.
\end{proposition}

\begin{proof}
Choose $x\in V$. Let $\tau\in\Aut(T_{d})$ denote the automorphism of $T_{d}$ associated to $(l,l',x,x)$ by Lemma \ref{lem:ext}, then  $\smash{\mathrm{U}_{k}^{(l)}(F)=\tau\mathrm{U}_{k}^{(l')}(F)\tau^{-1}}$.
\end{proof}

The following basic properties of $\mathrm{U}_{k}(F)$ are as in Proposition \ref{prop:uf_basic_properties}.

\begin{proposition}\label{prop:ukf_basic_properties}
Let $F\le\Aut(B_{d,k})$. The group $\mathrm{U}_{k}(F)$ is
\begin{enumerate}[(i)]
 \item closed in $\Aut(T_{d})$,
 \item vertex-transitive, and
 \item\label{item:ukf_comp_gen} compactly generated.
\end{enumerate}
\end{proposition}

\begin{proof}
As to (i), note that if $g\notin\mathrm{U}_{k}(F)$ then $\sigma_{k}(g,x)\notin F$ for some $x\in V$. In this case, the open neighbourhood $\{h\in\Aut(T_{d})\mid h|_{B(x,k)}=g|_{B(x,k)}\}$ of $g$ in $\Aut(T_{d})$ is also contained in the complement of $\mathrm{U}_{k}(F)$.

For (ii), let $x,x'\in V$ and let $g\in\Aut(T_{d})$ be the automorphism of $T_{d}$ associated to $(l,l,x,x')$ by Lemma \ref{lem:ext}. Then $g\in\mathrm{U}_{k}(F)$ as $\sigma_{k}(g,v)=\id\in F$ for all $v\in V$.

To prove (iii), fix $x\in V$. We show that $\mathrm{U}_{k}(F)$ is generated by the join of the compact set $\mathrm{U}_{k}(F)_{x}$ and the finite generating set of $\mathrm{U}_{1}(\{\id\})=\mathrm{U}_{k}(\{\id\})\le\mathrm{U}_{k}(F)$ guaranteed by Lemma \ref{lem:uid_fin_gen}: Indeed, for $g\in\mathrm{U}_{k}(F)$ pick $g'$ in the finitely generated, vertex-transitive subgroup $\mathrm{U}_{1}(\{\id\})$ of $\mathrm{U}_{k}(F)$ such that $g'gx=x$. We then have $g'g\in\mathrm{U}_{k}(F)_{x}$ and the assertion follows.
\end{proof}

For completeness, we explicitly state the following.

\begin{proposition}
Let $F\le\Aut(B_{d,k})$. Then $\mathrm{U}_{k}(F)$ is a compactly generated, totally disconnected, locally compact, second countable group.
\end{proposition}

\begin{proof}
The group $\mathrm{U}_{k}(F)$ is totally disconnected, locally compact, second countable as a closed subgroup of $\Aut(T_{d})$ and compactly generated by Proposition \ref{prop:ukf_basic_properties}.
\end{proof}

Finally, we record that the groups $\mathrm{U}_{k}(F)$ are $(P_{k})$-closed.

\begin{proposition}\label{prop:ukf_pk}
Let $F\le\Aut(B_{d,k})$. Then $\mathrm{U}_{k}(F)$ satisfies Property $(P_{k})$.
\end{proposition}

\begin{proof}
Let $e=(x,y)\in E$. Clearly, $\mathrm{U}_{k}(F)_{e^{k}}\supseteq\mathrm{U}_{k}(F)_{e^{k},T_{y}}\cdot\mathrm{U}_{k}(F)_{e^{k},T_{x}}$. Conversely, consider $g\in\mathrm{U}_{k}(F)_{e^{k}}$ and define $g_{y}\in\Aut(T_{d})$ and $g_{x}\in\Aut(T_{d})$ by
\begin{displaymath}
 \sigma_{k}(g_{y},v)=\begin{cases}\sigma_{k}(g,v) & v\in V(T_{x}) \\ \id & v\in V(T_{y})\end{cases} \quad\text{and}\quad \sigma_{k}(g_{x},v)=\begin{cases}\id & v\in V(T_{x}) \\ \sigma_{k}(g,v) & v\in V(T_{y})\end{cases}
\end{displaymath}
respectively. Then $g_{y}\in\mathrm{U}_{k}(F)_{e^{k},T_{y}}$, $g_{x}\in\mathrm{U}_{k}(F)_{e^{k},T_{x}}$ and $g=g_{y}\circ g_{x}$.
\end{proof}

\subsection{Compatibility and Discreteness}\label{sec:comp_disc}

We now generalize parts \ref{item:uf_local_f} and \ref{item:uf_discrete} of Burger--Mozes' Proposition \ref{prop:uf_basic_properties}. There are compatibility and discreteness conditions (C) and (D) on subgroups $F\le\Aut(B_{d,k})$ that hold if and only if the associated universal group is locally action isomorphic to $F$ and discrete respectively.

We introduce the following notation for vertices in the labelled tree $(T_{d},l)$: Given $x\in V$ and $w=(\omega_{1},\ldots,\omega_{n})\in\Omega^{n}$ $(n\in\bbN_{0})$, set $x_{w}:=\gamma_{x,w}(n)$ where
\begin{displaymath}
\psset{xunit=1.75cm}
\gamma_{x,w}:\mathrm{Path}_{n}^{(w)}:=\pspicture(-0.1,-0.075)(2.1,0.5)\psline(0,0)(1.25,0)\psline(1.75,0)(2,0)\psdots[](0,0)(0.5,0)(1,0)(2,0)\uput[d](0,0){$0$}\uput[d](0.5,0){$1$}\uput[d](1,0){$2$}\rput(1.5,0){$\ldots$}\uput[d](2,0){$n$}\uput[u](0.25,0){$\omega_{1}$}\uput[u](0.75,0){$\omega_{2}$}\endpspicture\to T_{d}
\end{displaymath}

\vspace{0.475cm}
\noindent
is the unique label-respecting morphism sending $0$ to $x\in V$. If $w$ is the empty word, set $x_{w}:=x$. Whenever admissible, we also adopt this notation in the case of $B_{d,k}$ and its labelling. In particular, $S(x,n)$ is in natural bijection with the set $\Omega^{(n)}:=\{(\omega_{1},\ldots,\omega_{n})\in\Omega^{n}\mid \forall k\in\{1,\ldots,n-1\}:\ \omega_{k+1}\neq\omega_{k}\}$.

\subsubsection{Compatibility}

First, we ask whether $\mathrm{U}_{k}(F)$ locally acts like $F$, that is whether the actions $\mathrm{U}_{k}(F)_{x}\curvearrowright B(x,k)$ and $F\curvearrowright B_{d,k}$ are isomorphic for every $x\in V$. Whereas this always holds for $k=1$ by Proposition \ref{prop:uf_basic_properties}\ref{item:uf_local_f} it need not be true for $k\ge 2$, the issue being (non)-compatibility among elements of $F$. See Example \ref{ex:ukf_not_c}. The condition developed in this section allows for computations. A more practical version from a theoretical viewpoint follows in Section \ref{sec:ukf_examples}.

\vspace{0.2cm}
Now, let $x\in V$ and suppose that $\alpha\in\mathrm{U}_{k}(F)_{x}$ realizes $a\in F$ at $x$, that is
\begin{displaymath}
 \alpha|_{B(x,k)}=(l_{x}^{k})^{-1}\circ a\circ l_{x}^{k}.
\end{displaymath}
Then given the condition that $\sigma_{k}(\alpha,x_{\omega})$ be in $F$ for all $\omega\in\Omega$, we obtain the following necessary \emph{compatibility condition} on $F$ for $\mathrm{U}_{k}(F)$ to act like $F$ at $x\in V$:
\begin{displaymath}
 \forall a\in F\ \forall \omega\in\Omega:\ \exists a_{\omega}\in F:\ (l_{x}^{k})^{-1}\circ a\circ l_{x}^{k}|_{S_{\omega}}=(l_{\alpha x_{\omega}}^{k})^{-1}\circ a_{\omega}\circ l_{x_{\omega}}^{k}|_{S_{\omega}}
\end{displaymath}
where $S_{\omega}:=B(x,k)\cap B(x_{\omega},k)\subseteq T_{d}$. Set $T_{\omega}:=l_{x}^{k}(S_{\omega})\subseteq B_{d,k}$. Then the above condition can be rewritten as
\begin{displaymath}
 \forall a\in F\ \forall \omega\in\Omega:\ \exists a_{\omega}\in F:\ a_{\omega}|_{T_{\omega}}=l_{\alpha x_{\omega}}^{k}\circ(l_{x}^{k})^{-1}\circ a\circ l_{x}^{k}\circ(l_{x_{\omega}}^{k})^{-1}|_{T_{\omega}}.
\end{displaymath}
Now observe the following: First, $\alpha x_{\omega}$ depends only on $a$. Second, the subtree $T_{\omega}$ of $B_{d,k}$ does not depend on $x$. Third, $\iota_{\omega}:=l_{x}^{k}|^{T_{\omega}}\circ(l_{x_{\omega}}^{k})^{-1}|_{T_{\omega}}$ is the unique non-trivial, involutive and label-respecting automorphism of $T_{\omega}$; it is given by
\begin{displaymath}
 \iota_{\omega}:=\left.l_{x}^{k}\right|^{T_{\omega}}\circ  \left.(l_{x_{\omega}}^{k})^{-1}\right|_{T_{\omega}}:T_{\omega}\to S_{\omega}\to T_{\omega},\ b_{w}\mapsto x_{\omega w}\mapsto b_{\omega w}
\end{displaymath}
for admissible words $w$. Hence the above condition may be rewritten as
\begin{equation}
  \forall a\in F\ \forall \omega\in\Omega:\ \exists a_{\omega}\in F:\ a_{\omega}|_{T_{\omega}}=\iota_{a\omega}\circ a\circ\iota_{\omega}.
  \label{eq:C}
  \tag{C}
\end{equation}
In this situation we shall say that \emph{$a_{\omega}$ is compatible with $a$ in direction $\omega$}.

\begin{proposition}\label{prop:ukf_local_f}
Let $F\le\Aut(B_{d,k})$. Then $\mathrm{U}_{k}(F)$ is locally action isomorphic to $F$ if and only if $F$ satisfies \eqref{eq:C}.
\end{proposition}

\begin{proof}
By the above, condition \eqref{eq:C} is necessary. To show that it is also sufficient, let $x\in V$ and $a\in F$. We aim to define an automorphism $\alpha\in\mathrm{U}_{k}(F)$ which realizes $a$ at $x$. This forces us to define
\begin{displaymath}
 \alpha|_{B(x,k)}:=(l_{x}^{k})^{-1}\circ a\circ l_{x}^{k}.
\end{displaymath}
Now, assume inductively that $\alpha$ is defined consistently on $B(x,n)$ in the sense that $\sigma_{k}(\alpha,y)\in F$ for all $y\in B(x,n)$ with $B(y,k)\subseteq B(x,n)$. In order to extend $\alpha$ to $B(x,n+1)$, let $y\in S(x,n-k+1)$ and let $\omega\in\Omega$ be the unique label such that $y_{\omega}\in S(x,n-k)$. Set $c:=\sigma_{k}(\alpha,y_{\omega})$. Applying condition \eqref{eq:C} to the pair $(c,\omega)$ yields an element $c_{\omega}\in F$ such that
\begin{displaymath}
 \left.(l_{\alpha y_{\omega}}^{k})^{-1}\circ c\circ l_{y_{\omega}}^{k}\right|_{S_{\omega}}=\left.(l_{\alpha y}^{k})^{-1}\circ c_{\omega}\circ l_{y}^{k}\right|_{S_{\omega}}
\end{displaymath}
where $S_{\omega}:=B(y,k)\cap B(y_{\omega},k)$ and we have realized
\begin{displaymath}
 \iota_{\omega} \text{ as } \left.l_{y_{\omega}}^{k}\right|^{T_{\omega}}\circ  \left.(l_{y}^{k})^{-1}\right|_{T_{\omega}} \quad\text{and}\quad	\iota_{c\omega} \text{ as } \left.l_{\alpha y}^{k}\right|^{T_{c\omega}}\circ  \left.(l_{\alpha y_{\omega}}^{k})^{-1}\right|_{T_{c\omega}}.
\end{displaymath}
Now extend $\alpha$ consistently to $B(v,n+1)$ by setting $\alpha|_{B(x,k)}:=(l_{\alpha x}^{k})^{-1}\circ c_{\omega}\circ l_{x}^{k}$.
\end{proof}

\begin{example}\label{ex:ukf_not_c}
Let $\Omega:=\{1,2,3\}$ and $a\in\Aut(B_{3,2})$ be the element which swaps the leaves $b_{12}$ and $b_{13}$ of $B_{3,2}$. Then $F:=\langle a\rangle=\{\id,a\}$ does not contain an element compatible with $a$ in direction $1\in\Omega$ and hence does not satisfy condition~\eqref{eq:C}.
\end{example}

We show that it suffices to check condition \eqref{eq:C} on the elements of a generating~set. Let $F\le\Aut(B_{d,k})$ and $a,b\in F$. Set $c:=ab$. Then
\begin{align*}
  c_{\omega}|_{T_{\omega}}=\iota_{c\omega}\circ a\circ b\circ\iota_{\omega}&=\left(\iota_{c\omega}\circ a\circ\iota_{b\omega}\right)\circ\left(\iota_{b\omega}\circ b\circ\iota_{\omega}\right) \\
  &=\left(\iota_{a(b\omega)}\circ a\circ\iota_{b\omega}\right)\circ\left(\iota_{b\omega}\circ b\circ\iota_{\omega}\right).
  \label{eq:M}
  \tag{M}
\end{align*}
Let $C_{F}(a,\omega)$ denote the \emph{compatibility set} of elements in $F$ which are compatible with $a\in F$ in direction $\omega\in\Omega$. Then \eqref{eq:M} shows that $C_{F}(ab,\omega)\supseteq C_{F}(a,b\omega)C_{F}(b,\omega)$. It therefore suffices to check condition \eqref{eq:C} on a generating set of $F$.

Given $S\subseteq\Omega$, we also define $C_{F}(a,S):=\bigcap_{\omega\in S}C_{F}(a,\omega)$, the set of elements in $F$ which are compatible with $a\in F$ in all directions from $S$. We omit $F$ in this notation when it is clear from the context.

\vspace{0.2cm}
As a consequence, we obtain the following description of the local action of $\mathrm{U}_{k}(F)$ when $F$ does not satisfy condition \eqref{eq:C}.

\begin{proposition}\label{prop:ukf_max_c}
Let $F\!\le\!\Aut(B_{d,k})$. Then $F$ has a unique maximal subgroup $C(F)$ which satisfies \eqref{eq:C}. We have $C(C(F))\!=\!C(F)$ and $\mathrm{U}_{k}(F)\!=\!\mathrm{U}_{k}(C(F))$.
\end{proposition}

\begin{proof}
By the above, $C(F)\!:=\!\langle H\le F\mid H \text{ satisfies \eqref{eq:C}}\rangle\!\le\! F$ satisfies condition \eqref{eq:C}. It is the unique maximal such subgroup of $F$ by definition, and $C(C(F))=C(F)$.

Furthermore, $\mathrm{U}_{k}(C(F))\le\mathrm{U}_{k}(F)$. Conversely, suppose $g\in\mathrm{U}_{k}(F)\backslash\mathrm{U}_{k}(C(F))$. Then there is $x\in V$ such that $\sigma_{k}(g,x)\in F\backslash C(F)$ and the group
\begin{displaymath}
  C(F)\lneq\langle C(F),\{\sigma_{k}(g,x)\mid x\in V\}\rangle\le F
\end{displaymath}
satisfies condition \eqref{eq:C}, too, as can be seen by setting $\sigma_{k}(g,x)_{\omega}:=\sigma_{k}(g,x_{\omega})$. This contradicts the maximality of $C(F)$.
\end{proof}

\begin{remark}\label{rem:ukf_elements}
Let $F\le\Aut(B_{d,k})$ satisfy \eqref{eq:C}. The proof of Proposition~\ref{prop:ukf_local_f} shows that elements of $\mathrm{U}_{k}(F)$ are readily constructed: Given $x,y\in V(T_{d})$ and $a\in F$, define $g:B(x,k)\to B(y,k)$ by setting $g(x)=y$ and $\sigma_{k}(g,x)=a$. Then, given elements $a_{\omega}\in F$ $(\omega\in\Omega)$ such that $a_{\omega}\in C_{F}(a,\omega)$ for all $\omega\in\Omega$, there is a unique extension of $g$ to $B(x,k+1)$ so that $\sigma_{k}(g,x_{\omega})=a_{\omega}$ for all $\omega\in\Omega$. Proceed iteratively.
\end{remark}

\subsubsection{Discreteness}\label{sec:ukf_discreteness} The group $F\le\Aut(B_{d,k})$ also determines whether or not $\mathrm{U}_{k}(F)$ is discrete. In fact, the following proposition generalizes Proposition \ref{prop:uf_basic_properties}\ref{item:uf_discrete}.

\begin{proposition}\label{prop:ukf_discrete}
Let $F\le\Aut(B_{d,k})$. Then $\mathrm{U}_{k}(F)$ is discrete if $F$ satisfies 
\begin{equation}
  \forall \omega\in\Omega:\ F_{T_{\omega}}=\{\id\}.
  \label{eq:D}
  \tag{D}
\end{equation}
Conversely, if $\mathrm{U}_{k}(F)$ is discrete and $F$ satisfies \eqref{eq:C}, then $F$ satisfies \eqref{eq:D}.
\end{proposition}

Alternatively, $\mathrm{U}_{k}(F)$ is discrete if and only if $C(F)$ satisfies \eqref{eq:D}. Example \ref{ex:ukf_not_c} shows that condition \eqref{eq:C} is necessary for the second part of Proposition \ref{prop:ukf_discrete}.

Finally, note that $F$ satisfies \eqref{eq:D} if and only if $C_{F}(\id,\omega)=\{\id\}$ for all $\omega\in\Omega$.

\begin{proof}
(Proposition \ref{prop:ukf_discrete}). Fix $x\in V$. A subgroup $H\le\Aut(T_{d})$ is non-discrete if and only if for every $n\in\bbN$ there is $h\in H\backslash \{\id\}$ such that $h|_{B(x,n)}=\id$.

Suppose that $\mathrm{U}_{k}(F)$ is non-discrete. Then there are $n\in\bbN_{\ge k}$ and $\alpha\in\text{U}_{k}(F)$ such that $\alpha|_{B(x,n)}=\id$ and $\alpha|_{B(x,n+1)}\neq\id$. Hence there is $y\in S(x,n-k+1)$ with $a:=\sigma_{k}(\alpha,y)\neq\id$. In particular, $a\in F_{T_{\omega}}\backslash\{\id\}$ where $\omega$ is the label of the unique edge $e\in E$ with $o(e)=y$ and $d(x,y)=d(x,t(e))+1$.

Conversely, suppose that $F$ satisfies \eqref{eq:C} and $F_{T_{\omega}}\neq\{\id\}$ for some $\omega\in\Omega$. Then for every $n\in\bbN_{\ge k}$, we define an automorphism $\alpha\in\text{U}_{k}(F)$ with $\alpha|_{B(x,n)}=\id$ and $\alpha|_{B(x,n+1)}\neq\id$: If $\alpha|_{B(x,n)}=\id$, then $\sigma_{k}(\alpha,y)\in F$ for all $y\in B(x,n-k)$. Choose $e\in E$ with $y:=o(e)\in S(x,n-k+1)$ and $t(e)\in S(x,n-k)$ such that $l(e)=\omega$. We extend $\alpha$ to $B(y,k)$ by setting $\alpha|_{B(y,k)}:=l_{y}^{k}\circ s\circ(l_{y}^{k})^{-1}$ where $s\in F_{T_{\omega}}\backslash\{\id\}$. Finally, we extend $\alpha$ to $T_{d}$ using \eqref{eq:C}.
\end{proof}

We define condition (CD) on $F\le\Aut(B_{d,k})$ as the conjunction of \eqref{eq:C} and \eqref{eq:D}. The following description is immediate from the above.
\begin{equation}
  \forall a\in F\ \forall \omega\in\Omega:\ \exists!\ a_{\omega}\in F:\ a_{\omega}|_{T_{\omega}}=\iota_{a\omega}\circ a\circ\iota_{\omega}.
  \label{eq:CD}
  \tag{CD}
\end{equation}
When $F$ satisfies \eqref{eq:CD}, an element of $\mathrm{U}_{k}(F)_{x}$ is determined by its action on $B(x,k)$. Hence $\mathrm{U}_{k}(F)_{x}\cong F$ for every $x\in V$ and $\mathrm{U}_{k}(F)_{(x,y)}\cong F_{(b,b_{\omega})}$ for every $(x,y)\in E$ with $l(x,y)=\omega$.
Furthermore, $F$ admits a unique \emph{involutive compatibility cocycle}, i.e. a map $z:F\times\Omega\to F,\ (a,\omega)\mapsto a_{\omega}$ which for all $a,b\in F$ and $\omega\in\Omega$ satisfies
\begin{enumerate}[(i)]
 \item (compatibility) $z(a,\omega)\in C_{F}(a,\omega)$,
 \item (cocycle) $z(ab,\omega)=z(a,b\omega)z(b,\omega)$, and 
 \item (involutive) $z(z(a,\omega),\omega)=a$.
\end{enumerate}
Note that $z$ restricts to an automorphism $z_{\omega}$ of $F_{(b,b_{\omega})}$ $(\omega\in\Omega)$ of order at most $2$.

\subsection{Group Structure}
For $\smash{\widetilde{F}\le\Aut(B_{d,k})}$, let $\smash{F:=\pi\widetilde{F}\le\Sym(\Omega)}$ denote the projection of $\smash{\widetilde{F}}$ onto $\Aut(B_{d,1})\cong\Sym(\Omega)$. As an illustration, we record that the group structure of $\smash{\mathrm{U}_{k}(\widetilde{F})}$ is particularly clear when $F$ is regular.

\begin{proposition}\label{prop:ukf_structure_reg}
Let $\smash{\widetilde{F}\le\Aut(B_{d,k})}$ satisfy \eqref{eq:C}. Suppose $\smash{F:=\pi\widetilde{F}}$ is regular. Then $\smash{\mathrm{U}_{k}(\widetilde{F})=\mathrm{U}_{1}(F)\cong F\ast\bbZ\!/2\bbZ}$.
\end{proposition}

\begin{proof}
Fix $x\in V$. Since $F$ is transitive, the group $\mathrm{U}_{k}(\widetilde{F})$ is generated by $\mathrm{U}_{k}(\widetilde{F})_{x}$ and an involution $\iota$ inverting an edge with origin $x$. Given $\smash{\alpha\in\mathrm{U}_{k}(\widetilde{F})_{x}}$, regularity of $F$ implies that $\sigma_{1}(\alpha,y)=\sigma_{1}(\alpha,x)\in F$ for all $y\in V$. Now, the subgroups $H_{1}:=\smash{\mathrm{U}_{k}(\widetilde{F})_{x}\cong F}$ and $H_{2}:=\langle\iota\rangle$ of $\smash{\mathrm{U}_{k}(\widetilde{F})}$ generate a free product within $\mathrm{U}_{k}(F)$ by the ping-pong lemma: Put $X_{1}:=V(T_{x})$ and $X_{2}:=V(T_{x_{\omega}})$. Any non-trivial element of $H_{1}$ maps $X_{2}$ into $X_{1}$ as $F_{\omega}=\{\id\}$, and $\iota\in H_{2}$ maps $X_{1}$ into $X_{2}$.
\end{proof}

More generally, Bass-Serre theory \cite{Ser03} identifies the universal groups $\mathrm{U}_{k}(F)$ as amalgamated free products, taking into account that $\mathrm{U}_{k}(F)$ acts with inversions.

\begin{proposition}\label{prop:ukf_structure}
Let $\smash{F\le\Aut(B_{d,k})}$ satisfy \eqref{eq:C} (and \eqref{eq:D}). If $\pi F$ is transitive,
\begin{displaymath}
 \mathrm{U}_{k}(F)\cong\mathrm{U}_{k}(F)_{x}\underset{\text{\makebox[0.5cm]{\raisebox{0pt}[0.3cm]{$\mathrm{U}_{k}(F)_{(x,y)}$}}}}{\ast}\mathrm{U}_{k}(F)_{\{x,y\}}\left(\cong F\underset{\text{\makebox[0.5cm]{\raisebox{0pt}[0.3cm]{$F_{(b,b_{\omega})}$}}}}{\ast}(F_{(b,b_{\omega})}\rtimes\bbZ\!/2\bbZ)\right)
\end{displaymath}
for any edge $(x,y)\in E$, where $\omega=l(x,y)$ and $\bbZ\!/2\bbZ$ acts on $F_{(b,b_{\omega})}$ as $z_{\omega}$.
\end{proposition}

\begin{corollary}\label{cor:ukf_cd_iso}
Let $F,F'\!\le\!\Aut(\! B_{d,k})$ satisfy \eqref{eq:CD}. If there are $\omega,\omega'\in\Omega$ and an isomorphism $\varphi\!:\!F\!\to\! F'$ such that $\varphi(F_{(b,b_{\omega})})=F'_{(b,b_{\omega'})}$, then $\mathrm{U}_{k}(F)\cong\mathrm{U}_{k}(F')$. \qed
\end{corollary}

Note that Corollary \ref{cor:ukf_cd_iso} applies to conjugate subgroups of $\Aut(B_{d,k})$ which satisfy \eqref{eq:CD}. The following example shows that the assumption that both $F$ and $F'$ in Corollary \ref{cor:ukf_cd_iso} satisfy \eqref{eq:CD} is indeed necessary.

\begin{example}
Let $\Omega:=\{1,2,3\}$ and $t\in\Aut(B_{3,2})$ be the element which swaps the leaves $x_{12}$ and $x_{13}$ of $B_{3,2}$. Using the notation of Section \ref{sec:ukf_examples_k=2}, consider the group $\Gamma(A_{3})\!\le\!\Aut(B_{3,2})$ which satisfies \eqref{eq:C}. In particular, $\mathrm{U}_{2}(\Gamma(A_{3}))\!\cong\! A_{3}\ast\bbZ\!/2\bbZ$ by Proposition \ref{prop:ukf_structure_reg}. On the other hand, set $F':=t\Gamma(A_{3})t^{-1}$. Then $\pi F'=A_{3}$ while for a non-trivial element $\alpha$ of $F'$, we have $\sigma_{1}(\alpha,b_{\omega})\!\in\! S_{3}\backslash A_{3}$ for some $\omega\in\Omega$. Therefore, $\mathrm{U}_{2}(F')=\mathrm{U}_{1}(\{\id\})$ is isomorphic to $\bbZ\!/2\bbZ\ast\bbZ\!/2\bbZ\ast\bbZ\!/2\bbZ$ by Lemma \ref{lem:uid_fin_gen}. In particular, $\mathrm{U}_{2}(\Gamma(A_{3}))$ and $\mathrm{U}_{2}(t\Gamma(A_{3})t^{-1})$ are not isomorphic.
\end{example}

Conversely, the following Proposition based on \cite[Appendix A]{Rad17}, which states that in certain cases the tree can be recovered from the topological group structure of a subgroup of $\Aut(T_{d})$, applies to appropriate universal groups.

\begin{proposition}\label{prop:isomorphism_conjugate}
Let $H,H'\le\Aut(T_{d})$ be closed and locally transitive with distinct point stabilizers. Then $H$ and $H'$ are isomorphic topological groups if and only if they are conjugate in $\Aut(T_{d})$.
\end{proposition}

\begin{proof}
By \cite{FTN91}, every compact subgroup of $H$ is either contained in a vertex stabilizer $H_{x}$ $(x\in V)$ or, in case $H\not\leq\Aut(T_{d})^{+}$, in a geometric edge stabilizer $H_{\{e,\overline{e}\}}$ $(e\in E)$. Since $H$ is locally transitive, the above are pairwise distinct.

The vertex stabilizers are precisely those maximal compact subgroups $K\le H$ for which there is no maximal compact subgroup $K'$ with $[K:K\cap K']=2$: Indeed, for $e\in E$ and $x\in\{o(e),t(e)\}$ we have $[H_{\{e,\overline{e}\}}:H_{\{e,\overline{e}\}}\cap H_{x}]=2$ whereas $[H_{x}:H_{x}\cap H_{y}],[H_{x}:H_{x}\cap H_{\{e,\overline{e}\}}]\ge 3$ for all distinct $x,y\in V$ and $e\in E$ by the orbit-stabilizer theorem because $d\ge 3$ and $H$ is locally transitive.

Adjacency can be expressed in terms of indices as well: Let $x,y\in V$ be distinct. Then $(x,y)\in E$ if and only if $[H_{x}:H_{x}\cap H_{y}]\le [H_{x}:H_{x}\cap H_{z}]$ for all $z\in V$: Indeed, if $(x,y)\in E$, then $[H_{x}:H_{x}\cap H_{y}]=d$ by the orbit-stabilizer theorem given that $H$ is locally transitive. If $z\in V$ is not adjacent to $x$ then $[H_{x}:H_{x}\cap H_{z}]>d$ because point stabilizers of every local action of $H$ are distinct.

Now, let $\Phi:H\to H'$ be an isomorphism of topological groups. Then $\Phi$ induces a bijection between the maximal compact subgroups of $H$ and $H'$, and preserves indices. Hence there is an automorphism $\varphi\in\Aut(T_{d})$ such that $\Phi(H_{x})=H'_{\varphi(x)}$ for all $x\in V$. Furthermore, since vertex stabilizers in $H'$ are pairwise distinct and
\begin{displaymath}
 H'_{\varphi h\varphi^{-1}(x)}=\Phi(H_{h\varphi^{-1}(x)})=\Phi(hH_{\varphi^{-1}(x)}h^{-1})=\Phi(h)H'_{x}\Phi(h^{-1})=H'_{\Phi(h)x}
\end{displaymath}
for all $x\in V$ we have $\varphi h\varphi^{-1}=\Phi(h)$ for all $h\in H$.
\end{proof}

The following Corollary uses the notation $\Phi^{k}(F')$ from Section \ref{sec:ukf_examples_general_case}.

\begin{corollary}\label{cor:ukf_c_iso}
Let $F\!\le\!\Aut(B_{d,k})$ and $F'\!\le\!\Aut(B_{d,k'})$ satisfy \eqref{eq:C}. Assume $k\!\ge\! k'$ and $\pi F,\pi F'\!\le\!\Sym(\Omega)$ are transitive with distinct point stabilizers. If $\mathrm{U}_{k}(F)$ and $\mathrm{U}_{k'}(F')$ are isomorphic topological groups then $F,\Phi^{k}(F')\!\le\!\Aut(\!B_{d,k})$ are conjugate.
\end{corollary}

\begin{proof}
By Proposition \ref{prop:isomorphism_conjugate}, the groups $\mathrm{U}_{k}(F)$ and $\mathrm{U}_{k}(F')$ are conjugate in $\Aut(T_{d})$, hence so are $\mathrm{U}_{k}(F)_{x}$ and $\mathrm{U}_{k'}(F')_{x}$ for every $x\in V$ and the assertion follows.
\end{proof}

\begin{example}
Section~\ref{sec:ukf_examples_k=2} introduces the isomorphic, non-conjugate subgroups $\Pi(S_{3},\mathrm{sgn},\{1\})$ and $\Pi(S_{3},\mathrm{sgn},\{0,1\})$ of $\Aut(B_{3,2})$, both of which project onto $S_{3}$ and satisfy \eqref{eq:C} but not~\eqref{eq:D}. An explicit isomorphism satisfies the assumption of Corollary \ref{cor:ukf_cd_iso}. However, by Corollary \ref{cor:ukf_c_iso} the universal groups $\mathrm{U}_{2}(\Pi(S_{3},\mathrm{sgn},\{1\}))$ and $\mathrm{U}_{2}(\Pi(S_{3},\mathrm{sgn},\{0,1\}))$ are non-isomorphic. Therefore, Corollary \ref{cor:ukf_cd_iso} does not generalize to the non-discrete case.
\end{example}

\begin{question}
Let $F,F'\le\Aut(B_{d,k})$ satisfy \eqref{eq:C} and be conjugate. Are the associated universal groups $\mathrm{U}_{k}(F)$ and $\mathrm{U}_{k}(F')$ necessarily isomorphic?
\end{question}

In the following, we determine the Burger--Mozes subquotient $H^{(\infty)}\!/\mathrm{QZ}(H^{(\infty)})$ of Theorem~\ref{thm:burger_mozes_structure_semiprimitive} for non-discrete, locally semiprimitive universal groups.

\begin{proposition}\label{prop:ukf_qz}
Let $F\le\Aut(B_{d,k})$ satisfy \eqref{eq:C}. If, in addition, $F$ satisfies \eqref{eq:D} then $\mathrm{QZ}(\mathrm{U}_{k}(F))=\mathrm{U}_{k}(F)$. Otherwise, $\mathrm{QZ}(\mathrm{U}_{k}(F))=\{\id\}$.
\end{proposition}

\begin{proof}
If $F$ satisfies \eqref{eq:D} then $\mathrm{U}_{k}(F)$ is discrete and hence $\mathrm{QZ}(\mathrm{U}_{k}(F))=\mathrm{U}_{k}(F)$. Conversely, if $F$ satisfies \eqref{eq:C} but not \eqref{eq:D} then the stabilizer of any half-tree $T\subseteq T_{d}$ in $\mathrm{U}_{k}(F)$ is non-trivial: We have $T\in\{T_{x},T_{y}\}$ for some edge $e:=(x,y)\in E$. Since $\mathrm{U}_{k}(F)$ is non-discrete by Proposition \ref{prop:ukf_discrete} and has Property $(P_{k})$ by Proposition~\ref{prop:ukf_pk}, the group $\mathrm{U}_{k}(F)_{e^{k}}=\mathrm{U}_{k}(F)_{e^{k},T_{y}}\cdot\mathrm{U}_{k}(F)_{e^{k},T_{x}}$ is non-trivial. In particular, either $\mathrm{U}_{k}(F)_{T_{x}}$ or $\mathrm{U}_{k}(F)_{T_{y}}$ is non-trivial. In view of the existence of label-respecting inversions, both are non-trivial and hence so is $\mathrm{U}_{k}(F)_{T}$. Therefore, $\mathrm{U}_{k}(F)$ has Property H of M{\"o}ller--Vonk \cite[Definition 2.3]{MV12} and \cite[Proposition 2.6]{MV12} implies that $\mathrm{U}_{k}(F)$ has trivial quasi-center.
\end{proof}

\begin{proposition}\label{prop:ukf_bm_subquotient}
Let $F\le\Aut(B_{d,k})$ satisfy \eqref{eq:C} but not \eqref{eq:D}. Suppose that $\pi F$ is semiprimitive. Then $\mathrm{U}_{k}(F)^{(\infty)}/\mathrm{QZ}(\mathrm{U}_{k}(F)^{(\infty)})=\mathrm{U}_{k}(F)^{(\infty)}=\mathrm{U}_{k}(F)^{+_{k}}$.
\end{proposition}

\begin{proof}
The subgroup $\mathrm{U}_{k}(F)^{+_{k}}\le\mathrm{U}_{k}(F)$ is open, hence closed, and normal in $\mathrm{U}_{k}(F)$ by definition. Since $\mathrm{U}_{k}(F)$ is non-discrete by Proposition \ref{prop:ukf_discrete}, so is $\mathrm{U}_{k}(F)^{+_{k}}$. Using Proposition \ref{prop:bm_1.2.1}\ref{item:bm_1.2.1_3}, we conclude that $\mathrm{U}_{k}(F)^{+_{k}}\ge\mathrm{U}_{k}(F)^{(\infty)}$. Since $\mathrm{U}_{k}(F)$ satisfies Property $(P_{k})$ by Proposition \ref{prop:ukf_pk}, the group $\mathrm{U}_{k}(F)^{+_{k}}$ is simple due to Theorem \ref{thm:bew_simplicity}. Thus $\mathrm{U}_{k}(F)^{+_{k}}=\mathrm{U}_{k}(F)^{(\infty)}$. Given that $\mathrm{QZ}(\mathrm{U}_{k}(F)^{(\infty)})=\mathrm{QZ}(\mathrm{U}_{k}(F))\cap\mathrm{U}_{k}(F)^{(\infty)}$ by Proposition \ref{prop:bm_1.2.1}\ref{item:bm_1.2.1_4}, the assertion follows from Proposition \ref{prop:ukf_qz}.
\end{proof}

In the context of Proposition \ref{prop:ukf_bm_subquotient}, the group $\mathrm{U}_{k}(F)^{+_{k}}$ is simple, compactly generated, non-discrete, totally disconnected, locally compact, second countable. Compact generation follows from \cite[Corollary 2.11]{KM08} given that $\smash{\mathrm{U}_{k}(F)^{+_{k}}}$ is cocompact in $\mathrm{U}_{k}(F)$ by Proposition \ref{prop:bm_1.2.1}\ref{item:bm_1.2.1_1}.

\subsection{Examples}\label{sec:ukf_examples}

We now construct various classes of examples of subgroups of $\Aut(B_{d,k})$ satisfying \eqref{eq:C} or \eqref{eq:CD}, and prove a rigidity result for certain local actions.

\vspace{0.2cm}
First, we give a suitable realization of $\Aut(B_{d,k})$ and the conditions \eqref{eq:C} and \eqref{eq:D}. Namely, we view an automorphism $\alpha$ of $B_{d,k}$ as the set $\{\sigma_{k-1}(\alpha,v)\mid v\in B(b,1)\}$ as follows: Let $\Aut(B_{d,1})\cong\Sym(\Omega)$ be the natural isomorphism. For $k\ge 2$, we iteratively identify $\Aut(B_{d,k})$ with its image under the map
\begin{displaymath}
 \Aut(B_{d,k})\to\Aut(B_{d,k-1})\ltimes\prod\nolimits_{\omega\in\Omega}\Aut(B_{d,k-1}),\ \alpha\mapsto(\sigma_{k-1}(\alpha,b),(\sigma_{k-1}(\alpha,b_{\omega}))_{\omega})
 \end{displaymath}
where $\Aut(B_{d,k-1})$ acts on $\smash{\prod_{\omega\in\Omega}\Aut(B_{d,k-1})}$ by permuting the factors according to its action on $S(b,1)\cong\Omega$. That is, multiplication in $\Aut(B_{d,k})$ is given by
\begin{displaymath}
 (\alpha,(\alpha_{\omega})_{\omega\in\Omega})\circ (\beta,(\beta_{\omega})_{\omega\in\Omega})=(\alpha\beta,(\alpha_{\beta\omega}\beta_{\omega})_{\omega\in\Omega}).
\end{displaymath}
Consider the homomorphism $\pi_{k-1}:\Aut(B_{d,k})\to\Aut(B_{d,k-1}),\ \alpha\mapsto\sigma_{k-1}(\alpha,b)$, the projections $\pr_{\omega}:\Aut(B_{d,k})\to\Aut(B_{d,k-1}),\ \alpha\mapsto\sigma_{k-1}(\alpha,b_{\omega})$ $(\omega\in\Omega)$, and
\begin{displaymath}
 p_{\omega}=(\pi_{k-1},\pr_{\omega}):\Aut(B_{d,k})\to\Aut(B_{d,k-1})\times\Aut(B_{d,k-1}),
\end{displaymath}
whose image we interpret as a relation on $\Aut(B_{d,k-1})$. The conditions \eqref{eq:C} and \eqref{eq:D} for a subgroup $F\le\Aut(B_{d,k})$ now read as follows.
\begin{equation}
  \forall \omega\in\Omega:\ p_{\omega}(F)\text{ is symmetric}
  \tag{C}
\end{equation}
\vspace{-0.6cm}
\begin{equation}
  \forall \omega\in\Omega:\ p_{\omega}|_{F}^{-1}(\id,\id)=\{\id\}
  \tag{D} 
\end{equation}

\subsubsection{The case $k=2$}\label{sec:ukf_examples_k=2}
We first consider the case $k=2$ which is all-encompassing in certain situations, see Theorem \ref{thm:ukf_rigid}. By the above, $\Aut(B_{d,2})$ is realized as follows: $\Aut(B_{d,2})=\{(a,(a_{\omega})_{\omega\in\Omega})\mid a\in\Sym(\Omega),\ \forall\omega\in\Omega:\ a_{\omega}\in\Sym(\Omega) \text{ and } a_{\omega}\omega=a\omega\}$.

Consider the map $\gamma:\Sym(\Omega)\to\Aut(B_{d,2})$, $a\mapsto (a,(a,\ldots,a))\in\Aut(B_{d,2})$, using the realization of $\Aut(B_{d,2})$ from above. For every $F\le\Sym(\Omega)$, the image
\begin{displaymath}
 \Gamma(F):=\image(\gamma|_{F})=\{(a,(a,\ldots,a))\mid a\in F\}\cong F
\end{displaymath}
is a subgroup of $\Aut(B_{d,2})$ which is isomorphic to $F$ and satisfies both \eqref{eq:C} and \eqref{eq:D}. The involutive compatibility cocycle is given by $\Gamma(F)\times\Omega\to\Gamma(F),\ (\gamma(a),\omega)\mapsto \gamma(a)$. Note that $\Gamma(F)\!\cong\! F$ implements the diagonal action $F\curvearrowright\Omega^{2}$ on $S(b,2)\cong\Omega^{(2)}\subset\Omega^{2}$.

We obtain $\mathrm{U}_{2}(\Gamma(F))\!=\!\{\alpha\in\Aut(T_{d})\mid \exists a\in F: \forall x\in V:\ \sigma_{1}(\alpha,x)=a\}=:\mathrm{D}(F)$, following the notation of \cite{BEW15}. Moreover, there is the following description of all subgroups $\smash{\widetilde{F}\le\Aut(B_{d,2})}$ with $\smash{\pi \widetilde{F}=F}$ that satisfy \eqref{eq:C} and contain $\Gamma(F)$.

\begin{proposition}\label{prop:u2f_split}
Let $F\le\Sym(\Omega)$. Given $K\le\prod_{\omega\in\Omega}F_{\omega}\cong\ker\pi\le\Aut(B_{d,2})$, there is $\smash{\widetilde{F}\le\Aut(B_{d,2})}$ satisfying \eqref{eq:C} and fitting into the split exact sequence
\begin{displaymath}
\xymatrix{
 1 \ar[r] & K \ar@{ >->}[r]^{\iota} & \widetilde{F} \ar@{-^>}@<.3ex>[r]^-{\pi} & F \ar@{-^>}@<.3ex>[l]^-{\gamma} \ar[r] & 1
}
\end{displaymath}
if and only if $K$ is preserved by the action $F\curvearrowright\prod_{\omega\in\Omega}F_{\omega}$, $\ a\cdot(a_{\omega})_{\omega}:=(aa_{a^{-1}\omega}a^{-1})_{\omega}$.
\end{proposition}

\begin{proof}
If there is a split exact sequence as above then $K\unlhd \widetilde{F}$ is invariant under conjugation by $\smash{\Gamma(F)\le \widetilde{F}}$, hence the assertion.

Conversely, if $K$ is invariant under the given action, then
\begin{displaymath}
  \smash{\widetilde{F}:=\{(a,(aa_{\omega})_{\omega})\mid a\in F,\ (a_{\omega})_{\omega}\in K\}}
\end{displaymath}
fits into the sequence: First, note that $\smash{\widetilde{F}}$ contains both $K$ and $\Gamma(F)$. It is also a subgroup of $\Aut(B_{d,2})$: For $\smash{(a,(aa_{\omega})_{\omega}),\ (b,(bb_{\omega})_{\omega})\in \widetilde{F}}$ we have
\begin{align*}
 (a,(aa_{\omega})_{\omega})\circ(b,(bb_{\omega})_{\omega})=(ab,(aa_{b\omega}bb_{\omega})_{\omega})=(ab,(ab\circ b^{-1}a_{b\omega}b\circ b_{\omega})_{\omega})\in \widetilde{F}
\end{align*}
by assumption. In particular, $\widetilde{F}=\langle \Gamma(F),K\rangle$. It suffices to check condition \eqref{eq:C} on these generators of $\smash{\widetilde{F}}$. As before, $\gamma(a)\in C(\gamma(a),\omega)$ for all $a\in F$ and $\omega\in\Omega$. Now let $k\in K$. Then $\gamma(\pr_{\omega}k)k^{-1}\in C(k,\omega)$ for all $\omega\in\Omega$.
\end{proof}

\begin{example}\label{ex:dihedral}
We show that for certain dihedral groups there are only four groups of the type given in Proposition~\ref{prop:u2f_split}: Set $F:=D_{p}\le\Sym(p)$ for some prime $p\ge 3$. Then $F_{\omega}\!\cong\!(\bbF_{2},+)$. Hence $\smash{U\!:=\!\prod_{\omega\in\Omega}F_{\omega}}$ is a $p$-dimensional vector space over $\bbF_{2}$ and the $F$-action on it permutes coordinates. When $2\in(\bbZ/p\bbZ)^{\ast}$ is primitive, there are only four $F$-invariant subspaces of $U$: The trivial subspace, the diagonal subspace $\langle(1,\ldots,1)\rangle$, the whole space, and $\smash{K:=\ker\sigma\cong\bbF_{2}^{(p-1)}}$ where $\smash{\sigma\!:U\to\bbF_{2}}$ is given by $\smash{(v_{1},\ldots,v_{p})\mapsto \sum_{i=1}^{p}v_{i}}$. Note that $K$ is $F$-invariant because the homomorphism $\sigma$ is. Conjecturally, there are infinitely many primes for which $2\in(\bbZ/p\bbZ)^{\ast}$ is primitive. The list starts with $3$, $5$, $11$, $13,\ldots$, see \cite[A001122]{Slo}.

Suppose that $W\le U$ is $F$-invariant. It suffices to show that $W$ contains $K$ as soon as $W\cap\ker\sigma$ contains a non-trivial element $w$. To see this, we show that the orbit of $w$ under the cyclic group $\langle\varrho\rangle=C_{p}\le D_{p}$ generates a $(p-1)$-dimensional subspace of $K$ which hence equals $K$: Indeed, the rank of the circulant matrix $C:=(w,\varrho w,\varrho^{2}w,\ldots,\varrho^{(p-1)}w)$ equals $p-\deg(\gcd(x^{p}-1,f(x)))$ where $f(x)\in\bbF_{2}[x]$ is the polynomial $f(x)=w_{p}x^{p-1}+\cdots+w_{2}x+w_{1}$, see e.g. \cite[Corollary 1]{Day60}. The polynomial $x^{p}-1\in\bbF_{2}[x]$ factors into the irreducibles $(x^{p-1}+x^{p-2}+\cdots+x+1)(x-1)$ by the assumption on $p$. Since $f$ has an even number of non-zero coefficients, we conclude that $\mathrm{rank}(C)=p-1$.
\end{example}

The following subgroups of $\Aut(B_{d,2})$ are of the type given in Proposition \ref{prop:u2f_split}. Let $F\le\Sym(\Omega)$ be transitive. Fix $\omega_{0}\in\Omega$, let $C\le Z(F_{\omega_{0}})$ and let $N\unlhd F_{\omega_{0}}$ be normal. Furthermore, fix elements $f_{\omega}\in F$ ($\omega\in\Omega$) satisfying $f_{\omega}(\omega_{0})=\omega$. We define
\begin{displaymath}
 \Delta(F,C):=\{(a,(a\circ f_{\omega}a_{0}f_{\omega}^{-1})_{\omega})\mid a\in F,\ a_{0}\in C\}\cong F\times C,\ \text{and}
\end{displaymath}
\vspace{-0.5cm}
\begin{displaymath}
 \Phi(F,N):=\{(a,(a\circ f_{\omega}a_{0}^{(\omega)}f_{\omega}^{-1})_{\omega})\mid a\in F,\ \forall \omega\in\Omega:\ a_{0}^{(\omega)}\in N\}\cong F\ltimes N^{d}.
\end{displaymath}
In the case of $\Delta(F,C)$ we have $K=\{(f_{\omega}a_{0}f_{\omega}^{-1})_{\omega}\mid a_{0}\in C\}$ whereas in the case of $\Phi(F,N)$ we have $\smash{K=\{(f_{\omega}a_{0}^{(\omega)}f_{\omega}^{-1})_{\omega}\mid \forall\omega\in\Omega: a_{0}^{(\omega)}\in N\}}$. In both cases, invariance under the action of $F$ is readily verified, as is condition \eqref{eq:D} for $\Delta(F,C)$.

\vspace{0.2cm}
The group $\Delta(F,F_{\omega_{0}})$ can be defined for non-abelian $F_{\omega_{0}}$ as well, namely
\begin{displaymath}
 \Delta(F):=\{(a,(f_{a\omega}f_{\omega}^{-1}\circ f_{\omega}a_{0}f_{\omega}^{-1})_{\omega})\mid a\in F, a_{0}\in F_{\omega_{0}}\}\cong F\times F_{\omega_{0}}.
\end{displaymath}
However, it need not contain $\Gamma(F)$. Note that $\Phi(F,N)$ does not depend on the choice of the elements $(f_{\omega})_{\omega\in\Omega}$ as $N$ is normal in $F_{\omega_{0}}$, whereas $\Delta(F,C)$ and $\Delta(F)$ may. However, any group of the form $\{(a,(z(a,\omega)\alpha_{\omega}(a_{0}))_{\omega})\mid a\in F,\ a_{0}\in F_{\omega_{0}}\}$, where $z$ is a compatibility cocycle of $F$ and $\alpha_{\omega}:F_{\omega_{0}}\to F_{\omega}$ $(\omega\in\Omega)$ are isomorphisms, which satisfies \eqref{eq:C} and in which $\{(a,(z(a,\omega))_{\omega})\mid a\in F\}$ and $\{(\id,(\alpha_{\omega}(a_{0}))_{\omega})\mid a_{0}\in F_{\omega_{0}}\}$ commute, will be referred to as $\Delta(F)$ in view of Corollary~\ref{cor:ukf_cd_iso}. 

The group $\Phi(F,F_{\omega_{0}})$ can be defined without assuming transitivity of $F$, namely
\begin{displaymath}
 \Phi(F):=\{(a,(a_{\omega})_{\omega})\mid a\in F,\ \forall \omega\in\Omega:\ a_{\omega}\in C_{F}(a,\omega)\}\cong F\ltimes\prod\nolimits_{\omega\in\Omega}F_{\omega}.
\end{displaymath}
We conclude that $\mathrm{U}_{2}(\Phi(F))=\mathrm{U}_{1}(F)$ for every $F\le \Sym(\Omega)$.

When $F\le\Sym(\Omega)$ preserves a partition $\calP:\Omega=\bigsqcup_{i\in I}\Omega_{i}$ of $\Omega$, we define
\begin{displaymath}
 \Phi(F,\calP):=\{(a,(a_{\omega})_{\omega})\mid a\in F,\ a_{\omega}\in C_{F}(a,\omega) \text{ constant w.r.t. $\calP$}\}\cong F\ltimes\prod\nolimits_{i\in I}F_{\Omega_{i}}.
\end{displaymath}
The group $\Phi(F,\calP)$ satisfies \eqref{eq:C} as well and features prominently in Section \ref{sec:non_trivial_qz}.

The following kind of $2$-local action generalises the sign construction in \cite{Rad17}. Let $F\le\Sym(\Omega)$ and $\rho:F\twoheadrightarrow A$ a homomorphism to an abelian group $A$. Define
\begin{align*}
 \Pi(F,\rho,\{1\})&:=\left\{(a,(a_{\omega})_{\omega})\in\Phi(F)\left|\ \prod\nolimits_{\omega\in\Omega}\rho(a_{\omega})=1\right.\right\},\ \text{and} \\
 \Pi(F,\rho,\{0,1\})&:=\left\{(a,(a_{\omega})_{\omega})\in\Phi(F)\left|\ \rho(a)\prod\nolimits_{\omega\in\Omega}\rho(a_{\omega})=1\right.\right\}.
\end{align*}
This construction is generalised to $k\geq 2$ in Section~\ref{sec:ukf_examples_general_case} where the third entry of~$\Pi$ is a set of radii over which the defining product is taken.

\begin{proposition}\label{prop:pif}
Let $F\le\Sym(\Omega)$ and $\rho:F\twoheadrightarrow A$ a homomorphism to an abelian group $A$. Let $\smash{\widetilde{F}\in\{\Pi(F,\rho,\{1\}), \Pi(F,\rho,\{0,1\})\}}$. If $\rho(F_{\omega})=A$ for all $\omega\in\Omega$ then $\smash{\pi\widetilde{F}=F}$ and $\smash{\widetilde{F}}$ satisfies \eqref{eq:C}.
\end{proposition}

\begin{proof}
As $C_{F}(a,\omega)\!=\!aF_{\omega}$, and $\rho(F_{\omega})\!=\!A$ for all $\omega\in\Omega$, an element $(a,(a_{\omega})_{\omega})\!\in\!\Phi(F)$ can be turned into an element of $\smash{\widetilde{F}}$ by changing $a_{\omega}$ for a single, arbitrary $\omega\!\in\!\Omega$. We conclude that $\smash{\pi\widetilde{F}=F}$ and that $\smash{\widetilde{F}}$ satisfies $\eqref{eq:C}$.
\end{proof}

\subsubsection{General case}\label{sec:ukf_examples_general_case}
We extend some constructions of Section \ref{sec:ukf_examples_k=2} to arbitrary~$k$. Given $\smash{F\le\Aut(B_{d,k})}$ satisfying \eqref{eq:C}, define the subgroup
\begin{displaymath}
 \Phi_{k}(F):=\{(\alpha,(\alpha_{\omega})_{\omega})\mid \alpha\in F,\ \forall\omega\in\Omega:\ \alpha_{\omega}\in C_{F}(\alpha,\omega)\}\le\Aut(B_{d,k+1}).
\end{displaymath}
Then $\Phi_{k}(F)$ inherits condition \eqref{eq:C} from $F$ and we obtain $\mathrm{U}_{k+1}(\Phi_{k}(F))=\mathrm{U}_{k}(F)$. Concerning the construction $\Gamma$ we have the following.

\begin{proposition}\label{prop:ukf_gamma_k}
Let $F\!\le\!\Aut(B_{d,k})$ satisfy \eqref{eq:C}. Then there exists a group $\Gamma_{k}(F)\!\le\!\Aut(B_{d,k+1})$ satisfying \eqref{eq:CD} such that $\pi_{k}:\Gamma_{k}(F)\to F$ is an isomorphism if and only if $F$ admits an involutive compatibility cocycle $z$.
\end{proposition}

\begin{proof}
If $F$ admits an involutive compatibility cocycle $z$, define
\begin{displaymath}
 \Gamma_{k}(F):=\{(\alpha,(z(\alpha,\omega))_{\omega})\mid\alpha\in F\}\le\Aut(B_{d,k+1}).
\end{displaymath}
Then $\gamma_{z}:F\to\Gamma_{k}(F),\ \alpha\mapsto(\alpha,(z(\alpha,\omega))_{\omega})$ is an isomorphism and the involutive compatibility cocycle of $\Gamma_{k}(F)$ is given by $\widetilde{z}:(\gamma_{z}(\alpha),\omega)\mapsto \gamma_{z}(z(\alpha,\omega))$. Conversely, if a group $\Gamma_{k}(F)$ with the asserted properties exists, set $\smash{z:(\alpha,\omega)\mapsto\pr_{\omega}\pi_{k}^{-1}\alpha}$.
\end{proof}

Let $F\le\Aut(B_{d,k})$ satisfy \eqref{eq:C} and let $l>k$. We set $\Gamma^{l}(F):=\Gamma_{l-1}\circ\cdots\circ\Gamma_{k}(F)$ for an implicit sequence of involutive compatibility cocycles. Similarly, we define $\Phi^{l}(F):=\Phi_{l-1}\circ\cdots\circ\Phi_{k}(F)$. Now, let $\smash{\widetilde{F}\le\Aut(B_{d,k})}$. Assume $\smash{F:=\pi\widetilde{F}\le\Sym(\Omega)}$ preserves a partition $\calP:\Omega=\bigsqcup_{i\in I}\Omega_{i}$ of $\Omega$. Define the group
\begin{displaymath}
 \Phi_{k}(\widetilde{F},\calP):=\{(\alpha,(\alpha_{\omega})_{\omega})\mid \alpha\in\widetilde{F},\ \alpha_{\omega}\in C_{\widetilde{F}}(\alpha,\omega) \text{ is constant w.r.t. $\calP$}\}.
\end{displaymath}
If $\smash{C_{\widetilde{F}}(\alpha,\Omega_{i})}$ is non-empty for all $\smash{\alpha\in\widetilde{F}}$ and $i\in I$ then $\smash{\Phi_{k}(\widetilde{F},\calP)}$ satisfies \eqref{eq:C}, and if $\smash{C_{\widetilde{F}}(\id,\Omega_{i})}$ is non-trivial for all $i\in I$ then $\smash{\Phi_{k}(\widetilde{F},\calP)}$ does not satisfy \eqref{eq:D}. 

The following statement generalizes Proposition \ref{prop:u2f_split}.

\begin{proposition}\label{prop:ukf_split}
Let $F\le\Aut(B_{d,k})$ satisfy \eqref{eq:C}. Suppose $F$ admits an involutive compatibility cocycle $z$. Given $K\le\Phi_{k}(F)\cap\ker(\pi_{k})$, there is $\smash{\widetilde{F}\le\Aut(B_{d,k+1})}$ satisfying \eqref{eq:C} and fitting into the split exact sequence
\begin{displaymath}
\xymatrix{
 1 \ar[r] & K \ar@{ >->}[r]^{\iota} & \widetilde{F} \ar@{-^>}@<.3ex>[r]^-{\pi} & F \ar@{-^>}@<.3ex>[l]^-{\gamma_{z}} \ar[r] & 1
}
\end{displaymath}
if and only if $\Gamma_{k}(F)$ normalizes $K$, and for all $k\in K$ and $\omega\in\Omega$ there is $k_{\omega}\in K$ such that $\pr_{\omega}k_{\omega}=z(\pr_{\omega}k,\omega)^{-1}$.
\end{proposition}

\begin{proof}
If there is a split exact sequence as above then $K\unlhd\widetilde{F}$ is invariant under conjugation by $\Gamma_{k}(F)$. Moreover, all elements of $\smash{\widetilde{F}}$ have the form $(\alpha,(z(\alpha,\omega)\alpha_{\omega})_{\omega})$ for some $\alpha\!\in\! F$ and $(\alpha_{\omega})_{\omega}\in K$. This implies the second assertion on $K$.

Conversely, if $K$ satisfies the assumptions, then
\begin{displaymath}
  \widetilde{F}:=\{(\alpha,(z(\alpha,\omega)\alpha_{\omega})_{\omega})\mid \alpha\in F,\ (\alpha_{\omega})_{\omega}\in K\}
\end{displaymath}
fits into the sequence: First, note that $\widetilde{F}$ contains both $K$ and $\Gamma_{k}(F)$. It is also a subgroup of $\Aut(B_{d,k+1})$: For $\smash{(\alpha,(z(\alpha,\omega)\alpha_{\omega})_{\omega}),\ (\beta,(z(\beta,\omega)\beta_{\omega})_{\omega})\in\widetilde{F}}$ we have
\begin{align*}
 (\alpha,(z(\alpha,\omega)\alpha_{\omega})_{\omega})&\circ(\beta,(z(\beta,\omega)\beta_{\omega})_{\omega})=(\alpha\beta,(z(\alpha,\beta\omega)\alpha_{\beta\omega}z(\beta,\omega)\beta_{\omega})_{\omega}) \\
 &=(\alpha\beta,(z(\alpha,\beta\omega)z(\beta,\omega)\circ z(\beta,\omega)^{-1}\alpha_{\beta\omega}z(\beta,\omega)\circ\beta_{\omega})_{\omega}) \\
 &=(\alpha\beta,(z(\alpha\beta,\omega)\alpha_{\omega}'\beta_{\omega})_{\omega})\in\widetilde{F}
\end{align*}
for some $(\alpha_{\omega}')_{\omega}\in K$ because $\Gamma_{k}(F)$ normalizes $K$. In particular, $\smash{\widetilde{F}=\langle \Gamma_{k}(F),K\rangle}$. We check condition \eqref{eq:C} on these generators. As before, $\gamma_{z}(z(\alpha,\omega))\in C(\gamma_{z}(\alpha),\omega)$ for all $\alpha\in F$ and $\omega\in\Omega$ because $z$ is involutive. Now, let $k\in K$. We then have $\gamma_{z}(\pr_{\omega}k)k_{\omega}\in C(k,\omega)$ for all $\omega\in\Omega$ by the assumption on $k_{\omega}$.
\end{proof}

In the split situation of Proposition \ref{prop:ukf_split} we also denote $\smash{\widetilde{F}}$ by $\Sigma_{k}(F,K)$. For instance, the group $\Pi(S_{3},\mathrm{sgn},\{1\})$ of Proposition \ref{prop:pif} satisfies \eqref{eq:C}, admits an involutive compatibility cocycle but does not satisfy \eqref{eq:D}, see Section \ref{sec:view_weiss}.

\vspace{0.2cm}
Now, let $F\le\Sym(\Omega)$ and $\rho:F\twoheadrightarrow A$ a homomorphism to an abelian group~$A$. Further, let $k\in\bbN$ and $X\subseteq\{0,\ldots,k-1\}$. Define
\begin{displaymath}
 \Pi^{k}(F,\rho,X):=\left\{\alpha\in\Phi^{k}(F)\left|\ \prod\nolimits_{r\in X}\prod\nolimits_{x\in S(b,r)}\rho(\sigma_{1}(\alpha,x))=1\right.\right\}.
\end{displaymath}

\begin{proposition}\label{prop:pikf}
Let $F\!\le\!\Sym(\Omega)$ and $\rho\!:\!F\twoheadrightarrow\! A$ a homomorphism to an abelian group $A$. Further, let $k\in\bbN$ and $X\subseteq\{0,\ldots,k-1\}$ non-empty and non-zero with $k\!-\!1\!\in\! X$. If $\rho(F_{\omega})\!=\!A$ for all $\omega\!\in\!\Omega$ then $\smash{\pi(\Pi^{k}(F,\rho,X))\!=\!F}$ and $\smash{\Pi^{k}(F,\rho,X)}$ has~\eqref{eq:C}.
\end{proposition}

\begin{proof}
As $C_{F}(a,\omega)=aF_{\omega}$, and $\rho(F_{\omega})=A$ for all $\omega\in\Omega$, an element $\alpha\in\Phi^{k}(F)$ can be turned into an element of $\Pi^{k}(F,\rho,X)$ by changing $\sigma_{1}(\alpha,x)$ for a single, arbitrary $x\in S(b,k-1)$. When $X$ is non-zero we conclude that $\pi(\Pi^{k}(F,\rho,X))\!=\!F$ and that $\Pi^{k}(F,\rho,X)$ satisfies $\eqref{eq:C}$.
\end{proof}

\subsubsection{A rigid case}\label{sec:ukf_examples_rigid}
For certain $F\le \Sym(\Omega)$ the groups $\Gamma(F)$, $\Delta(F)$ and $\Phi(F)$ already yield all possible $\smash{\mathrm{U}_{k}(\widetilde{F})}$ with $\smash{\pi\widetilde{F}=F}$. The main argument is based on Sections 3.4 and 3.5 of \cite{BM00a}. We first record the following lemma whose proof is due to M. Giudici by personal communication.

\begin{lemma}\label{lem:2tran_ext}
Let $F\le \Sym(\Omega)$ be $2$-transitive and $F_{\omega}$ $(\omega\in\Omega)$ simple non-abelian. Then every extension $\smash{\widetilde{F}}$ of $F_{\omega}$ ($\omega\in\Omega$) by $F$ is equivalent to $F_{\omega}\times F$.
\end{lemma}

\begin{proof}
Regarding $F_{\omega}$ as a normal subgroup of $\smash{\widetilde{F}}$, consider the conjugation map $\smash{\varphi:\widetilde{F}\to\Aut(F_{\omega})}$. We show that $\smash{K:=\ker\varphi=Z_{\widetilde{F}}(F_{\omega})\unlhd \widetilde{F}}$ complements $F_{\omega}$ in $\widetilde{F}$. Since $Z(F_{\omega})=\{\id\}$, we have $F_{\omega}\cap K=\{\id\}$. Hence $\smash{F_{\omega}K\unlhd \widetilde{F}}$. Next, consider $\smash{\widetilde{F}/(F_{\omega}K)\lesssim\mathrm{Out}(F_{\omega})}$. By the solution of Schreier' conjecture, $\mathrm{Out}(F_{\omega})$ is solvable. Since $\smash{\widetilde{F}/F_{\omega}\cong F}$ is not solvable we conclude $K\neq\{\id\}$. Now, by a theorem of Burnside, every $2$-transitive permutation group $F$ is either almost simple or affine type, see \cite[Theorem 4.1B and Section 4.8]{DM96}.

In the first case, $F$ is actually simple: Let $N\unlhd F$. Then $F_{\omega}\cap N\unlhd F_{\omega}$. Hence either $F_{\omega}\cap N=\{\id\}$ or $F_{\omega}\cap N=F_{\omega}$. Since $F$ is $2$-transitive and thereby primitive, every normal subgroup acts transitively. Hence, in the first case, $N$ is regular which contradicts $F$ being almost simple. Thus the second case holds and $N=NF_{\omega}=F$. Now $\smash{\widetilde{F}/F_{\omega}K}$ is a proper quotient of $F$ and therefore trivial. We conclude that $\smash{\widetilde{F}=F_{\omega}K\cong F_{\omega}\times K}$ and $\smash{K\cong \widetilde{F}/F_{\omega}\cong F}$.

In the second case, $\smash{F=F_{\omega}\ltimes C_{p}^{d}}$ for some $d\in\bbN$ and prime $p$. Given that $K$ is non-trivial and $K\cong F_{\omega}K/F_{\omega}\lhdsim F$, it contains the unique minimal normal subgroup $\smash{C_{p}^{d}}\lhdsim K\lhdsim F$. Since $F/C_{p}^{d}\cong F_{\omega}$ is non-abelian simple whereas the proper quotient $\smash{\widetilde{F}/F_{\omega}K}$ of $F$ is solvable, $\smash{K\neq C_{p}^{d}}$. But $\smash{F/C_{p}^{d}\cong F_{\omega}}$ is simple, so $\smash{\smash{F_{\omega}K=\widetilde{F}}}$.
\end{proof}

The following propositions are of independent interest and used in Theorem~\ref{thm:ukf_rigid} below. We introduce the following notation: Let $\smash{\widetilde{F}\!\le\!\Aut(B_{d,k})}$ and $\smash{K\le\widetilde{F}_{b_{w}}}$ for some $w=(\omega_{1},\ldots,\omega_{k-1})\in\Omega^{(k-1)}$. We set $\smash{\pi_{w}K:=\sigma_{1}(K,b_{w})\le\Sym(\Omega)_{\omega_{k-1}}}$.

\begin{proposition}\label{prop:ukf_subnormal}
Let $\smash{\widetilde{F}\le\Aut(B_{d,k})}$ satisfy \eqref{eq:C}. Suppose $\smash{F:=\pi\widetilde{F}}$ is transitive. Further, let $\omega\!\in\!\Omega$ and $w\!=\!(\omega_{1},\ldots,\omega_{k-1})\!\in\!\Omega^{(k-1)}$ with $\omega_{1}\!\neq\!\omega$. Then $\smash{\pi_{w}(\widetilde{F}_{b_{w}}\!\cap\ker\pi)}$ and $\smash{\pi_{w}\widetilde{F}_{T_{\omega}}}$ are subnormal in $F_{\omega_{k-1}}$ of depth at most $k-1$ and $k$ respectively.
\end{proposition}

\begin{proof}
We argue by induction on $k\!\ge\! 2$. For $k\!=\!2$, the assertion that $\smash{\pi_{w}(\widetilde{F}_{b_{w}}\cap\ker\pi)}$ is normal in $F_{\omega_{1}}$ is a consequence of condition \eqref{eq:C}. Now, suppose $\smash{\widetilde{F}\le\Aut(B_{d,k+1})}$ satisfies the assumptions, and let $\omega\in\Omega$ and $w=(\omega_{1},\ldots,\omega_{k})\in\Omega^{(k)}$ be such that $\omega_{1}\!\neq\!\omega$. Since $\smash{\widetilde{F}}$ satisfies \eqref{eq:C}, we have $\smash{\pr_{\omega_{1}}(\widetilde{F}_{b_{w}}\cap\ker\pi)\unlhd(\pi_{k}\widetilde{F})_{b_{w'}}\cap\ker\pi}$, where $w':=(\omega_{2},\ldots,\omega_{k-1})$ and the right hand side $\pi$ implicitly has domain $\smash{\pi_{k}\widetilde{F}}$. Hence
\begin{displaymath}
 \pi_{w}(\widetilde{F}_{b_{w}}\cap\ker\pi)=\pi_{w'}(\pr_{\omega_{1}}(\widetilde{F}_{b_{w}}\cap\ker\pi))\unlhd\pi_{w'}((\pi_{k}\widetilde{F})_{b_{w'}}\cap\ker\pi)\unlhd F_{\omega_{k-1}}
\end{displaymath}
by the induction hypothesis. The second assertion follows as $\smash{\widetilde{F}_{T_{\omega}}\unlhd\smash{\widetilde{F}_{b_{w}}\cap\ker\pi}}$.
\end{proof}

\begin{proposition}\label{prop:ukf_transitive}
Let $\smash{\widetilde{F}\le\Aut(B_{d,k})}$ satisfy \eqref{eq:C} but not \eqref{eq:D}. Suppose $\smash{F:=\pi\widetilde{F}}$ is transitive, and every non-trivial subnormal subgroup of $F_{\omega}$ ($\omega\in\Omega$) of depth at most $k-1$ is transitive on $\Omega\backslash\{\omega\}$. Then $\smash{\mathrm{U}_{k}(\widetilde{F})}$ is locally $k$-transitive.
\end{proposition}

\begin{proof}
We argue by induction on $k$. For $k=1$, the assertion follows from transitivity of $F$. Now, let $\smash{\widetilde{F}\le\Aut(B_{d,k+1})}$ satisfy \eqref{eq:C} but not \eqref{eq:D}. Then the same holds for $\smash{F^{(k)}:=\pi_{k}\widetilde{F}\le\Aut(B_{d,k})}$. Given $\smash{\widetilde{w},\widetilde{w}'\in\Omega^{(k)}}$, write $\smash{\widetilde{w}=(w,\omega)}$ and $\smash{\widetilde{w}'=(w',\omega')}$ where $w,w'\in\Omega^{(k-1)}$ and $\omega,\omega'\in\Omega$. By the induction hypothesis, the group $\smash{F^{(k)}}$ acts transitively on $S(b,k)$. Hence, using \eqref{eq:C}, there is $\smash{g\in\widetilde{F}}$ such that $gb_{w}=b_{w'}$. As $\smash{\widetilde{F}}$ does not satisfy \eqref{eq:D} said transitivity further implies that $\smash{\pi_{w'}(\widetilde{F}_{b_{w'}}\cap\ker\pi))}$ is non-trivial. By Proposition \ref{prop:ukf_subnormal}, it is also subnormal of depth at most $k-1$ in $F_{\omega'}$ and thus transitive. Hence there is $\smash{g'\in\widetilde{F}_{b_{w'}}}$ with $g'gb_{\widetilde{w}}=b_{\widetilde{w}'}$.
\end{proof}

The following theorem is closely related to \cite[Proposition 3.3.1]{BM00a}.

\begin{theorem}\label{thm:ukf_rigid}
\hspace{-0.05cm}Let $F\!\le\!\Sym(\Omega)$ be $2$-transitive and $F_{\omega}$ $(\omega\!\in\!\Omega)$ simple non-abelian. Further, let $\smash{\widetilde{F}\le\Aut(B_{d,k})}$ with $\smash{\pi\widetilde{F}=F}$ satisfy \eqref{eq:C}. Then $\smash{\mathrm{U}_{k}(\widetilde{F})}$ equals either
\begin{displaymath}
 \mathrm{U}_{2}(\Gamma_{1}(F)),\quad \mathrm{U}_{2}(\Delta(F))\quad\text{or}\quad\mathrm{U}_{2}(\Phi(F))=\mathrm{U}_{1}(F).
\end{displaymath}
\end{theorem}

\begin{proof}
Since $\mathrm{U}_{1}(F)=\mathrm{U}_{2}(\Phi(F))$, we may assume $k\ge 2$. Given that $\widetilde{F}\le\Aut(B_{d,k})$ satisfies \eqref{eq:C} so does the restriction $\smash{F^{(2)}:=\pi_{2}\widetilde{F}\le\Phi(F)\le\Aut(B_{d,2})}$. Consider the projection $\smash{\pi:F^{(2)}\twoheadrightarrow F}$. We have $\ker\pi\le\prod_{\omega\in\Omega}F_{\omega}$ and $\pr_{\omega}\ker\pi\unlhd F_{\omega}$ for all $\omega\in\Omega$ by Proposition \ref{prop:ukf_subnormal}. Since $F_{\omega}$ is simple, $\smash{\ker\pi\unlhd F^{(2)}}$ and $F$ is transitive this implies that either $\pr_{\omega}\ker\pi=\{\id\}$ for all $\omega\in\Omega$ or $\pr_{\omega}\ker\pi=F_{\omega}$ for all $\omega\in\Omega$.

In the first case, $\smash{\pi:F^{(2)}\to F}$ is an isomorphism. Hence $\smash{F^{(2)}}$ satisfies~\eqref{eq:CD} and $\smash{\mathrm{U}_{k}(\widetilde{F})\!=\!\mathrm{U}_{2}(\Gamma_{1}(F))}$ for an involutive compatibility cocycle of $F$ by Proposition~\ref{prop:ukf_gamma_k}.

In the second case, fix $\omega_{0}\in\Omega$. We have $\ker\pi\le\prod_{\omega\in\Omega}F_{\omega}\cong F_{\omega_{0}}^{d}$ by transitivity of $F$. Since $F_{\omega_{0}}$ is simple non-abelian, \cite[Lemma 2.3]{Rad17} implies that the group $\ker\pi$ is a product of subdiagonals preserved by the primitive action of $F$ on the index set of $F_{\omega_{0}}^{d}$. Hence, either there is just one block and $\ker\pi\cong F_{\omega_{0}}$ has the form $\{(\id,(\alpha_{\omega}(a_{0}))_{\omega})\mid a_{0}\in F_{\omega_{0}}\}$ for some isomorphisms $\alpha_{\omega}:F_{\omega_{0}}\to F_{\omega}$, or all blocks are singletons and $\smash{\ker\pi=\prod_{\omega\in\Omega}F_{\omega}\cong F_{\omega_{0}}^{d}}$. In the first case, there is a compatibility cocycle $z$ of $F$ such that $F\cong\{(a,(z(a,\omega))_{\omega})\mid a\in F\}\le F^{(2)}$ commutes with $\ker\pi\le F^{(2)}$ by Lemma \ref{lem:2tran_ext}. Thus $F^{(2)}=\{a,(z(a,\omega)\alpha_{\omega}(a_{0}))_{\omega}\mid a\in F,\ a_{0}\in F_{\omega_{0}}\}$. In particular, $F^{(2)}$ satisfies \eqref{eq:CD}. Hence $\smash{\mathrm{U}_{k}(\widetilde{F})=\mathrm{U}_{2}(\Delta(F))}$. 

When $\smash{\ker\pi\cong F_{\omega_{0}}^{d}}$, we have $\smash{\mathrm{U}_{k}(\widetilde{F})=\mathrm{U}_{1}(F)}$ by \cite[Proposition 3.3.1]{BM00a}.
\end{proof}

If $F$ does not have simple point stabilizers or preserves a non-trivial partition, more universal groups are given by $\mathrm{U}_{2}(\Phi(F,N))$ and $\mathrm{U}_{2}(\Phi(F,\calP))$, see Section~\ref{sec:ukf_examples_k=2}. When $F$ is $2$-transitive and has abelian point stabilizers then $F\!\cong\!\AGL(1,q)$ 
for some prime power $q$ by \cite{KKP90}. Hence point stabilizers in $F$ are isomorphic to $\bbF_{q}^{\ast}$ and simple if and only if $q-1$ is a Mersenne prime. For any value of $q$, the projection $\rho:\AGL(1,q)\to \bbF_{q}^{\ast}$ satisfies the assumptions of Proposition~\ref{prop:pikf} and so the groups $\smash{\mathrm{U}_{k}(\Pi^{k}(\AGL(1,q),\rho,X))}$ provide further examples.
The following question remains.

\begin{question}
Let $F\!\le\!\Sym(\Omega)$ be primitive and $F_{\omega}$ $(\omega\in\Omega)$ simple non-abelian. Is there $\smash{\widetilde{F}\!\le\!\Aut(B_{d,k})}$ with \eqref{eq:C} and $\pi\widetilde{F}\!=\!F$ other than $\Gamma^{k}(F)$, $\Delta(F)$ and $\Phi^{k}(F)$?
\end{question}

\newpage
\subsection{Universality}
The constructed groups $\mathrm{U}_{k}(F)$ are universal in the sense of the following maximality statement, which should be compared to Proposition \ref{prop:uf_universal}.

\begin{theorem}\label{thm:ukf_universal}
Let $H\le\Aut(T_{d})$ be locally transitive and contain an involutive inversion. Then there is a labelling $l$ of $T_{d}$ such that
\begin{displaymath}
 \mathrm{U}_{1}^{(l)}(F^{(1)})\ge\mathrm{U}_{2}^{(l)}(F^{(2)})\ge\cdots\ge\mathrm{U}_{k}^{(l)}(F^{(k)})\ge\cdots\ge H\ge\mathrm{U}_{1}^{(l)}(\{\id\})
\end{displaymath}
where $F^{(k)}\le\Aut(B_{d,k})$ is action isomorphic to the $k$-local action of $H$.
\end{theorem}

\begin{proof}
First, we construct a labelling $l$ of $T_{d}$ such that $H\ge\mathrm{U}_{1}^{(l)}(\{\id\})$: Fix $x\in V$ and choose a bijection $l_{x}:E(x)\to\Omega$. By the assumptions, there is an involutive inversion $\iota_{\omega}\in H$ of the edge $(x,x_{\omega})\in E$ for every $\omega\in\Omega$. Using these inversions, we define the announced labelling inductively: Set $l|_{E(x)}:=l_{x}$ and assume that $l$ is defined on $E(x,n)$. For $e\in E(x,n+1)\backslash E(x,n)$ put $l(e):=l(\iota_{\omega}(e))$ if $x_{\omega}$ is part of the unique reduced path from $x$ to $o(e)$. Since the $\iota_{\omega}$ $(\omega\in\Omega)$ have order $2$, we obtain $\sigma_{1}(\iota_{\omega},y)=\id$ for all $\omega\in\Omega$ and $y\in V$. Therefore, $\smash{\langle\{\iota_{\omega}\mid\omega\in\Omega\}\rangle=\mathrm{U}_{1}^{(l)}(\{\id\})\le H}$, following the proof of Lemma \ref{lem:uid_fin_gen}.

Now, let $h\in H$ and $y\in V$. Further, let $(x,x_{1},\ldots,x_{n},y)$ and $(x,x_{1}',\ldots,x_{m}',h(y))$ be the unique reduced paths from $x$ to $y$ and $h(y)$ respectively. Since $\smash{\mathrm{U}_{1}^{(l)}(\{\id\})}\le H$, the group $H$ contains the unique label-respecting inversion $\iota_{e}$ of every edge $e\in E$. We therefore have
\begin{displaymath}
 s:=\iota_{(x_{1}',x)}^{-1}\cdots\iota_{(x_{m}',x_{m-1}')}^{-1}\iota_{(h(y),x_{m}')}^{-1}\circ h\circ \iota_{(y,x_{n})}\cdots\iota_{(x_{2},x_{1})}\iota_{(x_{1},x)}\in H.
\end{displaymath}
Also, $s$ stabilizes $x$. The cocycle identity implies for every $k\in\bbN$:
\begin{displaymath}
 \sigma_{k}(h,y)=\sigma_{k}(\iota_{(h(y),x_{m}')}\cdots\iota_{(x_{1}',x)}\circ s\circ \iota_{(x_{1},x)}^{-1}\cdots\iota_{(y,x_{n})}^{-1},y)=\sigma_{k}(s,x)\in F^{(k)}.
\end{displaymath}
where $F^{(k)}\le\Aut(B_{d,k})$ is defined by $l_{x}^{k}\circ H_{x}|_{B(x,k)}\circ (l_{x}^{k})^{-1}$.
\end{proof}

\begin{remark}\label{rem:inv_necessary}
Retain the notation of Theorem \ref{thm:ukf_universal}. By Proposition \ref{prop:uf_universal}, there is a labelling $l$ of $T_{d}$ such that $\smash{\mathrm{U}_{1}^{(l)}(F^{(1)})\ge H}$ regardless of the minimal order of an inversion in $H$. This labelling may be distinct from the one of Theorem \ref{thm:ukf_universal} which fails without assuming the existence of an involutive inversion: For example, a vertex-stabilizer of the group $G_{2}^{1}$ of Example \ref{ex:s3} below is action isomorphic to $\Gamma(S_{3})$ but $\smash{G_{2}^{1}\not\le\mathrm{U}_{2}^{(l)}(\Gamma(S_{3}))}$ for any labelling $l$ because $(G_{2}^{1})_{\{b,b_{\omega}\}}\cong\bbZ/4\bbZ$ whereas
\begin{displaymath}
  \mathrm{U}_{2}^{(l)}(\Gamma(S_{3}))_{\{b,b_{\omega}\}}\cong\Gamma(S_{3})_{(b,b_{\omega})}\rtimes\bbZ\!/2\bbZ\cong\bbZ\!/2\bbZ\times\bbZ\!/2\bbZ
\end{displaymath}
by Proposition \ref{prop:ukf_structure}.
\end{remark}


We complement Theorem \ref{thm:ukf_universal} with the following criterion for certain subgroups of $\Aut(T_{d})$ to contain an involutive inversions.

\begin{proposition}\label{prop:crit_inv_edge}
Let $H\le\Aut(T_{d})$ be locally transitive with odd order point stabilizers. If $H$ contains a finite order inversion then it contains an involutive one.
\end{proposition}

\begin{proof}
Let $\iota\in H$ be a finite order inversion of an edge $e\in E$ and $\mathrm{ord}(\iota)=2^{k}\cdot m$ for some odd $m\in\bbN$ and some $k\in\bbN$. It suffices to show that $k=1$, in which case $\iota^{m}$ is an involutive inversion. Suppose $k\ge 1$. Then $\smash{\iota^{2^{k-1}\cdot m}}$ is non-trivial and fixes the edge $e$. Because point stabilizers in the local action of $H$ have odd order, it follows that $\smash{(\iota^{2^{k-1}\cdot m})^{2}}$ is non-trivial as well, but $\smash{(\iota^{2^{k-1}\cdot m})^{2}=\iota^{\mathrm{ord}(\iota)}=\id}$.
\end{proof}

For example, Proposition \ref{prop:crit_inv_edge} applies when $H$ is discrete and vertex-transitive: Combined with local transitivity this implies the existence of a finite order inversion.

We remark that primitive permutation groups with odd order point stabilizers were classified in \cite{LS91}. For instance, they include $\mathrm{PSL}(2,q)\curvearrowright\mathrm{P}^{1}(\bbF_{q})$ for any prime power $q$ that satisfies $q\equiv 3\ \mathrm{mod}\ 4$.

\subsection{A Bipartite Version}\label{sec:universal_groups_bipartite}

In this section, we introduce a bipartite version of the universal groups developed in Section~\ref{sec:universal_groups} which plays a critical role in the proof of Theorem \ref{thm:local_action_qz_sharp}\ref{item:qz_sharp_k_transitive}\ref{item:qz_sharp_quasiprimitive_hyp} below. Retain the notation of Section \ref{sec:universal_groups}. In particular, let $T_{d}=(V,E)$ denote the $d$-regular tree. Fix a regular bipartition $V=V_{1}\sqcup V_{2}$ of $V$.

\subsubsection{Definition and Basic Properties}
The groups to be defined are subgroups of $\tensor[^{+}]{\Aut(T_{d})}{}\le\Aut(T_{d})$, the maximal subgroup of $\Aut(T_{d})$ preserving the bipartition $V=V_{1}\sqcup V_{2}$. Alternatively, it can be described as the subgroup generated by all point stabilizers, or all edge-stabilizers.

\begin{definition}
Let $F\le\Aut(B_{d,2k})$ and $l$ be a labelling of $T_{d}$. Define
\begin{displaymath}
 \mathrm{BU}_{2k}^{(l)}(F):=\{\alpha\in\!\!\tensor[^{+}]{\Aut(T_{d})}{}\mid \forall v\in V_{1}:\ \sigma_{2k}(\alpha,v)\in F\}.
\end{displaymath}
\end{definition}

Note that $\mathrm{BU}_{2k}^{(l)}(F)$ is a subgroup of $\tensor[^{+}]{\Aut(T_{d})}{}$ thanks to Lemma \ref{lem:cocycle} and the assumption that it is a subset of $\tensor[^{+}]{\Aut(T_{d})}{}$. Further, Proposition \ref{prop:ukf_labelling} carries over to the groups $\smash{\mathrm{BU}_{2k}^{(l)}(F)}$. We shall therefore omit the reference to an explicit labelling in the following. Also, we recover the following basic properties.

\begin{proposition}\label{prop:bukf_basic_properties}
Let $F\le\Aut(B_{d,2k})$. The group $\mathrm{BU}_{2k}(F)$ is
\begin{enumerate}[(i)]
 \item\label{item:bukf_closed} closed in $\Aut(T_{d})$
 \item\label{item:bukf_vertex_transitive} transitive on both $V_{1}$ and $V_{2}$, and
 \item\label{item:bukf_comp_gen} compactly generated.
\end{enumerate}
\end{proposition}

Parts \ref{item:bukf_closed} and \ref{item:bukf_vertex_transitive} are proven as their analogues in Proposition \ref{prop:ukf_basic_properties} whereas part \ref{item:bukf_comp_gen} relies on part \ref{item:bukf_vertex_transitive} and the subsequent analogue of Lemma \ref{lem:uid_fin_gen}, for which we introduce the following notation: Given $x\in V$ and $\smash{w\in\Omega^{(2k)}}$, let $\smash{t_{w}^{(x)}\in\mathrm{BU}_{2}(\{\id\})}$ denote the unique label-respecting translation with $\smash{t_{w}^{(x)}(x)=x_{w}}$. Given an element $\smash{w=(\omega_{1},\ldots,\omega_{2k})\in\Omega^{(2k)}}$, we set $\smash{\overline{w}:=(\omega_{2k},\ldots,\omega_{1})\in\Omega^{(2k)}}$. Then $\smash{(t_{w}^{(x)})^{-1}=t_{\overline{w}}^{(x)}}$ and if $\smash{\Omega_{+}^{(2k)}\subseteq\Omega^{(2k)}}$ is such that for every $\smash{w\in\Omega^{(2k)}}$ exactly one of $\{w,\overline{w}\}$ belongs to $\smash{\Omega_{+}^{(2k)}}$, then $\smash{\Omega_{+}^{(2k)}=\Omega_{+}^{(2k)}\sqcup\overline{\Omega}_{+}^{(2k)}}$ where $\smash{\overline{\Omega}_{+}^{(2k)}:=\{\overline{w}\mid w\in\Omega_{+}^{(2k)}\}}$.

\begin{lemma}\label{lem:bu_comp_gen}
Let $x\in V_{1}$. Then $\smash{\mathrm{BU}_{2}(\{\id\})=\langle\{t_{w}^{(x)}\mid w\in\Omega^{(2)}\}\rangle\cong F_{\Omega_{+}^{(2)}}}$, the free group on the set $\smash{\Omega_{+}^{(2)}}$.
\end{lemma}

\begin{proof}
Every element of $\mathrm{BU}_{2k}(\{\id\})$ is uniquely determined by its image on $x$. To see that $\smash{\mathrm{BU}_{2}(\{\id\})=\langle\{t_{w}^{(x)}\mid w\in\Omega^{(2)}\}\rangle}$ it hence suffices to show that $\smash{\{t_{w}^{(x)}\!\mid\! w\!\in\!\Omega^{(2)}\}}$ is transitive on $V_{1}$. Indeed, let $y\in V_{1}$. Then $y=x_{w}$ for some $w\in\Omega^{(2k)}$ where $2k=d(x,y)$. Write $w=(w_{1},\ldots,w_{k})\in(\Omega^{(2)})^{k}$. Then $\smash{t_{w_{1}}^{(x)}\circ\cdots\circ t_{w_{k}}^{(x)}=t_{w}^{(x)}}$ as every $\smash{t_{w_{i}}^{(x)}}$ $(i\in\{1,\ldots,k\})$ is label-respecting. Hence $\smash{t_{w_{1}}^{(x)}\circ\dots\circ t_{w_{k}}^{(x)}(x)=y}$ and that
\begin{displaymath}
 \langle\{t_{w}^{(x)}\!\mid\! w\in\Omega^{(2)}\}\rangle\to F_{\Omega_{+}^{(2)}},\ \begin{cases}t_{w}^{(x)}\mapsto w & w\in\Omega_{+}^{(2)} \\ t_{w}^{(x)}\mapsto \overline{w}^{-1} & w\not\in\Omega_{+}^{(2)}\end{cases}
\end{displaymath}
yields a well-defined isomorphism.
\end{proof}

\subsubsection{Compatibility and Discreteness}

In order to describe the compatibility and the discreteness condition in the bipartite setting, we first introduce a suitable realization of $\Aut(B_{d,2k})$ $(k\in\bbN)$, similar to the one at the beginning of Section~\ref{sec:ukf_examples}. Let $\Aut(B_{d,1})\cong\Sym(\Omega)$ and $\Aut(B_{d,2})$ be as before. For $k\ge 2$, we inductively identify $\Aut(B_{d,2k})$ with its image under
\begin{align*}
 \Aut(B_{d,2k})&\to\Aut(B_{d,2(k-1)})\ltimes\prod_{w\in\Omega^{(2)}}\Aut(B_{d,2(k-1)}) \\
 \alpha&\mapsto(\sigma_{2(k-1)}(\alpha,b),(\sigma_{2(k-1)}(\alpha,b_{w}))_{w}))
\end{align*}
where $\Aut(B_{d,2(k-1)})$ acts on $\Omega^{(2)}$ by permuting the factors according to its action on $\smash{S(b,2)\cong\Omega^{(2)}}$. In addition, consider the map $\pr_{w}:\Aut(B_{d,2k})\to\Aut(B_{d,2(k-1)})$, $\alpha\mapsto\sigma_{2(k-1)}(\alpha,b_{w})$ for every $\smash{w\in\Omega^{(2)}}$, as well as
\begin{displaymath}
 p_{w}=(\pi_{2(k-1)},\pr_{w}):\Aut(B_{d,2k})\to\Aut(B_{d,2(k-1)})\times\Aut(B_{d,2(k-1)})
\end{displaymath}
For $k\ge 2$, conditions (C) and (D) for $F\le\Aut(B_{d,2k})$ now read as follows.
\begin{equation}
  \forall\alpha\in F\ \forall w\in\Omega^{(2)}\ \exists \alpha_{w}\in F:\ \pi_{2(k-1)}(\alpha_{w})=\pr_{w}(\alpha),\ \pr_{\overline{w}}(\alpha_{w})=\pi_{2(k-1)}(\alpha)
  \tag{C}
  \label{eq:C_BU}
\end{equation}
\vspace{-0.6cm}
\begin{equation}
  \forall w\in\Omega^{(2)}:\ p_{w}|_{F}^{-1}(\id,\id)=\{\id\}
  \tag{D}
  \label{eq:D_BU}
\end{equation}
For $k=1$ we have, using the maps $\pr_{\omega}$ ($\omega\in\Omega$) as in Section \ref{sec:ukf_examples},
\begin{equation}
  \forall\alpha\in F\ \forall w=(\omega_{1},\omega_{2})\in\Omega^{(2)}\ \exists \alpha_{w}\in F:\ \pr_{\omega_{2}}(\alpha_{w})=\pr_{\omega_{1}}(\alpha).
  \tag{C}
  \label{eq:C_BU_1}
\end{equation}
\vspace{-0.6cm}
\begin{equation}
  \forall\omega\in\Omega:\ \pr_{\omega}|_{F}^{-1}(\id)=\{\id\}.
  \tag{D}
  \label{eq:D_BU_1}
\end{equation}
Analogues of Proposition \ref{prop:ukf_discrete} are proven using the discreteness conditions (D) above. We do not introduce new notation for any of the above as the context always implies which condition is to be considered. The definition of the compatibility sets $C_{F}(\alpha,S)$ for $F\le\Aut(B_{d,2k})$ and $\smash{S\subseteq\Omega^{(2)}}$ carries over from Section \ref{sec:comp_disc} in a straightforward fashion.

\subsubsection{Examples}\label{sec:bukf_examples}
Let $F\le\Sym(\Omega)$. Then the group $\Gamma(F)\le\Aut(B_{d,2})$ introduced in Section \ref{sec:ukf_examples_k=2} satisfies conditions \eqref{eq:C_BU_1} and \eqref{eq:D_BU_1} for the case $k=1$ above, and we have $\mathrm{BU}_{2}(\Gamma(F))=\mathrm{U}_{2}(\Gamma(F))\cap\!\tensor[^{+}]{\Aut(T_{d})}{}$.

Similarly, the group $\Phi(F)\le\Aut(B_{d,2})$ satisfies condition \eqref{eq:C_BU_1} for the case $k=1$ as $\Gamma(F)\le\Phi(F)$, and we have $\mathrm{BU}_{2}(\Phi(F))=\mathrm{U}_{1}(F)\cap\!\tensor[^{+}]{\Aut(T_{d})}{}$.

\vspace{0.2cm}
The following example gives an analogue of the groups $\Phi(F,N)$. Notice, however, that in this case the second argument is a subgroup of $F$ rather than $F_{\omega_{0}}$ and need not be normal, as the $1$-local action at vertices in $V_{1}$ and $V_{2}$ need not be the same.

\begin{example}
Let $F'\le F\le\Sym(\Omega)$. Then
\begin{displaymath}
 \mathrm{B}\Phi(F,F'):=\{(a,(a_{\omega})_{\omega\in\Omega})\mid a\in F,\ \forall\omega\in\Omega:\ a_{\omega}\in C_{F}(a,\omega)\cap F'\}\le\Aut(B_{d,2})
\end{displaymath}
satisfies condition \eqref{eq:C_BU_1} for the case $k=1$ above given that $\Gamma(F')\le\mathrm{B}\Phi(F,F')$. If $F'\backslash\Omega=F\backslash\Omega$, the $1$-local action of $\mathrm{B}\Phi(F,F')$ at vertices in $V_{1}$ is indeed $F$, whereas it is $F'^{+}$ at vertices in $V_{2}$. This construction is similar to $\calU_{\calL}(M,N)$ in \cite{Smi17}.
\end{example}

The next example constitutes the base case in Section \ref{sec:qz_sharp_proof_iv_b} below.

\begin{example}\label{ex:bukf_orbit_self_compatible}
Let $F\le\Sym(\Omega)$. Suppose $F$ preserves a non-trivial partition $\calP:\Omega=\bigsqcup_{i\in I}\Omega_{i}$ of $\Omega$. Then
\begin{displaymath}
 \Omega^{(2)}_{0}:=\{(\omega_{1},\omega_{2})\mid\exists i\in I:\ \omega_{1},\omega_{2}\in\Omega_{i}\}\subseteq\Omega^{(2)}.
\end{displaymath}
is preserved by the action of $\Phi(F)$ on $S(b,2)\cong\Omega^{(2)}$: Let $\alpha\!=\!(a,(a_{\omega})_{\omega})\!\in\!\Phi(F)$ and $\smash{(\omega_{1},\omega_{2})\in\Omega^{(2)}_{0}}$. Then $\smash{\alpha(\omega_{1},\omega_{2})=(a\omega_{1},a_{\omega_{1}}\omega_{2})=(a_{\omega_{1}}\omega_{1},a_{\omega_{1}}\omega_{1})\in\Omega^{(2)}_{0}}$. Also, note that if $\smash{w=(\omega_{1},\omega_{2})\in\Omega^{(2)}_{0}}$ then so is $\overline{w}=(\omega_{2},\omega_{1})$.

The subgroup of $\Phi(F)$ consisting of those elements which are self-compatible in all directions from $\smash{\Omega^{(2)}_{0}}$ is precisely given by
\begin{displaymath}
 F^{(2)}:=\{(a,(a_{\omega})_{\omega})\mid a\in F, a_{\omega}\in C_{F}(a,\omega) \text{ constant w.r.t. $\calP$}\}.
\end{displaymath}
in view of condition \eqref{eq:C_BU_1} for the case $k=1$ above.
\end{example}

Suppose now that $F\le\Aut(B_{d,2k})$ satisfies \eqref{eq:C_BU}. Analogous to the group $\Phi_{k}(F)$ of Section \ref{sec:ukf_examples_general_case}, we define
\begin{displaymath}
\mathrm{B}\Phi_{2k}(F)\!:=\!\{(\alpha,(\alpha_{w})_{w\in\Omega^{(2)}})\!\mid\! \alpha\!\in\! F,\ \forall w\!\in\!\Omega^{(2)}\!:\ \alpha_{w}\!\in\! C_{F}(\alpha,w)\}\le\Aut(B_{d,2(k+1)}).
\end{displaymath}
Then $\mathrm{B}\Phi_{2k}(F)\le\Aut(B_{d,2(k+1)})$ satisfies \eqref{eq:C_BU} and $\mathrm{BU}_{2(k+1)}(\mathrm{B}\Phi_{2k}(F))=\mathrm{BU}_{2k}(F)$. Given $l>k$, we also set $\mathrm{B}\Phi^{2l}(F):=\mathrm{B}\Phi_{2(l-1)}\circ\cdots\circ\mathrm{B}\Phi_{2k}(F)$, c.f. Section \ref{sec:ukf_examples_general_case}.

\newpage
\section{Applications}\label{sec:applications}

In this section, we give three applications of the framework of universal groups. First, we characterize the automorphism types which the quasi-center of a non-discrete subgroup of $\Aut(T_{d})$ may feature in terms of the group's local action, and see that Burger--Mozes theory does not extend to the transitive case. Second, we give an algebraic characterization of the $(P_{k})$-closures of locally transitive subgroups of $\Aut(T_{d})$ containing an involutive inversion, and thereby partially answer two question by Banks--Elder--Willis. Third, we offer a new view on the Weiss conjecture.

\subsection{Groups Acting on Trees With Non-Trivial Quasi-Center}\label{sec:non_trivial_qz}

By Proposition \ref{prop:bm_1.2.1}\ref{item:bm_1.2.1_2}, a non-discrete, locally semiprimitive subgroup of $\Aut(T_{d})$ does not contain any non-trivial quasi-central edge-fixating elements. We complete this fact to the following local-to-global-type characterization of quasi-central elements.

\begin{theorem}\label{thm:local_action_qz}
Let $H\le\Aut(T_{d})$ be non-discrete. If $H$ is locally
\begin{enumerate}[(i)]
 \item\label{item:qz_transitive} transitive then $\mathrm{QZ}(H)$ contains no inversion.
 \item\label{item:qz_semiprimitive} semiprimitive then $\mathrm{QZ}(H)$ contains no non-trivial edge-fixating element.
 \item\label{item:qz_quasiprimitive} quasiprimitive then $\mathrm{QZ}(H)$ contains no non-trivial elliptic element.
 \item\label{item:qz_k_transitive} $k$-transitive $(k\in\bbN)$ then $\mathrm{QZ}(H)$ contains no hyperbolic element of length~$k$.
\end{enumerate}
\end{theorem}

\begin{theorem}\label{thm:local_action_qz_sharp}
There is $d\in\bbN_{\ge 3}$ and a closed, non-discrete, compactly generated subgroup of $\Aut(T_{d})$ which is locally
\begin{enumerate}[(i)]
 \item\label{item:qz_sharp_intransitive} intransitive and contains a quasi-central inversion.
 \item\label{item:qz_sharp_transitive} transitive and contains a non-trivial quasi-central edge-fixating element.
 \item\label{item:qz_sharp_semiprimitive} semiprimitive and contains a non-trivial quasi-central elliptic element.
 \item\label{item:qz_sharp_k_transitive} \begin{enumerate}[(a),itemindent=-0.4cm]
        \item\label{item:qz_sharp_intransitive_hyp} \hspace{-0.15cm} intransitive and contains a quasi-central hyperbolic element of length~$1$.
        \item\label{item:qz_sharp_quasiprimitive_hyp} \hspace{-0.15cm} quasiprimitive and contains a quasi-central hyperbolic element of length~$2$.
       \end{enumerate}
\end{enumerate}
\end{theorem}


\begin{proof}
(Theorem \ref{thm:local_action_qz}). Fix a labelling of $T_{d}$ and let $H\le\Aut(T_{d})$ be non-discrete.

For \ref{item:qz_transitive}, suppose $\iota\in\mathrm{QZ}(H)$ inverts $(x,x_{\omega})\in E$. Since $H$ is locally transitive and $\mathrm{QZ}(H)\unlhd H$, there is an inversion $\iota_{\omega}\in\mathrm{QZ}(H)$ of $(x,x_{\omega})\in E$ for all $\omega\in\Omega$. By definition, the centralizer of $\iota_{\omega}$ in $H$ is open for all $\omega\in\Omega$. Hence, using non-discreteness of $H$, there is $n\in\bbN$ such that $H_{B(x,n)}$ commutes with $\iota_{\omega}$ for all $\omega\in\Omega$ and $H_{B(x,n+1)}\neq\{\id\}$. However, $H_{B(x,n)}=\iota_{\omega}H_{B(x,n)}\iota_{\omega}^{-1}=H_{B(x_{\omega},n)}$ for all $\omega\in\Omega$, that is $H_{B(x,n+1)}\subseteq H_{B(x,n)}$ in contradiction to the above.

Part \ref{item:qz_semiprimitive} is Proposition \ref{prop:bm_1.2.1}\ref{item:bm_1.2.1_2} and part \ref{item:qz_quasiprimitive} is \cite[Proposition 1.2.1 (ii)]{BM00a}. Here, the closedness assumption is unnecessary.

For part \ref{item:qz_k_transitive}, suppose $\tau\in\mathrm{QZ}(H)$ is a translation of length $k$ which maps $x\in V$ to $x_{w}\in V$ for some $w\in\Omega^{(k)}$. Since $H$ is locally $k$-transitive and $\mathrm{QZ}(H)\unlhd H$, there is a translation $\tau_{w}\in\mathrm{QZ}(H)$ which maps $x$ to $x_{w}$ for all $\smash{w\in\Omega^{(k)}}$. By definition, the centralizer of $\tau_{w}$ in $H$ is open for all $w\in\Omega^{(k)}$. Hence, using non-discreteness of $H$ there is $n\in\bbN$ such that $H_{B(x,n)}$ commutes with $\tau_{w}$ for all $\smash{w\in\Omega^{(k)}}$ and $H_{B(x,n+1)}\neq\{\id\}$. However, $H_{B(x,n)}=\tau_{w}H_{B(x,n)}\tau_{w}^{-1}=H_{B(x_{w},n)}$ for all $\smash{w\in\Omega^{(k)}}$, that is $H_{B(x,n+k)}\subseteq H_{B(x,n)}$ in contradiction to the above.
\end{proof}

We complement part \ref{item:qz_semiprimitive} of Theorem \ref{thm:local_action_qz} with the following result inspired by \cite[Proposition 3.1.2]{BM00a} and \cite[Conjecture 2.63]{Rat04}, 

\begin{proposition}
Let $H\le\Aut(T_{d})$ be non-discrete and locally semiprimitive. If all orbits of $H\curvearrowright\partial T_{d}$ are uncountable then $\mathrm{QZ}(H)$ is trivial.
\end{proposition}

\begin{proof}
By Theorem \ref{thm:local_action_qz}, the group $\mathrm{QZ}(H)$ contains no inversions. Let $S\subseteq\partial T_{d}$ be the set of fixed points of hyperbolic elements in~$\mathrm{QZ}(H)$. Since $\mathrm{QZ}(H)\unlhd H$, the set $S$ is $H$-invariant. Also, $\mathrm{QZ}(H)$ is discrete by Theorem~\ref{thm:local_action_qz} and hence countable as $H$ is second-countable. Thus $S$ is countable and hence empty. We conclude that $\mathrm{QZ}(H)\unlhd H$ does not contain elliptic elements in view of \cite[Lemma 6.4]{GGT18}.
\end{proof}

The following strengthening of Theorem \ref{thm:local_action_qz_sharp}\ref{item:qz_sharp_transitive} proved in Section \ref{sec:qz_sharp_proof_ii} shows that Burger--Mozes theory does not generalize to the locally transitive case.

\begin{theorem}\label{thm:qz_open}
There is $d\in\bbN_{\ge 3}$ and a closed, non-discrete, compactly generated, locally transitive subgroup of $\Aut(T_{d})$ with open, hence non-discrete, quasi-center.
\end{theorem}

We prove Theorem \ref{thm:local_action_qz_sharp} by construction in the consecutive sections. Whereas parts \ref{item:qz_sharp_intransitive} to \ref{item:qz_sharp_k_transitive}\ref{item:qz_sharp_intransitive_hyp} all use groups of the form $\bigcap_{k\in\bbN}\mathrm{U}_{k}(F^{(k)})$ for appropriate local actions $\smash{F^{(k)}\le\Aut(B_{d,k})}$, part \ref{item:qz_sharp_k_transitive}\ref{item:qz_sharp_quasiprimitive_hyp} uses a group of the form $\smash{\bigcap_{k\in\bbN}\mathrm{BU}(F^{(2k)})}$. All sections appear similar at first glance but vary in detail.

\subsubsection{Theorem \ref{thm:local_action_qz_sharp}\ref{item:qz_sharp_intransitive}}\label{sec:qz_sharp_proof_i}

For certain intransitive $F\le\mathrm{Sym}(\Omega)$ we construct a closed, non-discrete, compactly generated, vertex-transitive group $H(F)\le\Aut(T_{d})$ which locally acts like $F$ and contains a quasi-central involutive inversion.

\vspace{0.2cm}
Let $F\le\Sym(\Omega)$. Assume that the partition $F\backslash\Omega=\bigsqcup_{i\in I}\Omega_{i}$ of $\Omega$ into $F$-orbits has at least three elements and that $F_{\Omega_{i}}\neq\{\id\}$ for all $i\in I$.
\vspace{0.2cm}

Fix an orbit $\Omega_{0}$ of size at least $2$ and $\omega_{0}\in\Omega_{0}$. Define groups $F^{(k)}\le\Aut(B_{d,k})$ for $k\in\bbN$ inductively by $F^{(1)}:=F$ and
\begin{displaymath}
 F^{(k+1)}\!:=\!\{(\alpha,(\alpha_{\omega})_{\omega})\!\mid\! \alpha\!\in\! F^{(k)},\ \alpha_{\omega}\!\in\! C_{F^{(k)}}(\alpha,\omega)\text{ constant w.r.t. $F\backslash\Omega$},\ \alpha_{\omega_{0}}\!=\!\alpha\}. 
\end{displaymath}

\begin{proposition}\label{prop:qz_inversion}
The groups $F^{(k)}\le\Aut(B_{d,k})$ ($k\in\bbN$) defined above satisfy:
\begin{enumerate}[(i)]
  \item\label{item:qz_inversion_self_comp} Every $\alpha\in F^{(k)}$ is self-compatible in directions from $\Omega_{0}$.
  \item\label{item:qz_inversion_C} The compatibility set $C_{F^{(k)}}(\alpha,\Omega_{i})$ is non-empty for all $\alpha\in F^{(k)}$ and $i\in I$. \newline In particular, the group $F^{(k)}$ satisfies \eqref{eq:C}.
 \item\label{item:qz_inversion_D} The compatibility set $C_{F^{(k)}}(\id,\Omega_{i})$ is non-trivial for all $\Omega_{i}\neq\Omega_{0}$.
 \newline In particular, the group $F^{(k)}$ does not satisfy \eqref{eq:D}.
\end{enumerate}
\end{proposition}

\begin{proof}
We prove all three properties simultaneously by induction: For $k=1$, the assertions \ref{item:qz_inversion_self_comp} and \ref{item:qz_inversion_C} are trivial. The third translates to $F_{\Omega_{i}}$ being non-trivial for all $\Omega_{i}\neq\Omega_{0}$ which is an assumption. Now, assume that all properties hold for $F^{(k)}$. Then the definition of $F^{(k+1)}$ is meaningful because of \ref{item:qz_inversion_self_comp} and it is a subgroup of $\Aut(B_{d,k+1})$ because $F$ preserves each $\Omega_{i}$ $(i\in I)$. Assertion \ref{item:qz_inversion_self_comp} is now evident. Statement \ref{item:qz_inversion_C} carries over from $F^{(k)}$ to $F^{(k+1)}$. So does \ref{item:qz_inversion_D} since $|F\backslash\Omega|\ge 3$.
\end{proof}

\begin{definition}
Retain the above notation. Define $H(F):=\bigcap_{k\in\bbN}\mathrm{U}_{k}(F^{(k)})$.
\end{definition}

Now, $H(F)$ is compactly generated, vertex-transitive and contains an involutive inversion because $\smash{\mathrm{U}_{1}(\{\id\})\le H(F)}$. Also, $H(F)$ is closed as an intersection of closed sets. The $1$-local action of $H$ is given by $F=F^{(1)}$ because $\Gamma^{k}(F)\le F^{(k)}$ for all $k\in\bbN$ and therefore $\mathrm{D}(F)\le H(F)$.

\begin{lemma}
The group $H(F)$ is non-discrete.
\end{lemma}

\begin{proof}
Let $x\in V$ and $n\in\bbN$. We construct a non-trivial element $h\in H(F)$ which fixes $B(x,n)$: Set $\alpha_{n}:=\id\in F^{(n)}$. By parts \ref{item:qz_inversion_self_comp} and \ref{item:qz_inversion_D} of Proposition \ref{prop:qz_inversion} as well as the definition of $F^{(n+1)}$, there is a non-trivial element $\alpha_{n+1}\in F^{(n+1)}$ with $\pi_{n}\alpha_{n+1}=\alpha_{n}$. Applying parts \ref{item:qz_inversion_self_comp} and \ref{item:qz_inversion_C} of Proposition \ref{prop:qz_inversion} repeatedly, we obtain non-trivial elements $\alpha_{k}\!\in\! F^{(k)}$ for all $k\ge n+1$ with $\pi_{k}\alpha_{k+1}=\alpha_{k}$. Set $\alpha_{k}:=\id\!\in\! F^{(k)}$ for all $k\!\le\! n$ and define $h\!\in\!\Aut(T_{d})_{x}$ by fixing $x$ and setting $\sigma_{k}(h,x)\!:=\!\alpha_{k}\in F^{(k)}$. Since $F^{(l)}\!\le\!\Phi^{l}(F^{(k)})$ for all $k\!\le\! l$ we conclude that $h\in\bigcap_{k\in\bbN}\mathrm{U}_{k}(F^{(k)})=H(F)$.
\end{proof}

\begin{proposition}
The quasi-center of $H(F)$ contains an involutive inversion.
\end{proposition}

\begin{proof}
Let $x\!\in\! V$. The group $\mathrm{QZ}(H(F))$ contains the label-respecting inversion~$\iota_{\omega}$ of $(x,x_{\omega})\!\in\! E$ for all $\omega\!\in\!\Omega_{0}$: Let $h\!\in\! H(F)_{B(x,1)}$ and $\omega\!\in\!\Omega_{0}$. Then $h\iota_{\omega}(x)\!=\!x_{\omega}\!=\!\iota_{\omega}h(x)$ and $\sigma_{k}(h\iota_{\omega},x)=\sigma_{k}(h,\iota_{\omega}x)\sigma_{k}(\iota_{\omega},x)=\sigma_{k}(h,x_{\omega})=\sigma_{k}(\iota_{\omega},hx)\sigma_{k}(h,x)=\sigma_{k}(\iota_{\omega}h,x)$ for all $k\in\bbN$ since $h\in\mathrm{U}_{k+1}(F^{(k+1)})$. That is, $\iota_{\omega}$ commutes with $H(F)_{B(b,1)}$.
\end{proof}

\subsubsection{Theorem \ref{thm:local_action_qz_sharp}\ref{item:qz_sharp_transitive}}\label{sec:qz_sharp_proof_ii} For certain transitive $F\le\mathrm{Sym}(\Omega)$ we construct a closed, non-discrete, compactly generated, vertex-transitive group $H(F)\le\Aut(T_{d})$ which locally acts like $F$ and has open quasi-center.

\vspace{0.2cm}
Let $F\le\Sym(\Omega)$ be transitive. Assume that $F$ preserves a non-trivial partition $\calP:\Omega=\bigsqcup_{i\in I}\Omega_{i}$ of $\Omega$ and that $F_{\Omega_{i}}\neq\{\id\}$ for all $i\in I$. Further, suppose that $F^{+}$ is abelian and preserves $\calP$ setwise.

\begin{example}
Let $F'\le\Sym(\Omega')$ be regular abelian and $P\le\Sym(\Lambda)$ regular. Then $F:=F'\wr P\le\Sym(\Omega'\times\Lambda)$ satisfies the above properties as $F^{+}=\prod_{\lambda\in\Lambda}F'$.
\end{example}

Define groups $F^{(k)}\le\Aut(B_{d,k})$ for $k\in\bbN$ inductively by $F^{(1)}:=F$ and
\begin{displaymath}
 F^{(k+1)}:=\{(\alpha,(\alpha_{\omega})_{\omega})\mid \alpha\in F^{(k)},\ \alpha_{\omega}\in C_{F^{(k)}}(\alpha,\omega) \text{ constant w.r.t. $\calP$}\}.
\end{displaymath}

\begin{proposition}\label{prop:qz_edge_fixating}
The groups $F^{(k)}\le\Aut(B_{d,k})$ ($k\in\bbN$) defined above satisfy:
\begin{enumerate}[(i)]
 \item\label{item:qz_edge_fixating_C} The compatibility set $C_{F^{(k)}}(\alpha,\Omega_{i})$ is non-empty for all $\alpha\in F^{(k)}$ and $i\in I$.
 \newline In particular, the group $F^{(k)}$ satisfies \eqref{eq:C}.
 \item\label{item:qz_edge_fixating_D} The compatibility set $C_{F^{(k)}}(\id,\Omega_{i})$ is non-trivial for all $i\in I$.
 \newline In particular, the group $F^{(k)}$ does not satisfy \eqref{eq:D}.
 \item\label{item:qz_edge_fixating_abelian} The group $F^{(k)}\cap\Phi^{k}(F^{+})$ is abelian.
\end{enumerate}
\end{proposition}

\begin{proof}
We prove all three properties simultaneously by induction: For $k=1$, the assertion \ref{item:qz_edge_fixating_C} is trivial whereas \ref{item:qz_edge_fixating_abelian} is an assumption. The second translates to $F_{\Omega_{i}}$ being non-trivial for all $i\in I$ which is an assumption. Now, assume all properties hold for $F^{(k)}$. Then the definition of $F^{(k+1)}$ is meaningful because of \ref{item:qz_edge_fixating_C} and it is a subgroup of $\Aut(B_{d,k})$ because $F$ preserves $\calP$. Statement \ref{item:qz_edge_fixating_D} carries over from $F^{(k)}$ to $F^{(k+1)}$. Finally, \ref{item:qz_edge_fixating_abelian} follows inductively because $F^{+}$ preserves $\calP$ setwise.
\end{proof}

\begin{definition}
Retain the above notation. Define $H(F):=\bigcap_{k\in\bbN}\mathrm{U}_{k}(F^{(k)})$.
\end{definition}

Now, $H(F)$ is compactly generated, vertex-transitive and contains an involutive inversion because $\smash{\mathrm{U}_{1}(\{\id\})\le H(F)}$. Also, $H(F)$ is closed as an intersection of closed sets. The $1$-local action of $H$ is given by $F=F^{(1)}$ because $\Gamma_{k}(F)\le F^{(k)}$ for all $k\in\bbN$ and therefore $\mathrm{D}(F)\le H(F)$.

\begin{lemma}
The group $H(F)$ is non-discrete.
\end{lemma}

\begin{proof}
Let $x\in V$ and $n\in\bbN$. We construct a non-trivial element $h\in H(F)$ which fixes $B(x,n)$: Consider $\smash{\alpha_{n}:=\id\in F^{(n)}}$. By part \ref{item:qz_edge_fixating_D} of Proposition \ref{prop:qz_edge_fixating} as well as the definition of $\smash{F^{(n+1)}}$, there is a non-trivial element $\smash{\alpha_{n+1}\in F^{(n+1)}}$ with $\pi_{n}\alpha_{n+1}=\alpha_{n}$. Applying part \ref{item:qz_edge_fixating_C} of Proposition \ref{prop:qz_edge_fixating} repeatedly, we obtain non-trivial elements $\smash{\alpha_{k}\in F^{(k)}}$ for all $k\ge n+1$ with $\pi_{k}\alpha_{k+1}=\alpha_{k}$. Set $\alpha_{k}:=\id\in F^{(k)}$ for all $k\le n$ and define $h\in\Aut(T_{d})_{x}$ by fixing $x$ and setting $\smash{\sigma_{k}(h,x):=\alpha_{k}\in F^{(k)}}$. Since $\smash{F^{(l)}\!\le\!\Phi^{l}(F^{(k)})}$ for all $k\!\le\! l$ we conclude that $h\in\bigcap_{k\in\bbN}\mathrm{U}_{k}(F^{(k)})=H(F)$.
\end{proof}

\begin{proposition}\label{prop:open_qz}
The group $H(F)$ has open quasi-center.
\end{proposition}

\begin{proof}
The group $H(F)_{B(x,1)}$ is a subgroup of the group $H(F^{+})_{x}$ which is abelian by part \ref{item:qz_edge_fixating_abelian} of Proposition \ref{prop:qz_edge_fixating}. Hence $H(F)_{B(x,1)}\le \mathrm{QZ}(H(F))$.
\end{proof}

\begin{remark}
Without assuming local transitivity one can achieve abelian point stabilizers, following the construction of the previous section. This cannot happen for non-discrete locally transitive groups $H\le\Aut(T_{d})$ which are vertex-transitive as the following argument shows: By Proposition \ref{prop:uf_universal}, the group $H$ is contained in $\mathrm{U}(F)$ where $F\le\Sym(\Omega)$ is the local action of $H$. If $H_{x}$ is abelian, then so is $F$. Since any transitive abelian permutation group is regular we conclude that $\mathrm{U}(F)$ and hence $H$ are discrete. In this sense, the construction of this section is efficient.
\end{remark}

\subsubsection{Theorem \ref{thm:local_action_qz_sharp}\ref{item:qz_sharp_semiprimitive}}\label{sec:qz_sharp_proof_iii} For certain semiprimitive $F\le\mathrm{Sym}(\Omega)$ we construct a closed, non-discrete, compactly generated, vertex-transitive group $H(F)\!\le\!\Aut(T_{d})$ which locally acts like $F$ and contains a non-trivial quasi-central elliptic element.

\vspace{0.2cm}
Let $F\!\le\!\Sym(\Omega)$ be semiprimitive. Suppose $F$ preserves a non-trivial partition $\calP:\Omega=\bigsqcup_{i\in I}\Omega_{i}$ of $\Omega$ and that $F_{\Omega_{i}}\neq\{\id\}$ for all $i\in I$. Further, suppose that $F$ contains a non-trivial central element $\tau$ which preserves $\calP$ setwise.

\begin{example}
Consider $\SL(2,3)\curvearrowright\bbF_{3}^{2}\backslash\{0\}=\{\pm e_{1},\pm e_{2},\pm(e_{1}+e_{2}),\pm(e_{1}-e_{2})\}$ where $e_{1},e_{2}$ are the standard basis vectors. We have $Z(\SL(2,3))=\{\pm\Id\}$. The blocks of size $2$ are as listed above given that $\SL(2,3)_{e_{1}}\le_{2}\pm\SL(2,3)_{e_{1}}$.
\end{example}

Define groups $F^{(k)}\le\Aut(B_{d,k})$ for $k\in\bbN$ inductively by $F^{(1)}:=F$ and
\begin{displaymath}
 F^{(k+1)}:=\{(\alpha,(\alpha_{\omega})_{\omega})\mid \alpha\in F^{(k)},\ \alpha_{\omega}\in C_{F^{(k)}}(\alpha,\omega) \text{ constant w.r.t $\calP$}\}.
\end{displaymath}

\begin{proposition}\label{prop:qz_elliptic}
The groups $F^{(k)}\le\Aut(B_{d,k})$ ($k\in\bbN$) defined above satisfy:
\begin{enumerate}[(i)]
 \item\label{item:qz_elliptic_C} The compatibility set $C_{F^{(k)}}(\alpha,\Omega_{i})$ is non-empty for all $\alpha\in F^{(k)}$ and $i\in I$. \newline In particular, the group $F^{(k)}$ satisfies \eqref{eq:C}.
 \item\label{item:qz_elliptic_D} The compatibility set $C_{F^{(k)}}(\id,\Omega_{i})$ is non-trivial for all $i\in I$.
 \newline In particular, the group $F^{(k)}$ does not satisfy \eqref{eq:D}.
 \item\label{item:qz_elliptic_central} The element $\gamma_{k}(\tau)\in\Aut(B_{d,k})$ is central in $F^{(k)}$.
\end{enumerate}
\end{proposition}

\begin{proof}
We prove all three properties simultaneously by induction: For $k=1$, the assertion \ref{item:qz_elliptic_C} is trivial whereas \ref{item:qz_elliptic_central} is an assumption. The second translates to $F_{\Omega_{i}}$ being non-trivial for all $i\in I$ which is an assumption. Now, assume all properties hold for $F^{(k)}$. Then the definition of $F^{(k+1)}$ is meaningful because of \ref{item:qz_elliptic_C} and it is a subgroup of $\Aut(B_{d,k+1})$ because $F$ preserves $\calP$. Statement \ref{item:qz_elliptic_D} carries over from $F^{(k)}$ to $F^{(k+1)}$. Finally, \ref{item:qz_elliptic_central} follows inductively because $\tau$ and hence $\tau^{-1}$ preserves $\calP$ setwise: For $\smash{\widetilde{\alpha}=(\alpha,(\alpha_{\omega})_{\omega})\in F^{(k+1)}}$ we have
\begin{displaymath}
\gamma^{k+1}(\tau)\widetilde{\alpha}\gamma^{k+1}(\tau)^{-1}=(\gamma^{k}(\tau)\alpha\gamma^{k}(\tau)^{-1},(\gamma^{k}(\tau)\alpha_{\tau^{-1}(\omega)}\gamma^{k}(\tau)^{-1})_{\omega}). \qedhere
\end{displaymath}
\end{proof}

\begin{definition}
Retain the above notation. Define $H(F):=\bigcap_{k\in\bbN}\mathrm{U}_{k}(F^{(k)})$.
\end{definition}

Now, $H(F)$ is compactly generated, vertex-transitive and contains an involutive inversion because $\smash{\mathrm{U}_{1}(\{\id\})\le H(F)}$. Also, $H(F)$ is closed as an intersection of closed sets. The $1$-local action of $H$ is given by $F=F^{(1)}$ because $\Gamma^{k}(F)\le F^{(k)}$ for all $k\in\bbN$ and therefore $\mathrm{D}(F)\le H(F)$.

\begin{lemma}
The group $H(F)$ is non-discrete.
\end{lemma}

\begin{proof}
Let $x\in V$ and $n\in\bbN$. We construct a non-trivial element $h\in H(F)$ which fixes $B(x,n)$: Consider $\alpha_{n}:=\id\in F^{(n)}$. By part \ref{item:qz_elliptic_D} of Proposition \ref{prop:qz_elliptic} and the definition of $F^{(n+1)}$, there is a non-trivial $\alpha_{n+1}\in F^{(n+1)}$ with $\pi_{n}\alpha_{n+1}=\alpha_{n}$. Applying part \ref{item:qz_elliptic_C} of Proposition \ref{prop:qz_elliptic} repeatedly, we obtain non-trivial elements $\alpha_{k}\in F^{(k)}$ for all $k\ge n+1$ with $\pi_{k}\alpha_{k+1}=\alpha_{k}$. Set $\alpha_{k}:=\id\in F^{(k)}$ for all $k\le n$ and define $h\in\Aut(T_{d})_{x}$ by fixing $x$ and setting $\sigma_{k}(h,x):=\alpha_{k}\in F^{(k)}$. Since $F^{(l)}\le\Phi^{l}(F^{(k)})$ for all $k\le l$ we conclude that $h\in\bigcap_{k\in\bbN}\mathrm{U}_{k}(F^{(k)})=H(F)$.
\end{proof}

\begin{proposition}
The quasi-center of $H\!(F)$ contains a non-trivial elliptic element.
\end{proposition}

\begin{proof}
By Proposition \ref{prop:qz_elliptic}, the element $d(\tau)$ which fixes $x$ and whose $1$-local action is $\tau$ everywhere commutes with $H(F)_{x}$. Hence $d(\tau)\in\mathrm{QZ}(H(F))$.
\end{proof}

\begin{remark}
The argument of this section does not work in the quasiprimitive case because a quasiprimitive group $F\le\Sym(\Omega)$ with non-trivial center is abelian and regular: If $Z(F)\unlhd F$ is non-trivial then it is transitive, and it suffices to show that $F^{+}$ is trivial. Suppose $a\in F_{\omega}$ moves $\omega'\in\Omega$. Pick $z\in Z(F)$ with $z(\omega)=\omega'$. Then $za(\omega)=\omega'\neq az(\omega)$, contradicting the assumption that $z\in Z(F)$.
\end{remark}

\subsubsection{Theorem \ref{thm:local_action_qz_sharp}\ref{item:qz_sharp_k_transitive}\ref{item:qz_sharp_intransitive_hyp}}\label{sec:qz_sharp_proof_iv_a} For certain intransitive $F\le\mathrm{Sym}(\Omega)$ we construct a closed, non-discrete, compactly generated, vertex-transitive group $H(F)\le\Aut(T_{d})$ which locally acts like $F$ and contains a quasi-central hyperbolic element of length~$1$.

\vspace{0.2cm}
Let $F\le\mathrm{Sym}(\Omega)$. Assume that the partition $F\backslash\Omega=\bigsqcup_{i\in I}\Omega_{i}$ of $\Omega$ has at least three elements and that $Z(F)\neq\{\id\}$. Choose a non-trivial element $\tau\in Z(F)$ and $\omega_{0}\in\Omega_{0}\in F\backslash\Omega$ with $\tau(\omega_{0})\neq\omega_{0}$. Further, suppose that $F_{\Omega_{i}}\neq\{\id\}$ for all $\Omega_{i}\neq\Omega_{0}$.

\vspace{0.2cm}
Define groups $F^{(k)}\le\Aut(B_{d,k})$ for $k\in\bbN$ inductively by $F^{(1)}:=F$ and
\begin{displaymath}
 F^{(k+1)}\!:=\!\{(\alpha,(\alpha_{\omega})_{\omega})\!\mid\! \alpha\!\in\! F^{(k)},\ \alpha_{\omega}\!\in\! C_{F^{(k)}}(\alpha,\omega)\text{ constant w.r.t. $F\backslash\Omega$},\ \alpha_{\omega_{0}}\!=\!\alpha\}.
\end{displaymath}

\begin{proposition}\label{prop:qz_intransitive_hyp}
The groups $F^{(k)}\le\Aut(B_{d,k})$ ($k\in\bbN$) defined above satisfy:
\begin{enumerate}[(i)]
  \item\label{item:qz_intransitive_hyp_self_comp} Every $\alpha\in F^{(k)}$ is self-compatible in directions from $\Omega_{0}$.
  \item\label{item:qz_intransitive_hyp_C} The compatibility set $C_{F^{(k)}}(\alpha,\Omega_{i})$ is non-empty for all $\alpha\in F^{(k)}$ and $i\in I$. \newline In particular, the group $F^{(k)}$ satisfies \eqref{eq:C}.
 \item\label{item:qz_intransitive_hyp_D} The compatibility set $C_{F^{(k)}}(\id,\Omega_{i})$ is non-trivial for all $i\in I\backslash\{0\}$.
 \newline In particular, the group $F^{(k)}$ does not satisfy \eqref{eq:D}.
 \item\label{item:qz_intransitive_hyp_central} The element $\gamma_{k}(\tau)\in\Aut(B_{d,k})$ is central in $F^{(k)}$.
\end{enumerate}
\end{proposition}

\begin{proof}
We prove all four properties simultaneously by induction: For $k=1$, the assertions \ref{item:qz_intransitive_hyp_self_comp} and \ref{item:qz_intransitive_hyp_C} are trivial. The third translates to $F_{\Omega_{i}}$ being non-trivial for all $i\in I\backslash\{0\}$ which is an assumption, as is \ref{item:qz_intransitive_hyp_central}. Now, assume that all properties hold for $F^{(k)}$. Then the definition of $F^{(k+1)}$ is meaningful because of \ref{item:qz_intransitive_hyp_self_comp} and it is a subgroup of $\Aut(B_{d,k})$ because $F$ preserves $F\backslash\Omega$. Assertion \ref{item:qz_intransitive_hyp_self_comp} is now evident. Statements \ref{item:qz_intransitive_hyp_C} and \ref{item:qz_intransitive_hyp_D} carry over from $F^{(k)}$ to $F^{(k+1)}$. Finally, \ref{item:qz_intransitive_hyp_D} follows inductively because $\tau$ and hence $\tau^{-1}$ preserves $F\backslash\Omega$ setwise: For $\smash{\widetilde{\alpha}=(\alpha,(\alpha_{\omega})_{\omega})\in F^{(k+1)}}$ we have
\begin{displaymath}
\gamma^{k+1}(\tau)\widetilde{\alpha}\gamma^{k+1}(\tau)^{-1}=(\gamma^{k}(\tau)\alpha\gamma^{k}(\tau)^{-1},(\gamma^{k}(\tau)\alpha_{\tau^{-1}(\omega)}\gamma^{k}(\tau)^{-1})_{\omega}). \qedhere
\end{displaymath}
\end{proof}

\begin{definition}
Retain the above notation. Define $H(F):=\bigcap_{k\in\bbN}\mathrm{U}_{k}(F^{(k)})$.
\end{definition}

Now, $H(F)$ is compactly generated, vertex-transitive and contains an involutive inversion because $\smash{\mathrm{U}_{1}(\{\id\})\le H(F)}$. Also, $H(F)$ is closed as the intersection of all its $(P_{k})$-closures. The $1$-local action of $H$ is given by $F=F^{(1)}$ as $\Gamma^{k}(F)\le F^{(k)}$ for all $k\in\bbN$ and therefore $\mathrm{D}(F)\le H$.

\begin{lemma}
The group $H(F)$ is non-discrete.
\end{lemma}

\begin{proof}
Let $x\in V$ and $n\in\bbN$. We construct a non-trivial element $h\in H(F)$ which fixes $B(x,n)$: Consider $\alpha_{n}:=\id\in F^{(n)}$. By parts \ref{item:qz_intransitive_hyp_self_comp} and \ref{item:qz_intransitive_hyp_D} of Proposition \ref{prop:qz_intransitive_hyp} as well as the definition of $F^{(n+1)}$, there is a non-trivial element $\alpha_{n+1}\in F^{(n+1)}$ with $\pi_{n}\alpha_{n+1}=\alpha_{n}$. Applying parts \ref{item:qz_intransitive_hyp_self_comp} and \ref{item:qz_intransitive_hyp_C} of Proposition \ref{prop:qz_intransitive_hyp} repeatedly, we obtain non-trivial elements $\alpha_{k}\in F^{(k)}$ for all $k\ge n+1$ with $\pi_{k}\alpha_{k+1}=\alpha_{k}$. Set $\alpha_{k}:=\id\in F^{(k)}$ for all $k\le n$ and define $h\in\Aut(T_{d})_{x}$ by fixing $x$ and setting $\sigma_{k}(h,x):=\alpha_{k}\in F^{(k)}$. Since $F^{(l)}\le\Phi^{l}(F^{(k)})$ for all $k\le l$ we conclude that $h\in\bigcap_{k\in\bbN}\mathrm{U}_{k}(F^{(k)})=H(F)$.
\end{proof}

\begin{proposition}
The quasi-center of $H(F)$ contains a translation of length $1$.
\end{proposition}

\begin{proof}
Fix $x\in V$ and let $\tau$ be as above. Consider the line $L$ through $x$ with labels
\begin{displaymath}
  \ldots,\tau^{-2}\omega_{0},\tau^{-1}\omega_{0},\omega_{0},\tau\omega_{0},\tau^{2}\omega_{0},\ldots
\end{displaymath}
Define $t\in \mathrm{D}(F)$ by $t(x)=x_{\omega_{0}}$ and $\sigma_{1}(t,y)=\tau$ for all $y\in V$. Then $t$ is a translation of length $1$ along $L$. Furthermore, $t$ commutes with $H(F)_{B(x,1)}$: Indeed, let $g\in H(F)_{B(x,1)}$. Then $(gt)(x)=t(x)=(tg)(x)$ and
\begin{displaymath}
 \sigma_{k}(gt,x)=\sigma_{k}(g,tx)\sigma_{k}(t,x)=\sigma_{k}(t,x)\sigma_{k}(g,x)=\sigma_{k}(t,gx)\sigma_{k}(g,x)=\sigma_{k}(tg,x)
\end{displaymath}
for all $k\in\bbN$ because $\sigma_{k}(t,x)=\gamma^{k}(\tau)\in Z(F^{(k)})$ and $g\in\mathrm{U}_{k+1}(F^{(k+1)})_{B(x,1)}$.
\end{proof}

\subsubsection{Theorem \ref{thm:local_action_qz_sharp}\ref{item:qz_sharp_k_transitive}\ref{item:qz_sharp_quasiprimitive_hyp}}\label{sec:qz_sharp_proof_iv_b} For certain quasiprimitive $F\le\mathrm{Sym}(\Omega)$ we construct a closed, non-discrete, compactly generated group $H(F)\le\Aut(T_{d})$ which locally acts like $F$ and contains a quasi-central hyperbolic element of length $2$.

\vspace{0.2cm}
Let $F\!\le\!\Sym(\Omega)$ be quasiprimitive. Suppose $F$ preserves a non-trivial~partition $\calP:\Omega=\bigsqcup_{i\in I}\Omega_{i}$. Further, suppose that $F_{\Omega_{i}}\!\neq\!\{\id\}$ and that $F_{\omega_{i}}\curvearrowright\Omega_{i}\backslash\{\omega_{i}\}$ is transitive for all $i\in I$ and $\omega_{i}\in\Omega_{i}$.

\begin{example}
Consider $A_{5}\curvearrowright A_{5}/C_{5}$ which has blocks of size $[D_{5}:C_{5}]=2$ and non-trivial block stabilizers as $C_{5}\cap\tau C_{5}\tau^{-1}=C_{5}$ for all $\tau\in D_{5}$ given that $C_{5}\unlhd D_{5}$.
\end{example}

Retain the notation of Example \ref{ex:bukf_orbit_self_compatible}. Define groups $\smash{F^{(2k)}\le\Aut(B_{d,2k})}$ for $k\in\bbN$ inductively by $\smash{F^{(2)}=\{(a,(a_{\omega})_{\omega})\mid a\in F, a_{\omega}\in C_{F}(a,\omega) \text{ constant w.r.t. $\calP$}\}}$ and
\begin{displaymath}
 F^{(2(k+1))}:=\{(\alpha,(\alpha_{w})_{w})\mid \alpha\in F^{(2k)}, \alpha_{w}\in C_{F^{(2k)}}(\alpha,w),\ \forall w\in\Omega^{(2)}_{0}:\ \alpha_{w}=\alpha\}.
\end{displaymath}

\begin{proposition}\label{prop:qz_quasiprimitive_hyp}
The groups $F^{(2k)}\le\Aut(B_{d,2k})$ ($k\in\bbN$) defined above satisfy:
\begin{enumerate}[(i)]
  \item\label{item:qz_quasiprimitive_hyp_self_comp} Every $\alpha\in F^{(2k)}$ is self-compatible in all directions from $\Omega^{(2)}_{0}$.
  \item\label{item:qz_quasiprimitive_hyp_C} The compatibility set $C_{F^{(2k)}}(\alpha,w)$ is non-empty for all $\alpha\!\in\! F^{(2k)}$ and $w\!\in\!\Omega^{(2)}$. \newline In particular, the group $F^{(2k)}$ satisfies (C).
  \item\label{item:qz_quasiprimitive_hyp_D} The compatibility set $C_{F^{(2k)}}(\id,w)$ is non-trivial for all $w\in\Omega^{(2)}$.
  \newline In particular, the group $F^{(2k)}$ does not satisfy (D).
\end{enumerate}
\end{proposition}


\begin{proof}
We prove all three properties simultaneously by induction: For $k=1$, the assertion \ref{item:qz_quasiprimitive_hyp_self_comp} holds by construction of $F^{(2)}$, as do \ref{item:qz_quasiprimitive_hyp_C} and \ref{item:qz_quasiprimitive_hyp_D}. Now assume that all properties hold for $F^{(2k)}$. Then the definition of $F^{(2(k+1))}$ is meaningful because of \ref{item:qz_quasiprimitive_hyp_self_comp} and it is a subgroup because $F^{(2)}$ preserves $\smash{\Omega^{(2)}_{0}}$. Also, $F^{(2(k+1))}$ satisfies \ref{item:qz_quasiprimitive_hyp_self_comp} because $\smash{\Omega^{(2)}_{0}}$ is inversion-closed. Statements \ref{item:qz_quasiprimitive_hyp_C} and \ref{item:qz_quasiprimitive_hyp_D} carry over from $\smash{F^{(2k)}}$.
\end{proof}

\begin{definition}
Retain the above notation. Define $H(F):=\bigcap_{k\in\bbN}\mathrm{BU}_{2k}(F^{(2k)})$.
\end{definition}

Now, $H(F)$ is closed as an intersection of closed sets and compactly generated by $H(F)_{x}$ for some $x\in V_{1}$ and a finite generating set of $\mathrm{BU}_{2}(\{\id\})^{+}$, see Lemma~\ref{lem:bu_comp_gen}. For vertices in $V_{1}$, the $1$-local action is $F$ because $\Gamma^{2k}(F)\le F^{(2k)}$. For vertices in $V_{2}$ the $1$-local action is $F^{+}=F$ as $\Gamma^{2}(F)\le F^{(2)}$.

\begin{lemma}
The group $H(F)$ is non-discrete.
\end{lemma}

\begin{proof}
Let $x\in V_{1}$ and $n\in\bbN$. We construct a non-trivial element $h\in H(F)$ which fixes $B(x,2n)$: Consider $\alpha_{2n}:=\id\in F^{(2n)}$: By parts \ref{item:qz_quasiprimitive_hyp_self_comp} and \ref{item:qz_quasiprimitive_hyp_D} of Proposition \ref{prop:qz_inversion} and the definition of $F^{(2(n+1))}$, there is a non-trivial element $\smash{\alpha_{2(n+1)}\in F^{(2(n+1))}}$ with $\pi_{2n}\alpha_{2(n+1)}=\alpha_{2n}$. Applying parts \ref{item:qz_quasiprimitive_hyp_self_comp} and \ref{item:qz_quasiprimitive_hyp_C} of Proposition \ref{prop:qz_quasiprimitive_hyp} repeatedly, we obtain non-trivial elements $\smash{\alpha_{2k}\in F^{(2k)}}$ for all $k\ge n+1$ with $\smash{\pi_{2k}\alpha^{2(k+1)}=\alpha_{2k}}$. Set $\smash{\alpha_{2k}:=\id\in F^{(2k)}}$ for all $k\le n$ and define $h\in\Aut(T_{d})_{x}$ by fixing $x$ and setting $\smash{\sigma_{2k}(h,x):=\alpha_{2k}\in F^{(2k)}}$. Since $\smash{F^{(2l)}\!\le\!\mathrm{B}\Phi^{2l}(F^{(2k)})}$ for all $k\!\le\! l$ we conclude that $\smash{h\in\bigcap_{k\in\bbN}\mathrm{BU}_{2k}(F^{(2k)})=H(F)}$.
\end{proof}

\begin{proposition}
The quasi-center of $H(F)$ contains a translation of length $2$.
\end{proposition}

\begin{proof}
Fix $x\!\in\! V_{1}$ and $w\!=\!(\omega_{1},\omega_{2})\!\in\!\Omega^{(2)}_{0}$. Consider the line $L$ through $b$ with labels
\begin{displaymath}
  \ldots,\omega_{1},\omega_{2},\omega_{1},\omega_{2},\ldots
\end{displaymath}
Define $t\in\mathrm{D}(F)$ by $t(x)= x_{w}$ and $\sigma_{1}(t,y)=\id$ for all $y\in V$. Then $t$ is a translation of length $2$ along $L$. Furthermore, $t$ commutes with $H(F)_{B(x,2)}$: Indeed, let $g\in H(F)_{B(x,2)}$. Then $gt(x)=t(x)=tg(x)$ and for all $k\in\bbN$:
\begin{align*}
 \sigma_{2k}(gt,x)=\sigma_{2k}(g,tx)\sigma_{2k}(t,x)&=\sigma_{2k}(g,x_{w}) \\
 &=\sigma_{2k}(g,x)=\sigma_{2k}(t,gx)\sigma_{2k}(g,x)=\sigma_{2k}(tg,x)
\end{align*}
as $\sigma_{l}(t,y)=\id$ for all $l\in\bbN$ and $y\in V(T_{d})$, and $g\in \mathrm{BU}_{2(k+1)}(F^{(2(k+1))})_{B(b,2)}$.
\end{proof}

\begin{remark}
We argue that the construction of this section does not carry over to any primitive $F\le\Sym(\Omega)$ and $\Gamma(F)\le F^{(2)}\le\Phi(F)$.

First, note that $\Phi(F)\backslash\Omega^{(2)}=\Gamma(F)\backslash\Omega^{(2)}$: For $\alpha:=(a,(a_{\omega})_{\omega\in\Omega})\in\Phi(F)$ and $\smash{(\omega_{1},\omega_{2})\!\in\!\Omega^{(2)}}$ we have $\alpha(\omega_{1},\omega_{2})=(a\omega_{1},a_{\omega_{1}}\omega_{2})\in\{(a\omega_{1},aF_{\omega_{1}}\omega_{2})\}\subseteq\Gamma(F)(\omega_{1},\omega_{2})$. We now observe the following obstruction to non-discreteness: Given any orbit $\smash{\Omega^{(2)}_{0}\in\Phi(F)\backslash\Omega^{(2)}=F^{(2)}\backslash\Omega^{(2)}}$, the subgroup of $\Phi(F)$ consisting of elements which are self-compatible in all directions from $\smash{\Omega^{(2)}_{0}}$ is precisely $\Gamma(F)$.

Indeed, every element of $\Gamma(F)$ is self-compatible in all directions from $\Omega^{(2)}\!\supseteq\!\Omega_{0}^{2}$. Conversely, let $(a,(a_{\omega})_{\omega})\in\Phi(F)$ be self-compatible in all directions from $\smash{\Omega^{(2)}_{0}}$. Consider the equivalence relation on $\Omega$ defined by $\omega_{1}\sim\omega_{2}$ if and only if $a_{\omega_{1}}=a_{\omega_{2}}$. Since $a_{\omega_{1}}=a_{\omega_{2}}$ whenever $w:=(\omega_{1},\omega_{2})\in\smash{\Omega^{(2)}_{0}}$, this relation is $F$-invariant: Since $\Gamma(F)\le\Phi(F)$ we have $\smash{\gamma(a)(\omega_{1},\omega_{2})=(a\omega_{1},a\omega_{2})\in\Omega^{(2)}_{0}}$ for all $a\in F$ whenever $\smash{(\omega_{1},\omega_{2})\in\Omega^{(2)}_{0}}$. Since $F$ is primitive, it is the universal relation, so $(a,(a_{\omega})_{\omega})\!\in\!\Gamma(F)$. 
\end{remark}

\subsection{Banks--Elder--Willis $(P_{k})$-closures}\label{sec:bew}

Theorem \ref{thm:ukf_universal} yields a description of the $(P_{k})$-closures of locally transitive subgroups of $\Aut(T_{d})$ which contain an involutive inversion, and thereby a characterization of the locally transitive universal groups. Recall that the $(P_{k})$-closure of a subgroup $H\le\Aut(T_{d})$ is
\begin{displaymath}
 H^{(P_{k})}=\{g\in\Aut(T_{d})\mid\forall x\in V\ \exists h\in H:\ g|_{B(x,k)}=h|_{B(x,k)}\}.
\end{displaymath}

Combined with Corollary \ref{cor:ukf_c_iso} the following partially answers the question for an algebraic description of a group's $(P_{k})$-closure in the last paragraph of \cite{BEW15}.

\begin{theorem}\label{thm:k_closure_char}
Let $H\le\Aut(T_{d})$ be locally transitive and contain an involutive inversion. Then $\smash{H^{(P_{k})}=\mathrm{U}_{k}^{(l)}(F^{(k)})}$ for some labelling $l$ of $T_{d}$ and $F^{(k)}\le\Aut(B_{d,k})$.
\end{theorem}

\begin{proof}
Let $l$ and $F^{(k)}\le\Aut(B_{d,k})$ be as in Theorem \ref{thm:ukf_universal}. Then $\smash{H^{(P_{k})}\!=\!\mathrm{U}_{k}^{(l)}(F^{(k)})}$:

Let $\smash{g\in\mathrm{U}_{k}(F^{(k)})}$ and $x\in V$. Since $\smash{\mathrm{U}_{1}^{(l)}(\{\id\})\le H}$ there is $\smash{h'\in\mathrm{U}_{1}^{(l)}(\{\id\})}\le H$ with $h'(x)=g(x)$, and since $H$ is $k$-locally action isomorphic to $\smash{F^{(k)}}$ there is $h''\!\in\! H_{x}$ such that $\sigma_{k}(h'',x)=\sigma_{k}(g,x)$. Then $h:=h'h''\in H$ satisfies $g|_{B(x,k)}=h|_{B(x,k)}$.

Conversely, let $\smash{g\in H^{(P_{k})}}$. Then all $k$-local actions of $g$ stem from elements of $H$. Given that $H\le\mathrm{U}_{k}(F^{(k)})$ by Theorem \ref{thm:ukf_universal}, we conclude that $g\in\mathrm{U}_{k}(F^{(k)})$.
\end{proof}

\begin{corollary}\label{cor:ukf_char}
Let $H\le\Aut(T_{d})$ be closed, locally transitive and contain an involutive inversion. Then $\smash{H=\mathrm{U}_{k}^{(l)}(F^{(k)})}$ for some labelling $l$ of $T_{d}$ and an action $\smash{F^{(k)}\le\Aut(B_{d,k})}$ if and only if $H$ satisfies Property $(P_{k})$.
\end{corollary}

\begin{proof}
If $H=\mathrm{U}_{k}^{(l)}(F^{(k)})$ then $H$ satisfies Property $(P_{k})$ by Proposition \ref{prop:ukf_pk}. Conversely, if $H$ satisfies Property $\smash{(P_{k})}$ then $\smash{H=\overline{H}\!=\!H^{(P_{k})}}$ by \cite[Theorem 5.4]{BEW15} and the assertion follows from Theorem \ref{thm:k_closure_char}.
\end{proof}

Banks--Elder--Willis use certain subgroups of $\Aut(T_{d})$ with pairwise distinct $(P_{k})$-closures to construct infinitely many, pairwise non-conjugate, non-discrete simple subgroups of $\Aut(T_{d})$ via Theorem \ref{thm:bew_simplicity} and \cite[Theorem 8.2]{BEW15}. For example, the group $\PGL(2,\bbQ_{p})\!\le\!\Aut(T_{p+1})$ qualifies by the argument in \cite[Section~4.1]{BEW15}. Whereas $\PGL(2,\bbQ_{p})$ has trivial quasi-center given that it is simple, certain groups with non-trivial quasi-center, always have infinitely many distinct $(P_{k})$-closures.

\begin{proposition}\label{prop:qz_k_closure}
Let $H\le\Aut(T_{d})$ be closed, non-discrete, locally transitive and contain an involutive inversion. If, in addition, $H$ has non-trivial quasi-center then $H$ has infinitely many distinct $(P_{k})$-closures.
\end{proposition}

\begin{proof}
We have $H^{(P_{k})}=\mathrm{U}_{k}(F^{(k)})$ by Theorem \ref{thm:k_closure_char}. Therefore, $H=\bigcap_{k\in\bbN}\mathrm{U}_{k}(F^{(k)})$ by \cite[Proposition 3.4 (iii)]{BEW15}. If $H$ had only finitely many distinct $(P_{k})$-closures, the sequence $(H^{(P_{k})})_{k\in\bbN}$ of subgroups of $\Aut(T_{d})$ would be eventually constant equal to, say, $H^{(n)}=\mathrm{U}_{n}(F^{(n)})\ge H$. However, since $H$ is non-discrete, so is $\mathrm{U}_{n}(F^{(n)})$ which thus has trivial quasi-center by Proposition \ref{prop:ukf_qz}.
\end{proof}

Banks--Elder--Willis ask whether the infinitely many, pairwise non-conjugate, non-discrete simple subgroups of $\Aut(T_{d})$ they construct are also pairwise non-isomorphic as topological groups. By Proposition \ref{prop:isomorphism_conjugate}, this is the case if said simple groups are locally transitive with distinct point stabilizers, which can be determined from the original group's $k$-local actions thanks to Theorem \ref{thm:k_closure_char}.

\begin{theorem}\label{thm:bew_non_isomorphism}
Let $H\le\Aut(T_{d})$ be non-discrete, locally permutation isomorphic to $F\le\Sym(\Omega)$ and contain an involutive inversion. Suppose that $F$ is transitive and that every non-trivial subnormal subgroup of $F_{\omega}$ $(\omega\!\in\!\Omega)$ is transitive on $\Omega\backslash\{\omega\}$. If $H^{(P_{k})}\neq H^{(P_{l})}$ for some $k,l\in\bbN$ then $(H^{(P_{k})})^{+_{k}}$ and $(H^{(P_{l})})^{+_{l}}$ are non-isomorphic.
\end{theorem}

\begin{proof}
In view of \cite[Theorem 8.2]{BEW15}, the groups $(H^{(P_{k})})^{+_{k}}$ and $(H^{(P_{l})})^{+_{l}}$ are non-conjugate. We show that they satisfy the assumptions of Proposition \ref{prop:isomorphism_conjugate} which then implies the assertion. It suffices to consider $H^{(P_{k})}$. By Theorem \ref{thm:k_closure_char}, we have $H^{(P_{k})}=\mathrm{U}_{k}(F^{(k)})$ for some $F^{(k)}\le\Aut(B_{d,k})$. By virtue of Proposition \ref{prop:ukf_max_c}, we may assume that $\smash{F^{(k)}}$ satisfies \eqref{eq:C}. Since $H$ is non-discrete, so is $\smash{H^{(P_{k})}=\mathrm{U}_{k}(F^{(k)})}$. Therefore, $F^{(k)}$ does not satisfy \eqref{eq:D}, see Proposition \ref{prop:ukf_discrete}. Hence, in view of the local action of $H$ and Proposition \ref{prop:ukf_transitive}, the group $\smash{\pi_{w}F^{(k)}_{T_{\omega}}}$ is non-trivial and thus transitive by Proposition \ref{prop:ukf_subnormal} for all $\smash{w=(\omega_{1},\ldots,\omega_{k-1})\in\Omega^{(k-1)}}$ and $\omega\in\Omega\backslash\{\omega_{1}\}$. Now, let $x\in V(T_{d})$. For every $\omega\in\Omega$ pick $w=(\omega_{1},\ldots,\omega_{k-2},\omega)\in\Omega^{(k-1)}$. Let $y\in V(T_{d})$ be such that $x=y_{w}$. Since $\smash{\pi_{w}F^{(k)}_{T_{\omega'}}}$ is transitive for every $\omega'\in\Omega\backslash\{\omega_{1}\}$ we conclude that $(H^{(P_{k})})^{+_{k}}$ is locally $2$-transitive at $x$. So Proposition \ref{prop:isomorphism_conjugate} applies.
\end{proof}

\begin{example}
Theorem \ref{thm:bew_non_isomorphism} applies to $\PGL(2,\bbQ_{p})\!\le\!\Aut(T_{p+1})$ for any prime $p$ by Lemma \ref{lem:subnormal_transitive} below. In fact, the local action is given by $\PGL(2,\bbF_{p})\!\curvearrowright\!\mathrm{P}^{1}(\bbF_{p})$, point stabilizers of which act like $\AGL(1,p)\!=\!\bbF_{p}^{\ast}\ltimes\bbF_{p}\curvearrowright\bbF_{p}$. Retaining the notation of \cite[Section 4.1]{BEW15}, an involutive inversion in $\PGL(2,\bbQ_{p})$ is given by
\begin{displaymath}
 \sigma:=\begin{bmatrix}0 & 1 \\ p & 0\end{bmatrix} \quad\text{with}\quad \sigma^{2}=\begin{bmatrix}p & 0 \\ 0 & p\end{bmatrix}=\begin{bmatrix}1 & 0 \\ 0 & 1\end{bmatrix}.
\end{displaymath}
Indeed, $\sigma$ swaps the vertices $v$ and $\mathbf{L}_{p}$.
\end{example}

\begin{lemma}\label{lem:subnormal_transitive}
Let $F\le\Sym(\Omega)$ be $2$-transitive. If $|\Omega|-1$ is prime then every non-trivial subnormal subgroup of $F_{\omega}$ ($\omega\in\Omega$) acts transitively on $\Omega\backslash\{\omega\}$.
\end{lemma}

\begin{proof}
Since $F_{\omega}$ acts transitively on $\Omega\backslash\{\omega\}$, which has prime order, $F_{\omega}$ is primitive. So every non-trivial normal subgroup of $F_{\omega}$ acts transitively on $\Omega\backslash\{\omega\}$. Iterate.
\end{proof}


\begin{example}\label{ex:bew_example}
The proof of Theorem \ref{thm:bew_non_isomorphism} shows that the assumptions on $F$ can be replaced with asking that $(H^{(P_{k})})^{+_{k}}$ be locally transitive with distinct point stabilizers, which may be feasible to check in a given example.

For instance, let $F\!\le\!\mathrm{Sym}(\Omega)$ be transitive with distinct point stabilizers. Assume that $F$ preserves a non-trivial partition $\calP:\Omega=\bigsqcup_{i\in I}\Omega_{i}$ of $\Omega$ and that it is generated by its block stabilizers, i.e. $F=\langle\{F_{\Omega_{i}}\mid i\in I\}\rangle$. 

Let $p:\Omega\to I$ be such that $\omega\in\Omega_{p\omega}$ for all $\omega\in\Omega$. Inductively define groups $\smash{F^{(k)}\le\Aut(B_{d,k})}$ by $\smash{F^{(1)}:=F}$ and $\smash{F^{(k+1)}\!:=\!\Phi_{k}(F^{(k)},\calP)}$, and check that
\begin{enumerate}[(i)]
 \item $C_{F^{(k)}}(\alpha,\Omega_{i})$ is non-empty for all $\smash{\alpha\in F^{(k)}}$ and $i\in I$,
 \item $C_{F^{(k)}}(\id,\Omega_{i})$ is non-trivial for all $i\in I$,
 \item\label{item:bew_example_proper} $F^{(k+1)}\lneq\Phi(F^{(k)})$, and
 \item\label{item:bew_example_action} $\smash{\pi_{w}F^{(k)}_{T_{\omega}}\!=\!F_{\Omega_{p\omega_{k-1}}}}$ for all $\omega\!\in\!\Omega$ and $w\!=\!(\omega_{1},\ldots,\omega_{k-1})\!\in\!\Omega^{(k-1)}$ with~$\omega_{1}\!\notin\!\Omega_{p\omega}$.
\end{enumerate}
In particular $F^{(k)}$ satisfies \eqref{eq:C} but not \eqref{eq:D} for all $k\in\bbN$. Set $H:=\bigcap_{k\in\bbN}\mathrm{U}_{k}(F^{(k)})$. By the above, $H$ is non-discrete and contains both $D(F)$ and $\mathrm{U}_{1}(\{\id\})$. Hence Theorem \ref{thm:k_closure_char} applies and we have $H^{(P_{k})}=\mathrm{U}_{k}(F^{(k)})$. From Item \ref{item:bew_example_proper}, we conclude that the $H^{(P_{k})}$ ($k\in\bbN$) are pairwise distinct. Given that $(H^{(P_{k})})^{+_{k}}$ locally acts like~$F$ due to Item \ref{item:bew_example_action}, the $\smash{(H^{(P_{k})})^{+_{k}}}$ $(k\in\bbN)$ are hence pairwise non-isomorphic.
\end{example}

\newpage
\subsection{A View on the Weiss Conjecture}\label{sec:view_weiss}

The Weiss conjecture states that there are only finitely many conjugacy classes of discrete, vertex-transitive, locally primitive subgroups of $\Aut(T_{d})$ for a given $d\in\bbN_{\ge 3}$. We now study the universal group construction in the discrete case and thereby offer a new view on this conjecture: Under the additional assumption that each group contains an involutive inversion, it suffices to show that for every primitive $F\le\Sym(\Omega)$ there are only finitely many $\smash{\widetilde{F}\le\Aut(B_{d,k})}$ $(k\in\bbN)$ with $\smash{\pi\widetilde{F}=F}$ and which satisfy \eqref{eq:CD} in a minimal fashion; see Definition~\ref{def:cd_dimension} and the discussion thereafter.


\vspace{0.2cm}
The following consequence of Theorem \ref{thm:k_closure_char} identifies certain groups relevant to the Weiss conjecture as universal groups for local actions satisfying condition \eqref{eq:CD}.

\begin{corollary}\label{cor:universal_discrete}
Let $H\le\Aut(T_{d})$ be discrete, locally transitive and contain an involutive inversion. Then $\smash{H=\mathrm{U}_{k}^{(l)}(F^{(k)})}$ for some $k\in\bbN$, a labelling $l$ of $T_{d}$ and $\smash{F^{(k)}\le\Aut(B_{d,k})}$ satisfying \eqref{eq:CD} which is isomorphic to the $k$-local action of $H$.
\end{corollary}

\begin{proof}
Discreteness of $H$ implies Property $(P_{k})$ for every $k\in\bbN$ such that stabilizers in $H$ of balls of radius $k$ in $T_{d}$ are trivial. Then apply Theorem \ref{thm:k_closure_char}.
\end{proof}

Therefore, studying the class of groups given in Corollary \ref{cor:universal_discrete} reduces to studying subgroups $F\le\Aut(B_{d,k})$ $(k\in\bbN)$ which satisfy \eqref{eq:CD} and for which $\pi F$ is transitive. By Corollary \ref{cor:ukf_cd_iso}, any two conjugate such groups yield isomorphic universal groups. In this sense, it suffices to examine conjugacy classes of subgroups of $\Aut(B_{d,k})$. This can be done computationally using the description of conditions \eqref{eq:C} and \eqref{eq:D} developed in Section \ref{sec:comp_disc}, using e.g. \cite{GAP4}.

\begin{example}\label{ex:s3}
Consider the case $d\!=\!3$. By \cite{Tut47}, \cite{Tut59} and \cite{DM80}, there are, up to conjugacy, seven discrete, vertex-transitive and locally transitive subgroups of $\Aut(T_{3})$. We denote them by $G_{1}$, $G_{2}$, $G_{2}^{1}$, $G_{3}$, $G_{4}$, $G_{4}^{1}$ and $G_{5}$. The subscript $n$ determines the isomorphism class of the vertex stabilizer, whose order is $3\cdot 2^{n-1}$. A group contains an involutive inversion if and only if it has no superscript. The minimal order of an inversion in $G_{2}^{1}$ and $G_{4}^{1}$ is $4$. See also \cite{CL89}. By Corollary~\ref{cor:universal_discrete}, the groups $G_{n}$ $(n\!\in\!\{1,\ldots,5\})$ are of the form $\mathrm{U}_{k}(F)$. We recover their local actions in the following table of conjugacy class representatives of subgroups $F$ of $\Aut(B_{3,2})$ and $\Aut(B_{3,3})$ which satisfy \eqref{eq:C} and project onto a transitive subgroup of $S_{3}$. The list is complete for $k=2$, and for $k=3$ in the case of \eqref{eq:CD}.

\vspace{0.1cm}
\bgroup
\centerline{
\begin{tabular}{c|c|c|c|c|c|c}
Description of $F$ & $k$ & $\pi F$ & $|F|$ & \eqref{eq:C} & \eqref{eq:D} & i.c.c. \\ \cline{1-7}\cline{1-7}
$\Phi(A_{3})$ & 2 & $A_{3}$ & 3 & yes & yes & yes \\ \cdashline{1-7}
$\Gamma(S_{3})$ & 2 & $S_{3}$ & 6 & yes & yes & yes \\
$\Delta(S_{3})$ & 2 & $S_{3}$ & 12 & yes & yes & yes \\
$\Pi(S_{3},\mathrm{sgn},\{0,1\})$ & 2 & $S_{3}$ & 24 & yes & no & \textbf{no} \\
$\Pi(S_{3},\mathrm{sgn},\{1\})$ & 2 & $S_{3}$ & 24 & yes & no & \textbf{yes} \\
$\Phi(S_{3})$ & 2 & $S_{3}$ & 48 & yes & no & \textbf{no} \\ \cline{1-7}\cline{1-7}
Description of $F$ & $k$ & $\pi_{2}F$ & $|F|$ & \eqref{eq:C} & \eqref{eq:D} & i.c.c. \\ \cline{1-7}\cline{1-7}
$\Gamma_{2}(\Pi(S_{3},\mathrm{sgn},\{1\}))$ & 3 & $\Pi(S_{3},\mathrm{sgn},\{1\})$ & 24 & yes & yes & yes \\
$\Sigma_{2}(\Pi(S_{3},\mathrm{sgn},\{1\}),K_{2})$ & 3 & $\Pi(S_{3},\mathrm{sgn},\{1\})$ & 48 & yes & yes & yes \\
\end{tabular}
}
\egroup
\vspace{0.1cm}

\noindent
The column labelled ``i.c.c.'' records whether $F$ admits an involutive compatibility cocycle. This can be determined in \cite{GAP4} and is automatic in the case of \eqref{eq:CD}. The group $\Pi(S_{3},\mathrm{sgn},\{1\})$ of Proposition \ref{prop:pif} admits an involutive compatibility cocycle $z$ which we describe as follows: Say $\Omega\!:=\!\{1,2,3\}$. Let $t_{i}\!\in\!\Sym(\Omega)$ be the transposition which fixes $i$, and let $\tau_{i}\!\in\! \Pi(S_{3},\mathrm{sgn},\{1\})$ be the element whose $1$-local action is $t_{i}$ everywhere except at $b_{i}$. Then $\Pi(S_{3},\mathrm{sgn},\{1\})=\langle\tau_{1},\tau_{2},\tau_{3}\rangle$. Further, let $\kappa_{i}\in \Pi(S_{3},\mathrm{sgn},\{1\})\cap\ker\pi$ be the non-trivial element with $\sigma_{1}(\kappa_{i},b_{i})=e$. We then have $z(\tau_{i},i)=\kappa_{i-1}$ and $z(\tau_{i},j)=\tau_{i}\kappa_{j}$ for all distinct $i,j\in\Omega$, with cyclic notation.

The kernel $K_{2}$ is the diagonal subgroup of $\smash{\bbZ\!/2\bbZ^{3\cdot(3-1)}\cong\ker\pi_{2}\le\Aut(B_{3,3})}$. Using the above, we conclude $G_{1}=\mathrm{U}_{1}(A_{3})$, $G_{2}=\mathrm{U}_{2}(\Gamma(S_{3}))$, $G_{3}=\mathrm{U}_{2}(\Delta(S_{3}))$, $G_{4}=\mathrm{U}_{3}(\Gamma_{2}(\Pi(S_{3},\mathrm{sgn},\{1\})))$ and $G_{5}=\mathrm{U}_{3}(\Sigma_{2}(\Pi(S_{3},\mathrm{sgn},\{1\}),K_{2}))$.
\end{example}

\begin{question}
Can the groups $G_{2}^{1}$ and $G_{4}^{1}$ be described as universal groups with prescribed local actions on edge neighbourhoods that prevent involutive inversions?
\end{question}

The long standing Weiss conjecture \cite{Wei78} states that there are only finitely many conjugacy classes of discrete, vertex-transitive, locally primitive subgroups of $\Aut(T_{d})$ for a given $d\in\bbN_{\ge 3}$. Poto{\v{c}}nic--Spiga--Verret \cite{PSV12} show that a permutation group $F\le\Sym(\Omega)$, for which there are only finitely many conjugacy classes of discrete, vertex-transitive subgroups of $\Aut(T_{d})$ that locally act like $F$, is necessarily semiprimitive, and conjecture the converse. Promising partial results were obtained in the same article as well as by Giudici--Morgan in \cite{GM14}.

Corollary \ref{cor:universal_discrete} suggests to restrict to discrete, locally semiprimitive subgroups of $\Aut(T_{d})$ containing an involutive inversion.

\begin{conjecture}\label{conj:weak_weiss_psv}
Let $F\le\Sym(\Omega)$ be semiprimitive. Then there are only finitely many conjugacy classes of discrete subgroups of $\Aut(T_{d})$ which locally act like $F$ and contain an involutive inversion. 
\end{conjecture}

For a transitive permutation group $F\le\Sym(\Omega)$, let $\calH_{F}$ denote the collection of subgroups of $\Aut(T_{d})$ which are discrete, locally act like $F$ and contain an involutive inversion. Then the following definition is meaningful by Corollary \ref{cor:universal_discrete}.

\begin{definition}\label{def:cd_dimension}
Let $F\le\Sym(\Omega)$ be transitive. Define
\begin{displaymath}
 \dim_{\mathrm{CD}}(F)\!:=\!\max_{H\in\calH_{F}}\!\min\left\{k\!\in\!\bbN\left|\exists F^{(k)}\!\in\!\Aut(B_{d,k}) \text{ with \eqref{eq:CD}}:\ H\!=\!\mathrm{U}_{k}(F^{(k)})\right.\right\}
\end{displaymath}
if the maximum exists and $\dim_{\mathrm{CD}}(F)=\infty$ otherwise.
\end{definition}

Given Definition \ref{def:cd_dimension}, Conjecture \ref{conj:weak_weiss_psv} is equivalent to asserting that $\dim_{\mathrm{CD}}(F)$ is finite whenever $F\le\Sym(\Omega)$ is semiprimitive. The remainder of this section is devoted to determining $\dim_{\mathrm{CD}}$ for certain classes of transitive permutation groups.

\begin{proposition}\label{prop:cd_dim_1}
Let $F\le \Sym(\Omega)$ be transitive. Then $\dim_{\mathrm{CD}}(F)=1$ if and only if $F$ is regular.
\end{proposition}

\begin{proof}
If $F$ is regular, then $\dim_{\mathrm{CD}}(F)=1$ by Proposition \ref{prop:ukf_structure_reg}. Conversely, if $\dim_{\mathrm{CD}}(F)=1$ then $\mathrm{U}_{2}(\Delta(F))=\mathrm{U}_{1}(F)=\mathrm{U}_{2}(\Gamma(F))$. Hence $\Gamma(F)\cong\Delta(F)$ which implies that $F_{\omega}$ is trivial for all $\omega\in\Omega$. That is, $F$ is regular.
\end{proof}

The next proposition provides a large class of primitive groups of dimension~$2$. It relies on the following relations between various characteristic subgroups of a finite group. Recall that the socle of a finite group is the subgroup generated by its minimal normal subgroups, which form a direct product.

\begin{lemma}\label{lem:radicals}
Let $G$ be a finite group. Then the following are equivalent.
\begin{enumerate}[(i)]
 \item\label{item:radicals_socle} The socle $\mathrm{soc}(G)$ has no abelian factor.
 \item\label{item:radicals_solvable} The solvable radical $\calO_{\infty}(G)$ is trivial.
 \item\label{item:radicals_nilpotent} The nilpotent radical $\mathrm{Fit}(G)$ is trivial.
\end{enumerate}
\end{lemma}

\begin{proof}
If $\mathrm{soc}(G)$ has no abelian factor then $\calO_{\infty}(G)$ is trivial: A non-trivial solvable normal subgroup of $G$ would contain a minimal solvable normal subgroup of $G$ which is necessarily abelian. Next, \ref{item:radicals_solvable} implies \ref{item:radicals_nilpotent} as every nilpotent group is solvable. Finally, if $\mathrm{soc}(G)$ has an abelian factor then $G$ contains a (minimal) normal abelian, hence nilpotent subgroup. Thus \ref{item:radicals_nilpotent} implies \ref{item:radicals_socle}.
\end{proof}

\begin{proposition}\label{prop:cd_dim_2}
Let $F\le \Sym(\Omega)$ be primitive, non-regular and assume that $F_{\omega}$ has trivial nilpotent radical for all $\omega\in\Omega$. Then $\dim_{\mathrm{CD}}(F)=2$.
\end{proposition}

\begin{proof}
Suppose that $F^{(2)}\le\Aut(B_{d,2})$ satisfies \eqref{eq:C} and that the sequence
\begin{displaymath}
\xymatrix{
  1 \ar[r] & \ker\pi \ar[r] & F^{(2)} \ar[r]^-{\pi} & F \ar[r] & 1
}
\end{displaymath}
is exact. Fix $\omega_{0}\in\Omega$. Then $\ker\pi\le\prod_{\omega\in\Omega}F_{\omega}\cong F_{\omega_{0}}^{d}$. Since $F^{(2)}$ satisfies \eqref{eq:C}, we have $\pr_{\omega}(\ker\pi)\unlhd F_{\omega_{0}}$ for all $\omega\in\Omega$, and since $F$ is transitive these projections all coincide with the same $N\unlhd F_{\omega_{0}}$. Now consider $\smash{F^{(2)}_{T_{\omega}}=\ker\pr_{\omega}|_{\ker\pi}\unlhd\ker\pi}$ for some $\omega\in\Omega$. Either $\smash{F^{(2)}_{T_{\omega}}}$ is trivial, in which case $\smash{F^{(2)}}$ has \eqref{eq:CD}, or $\smash{F^{(2)}_{T_{\omega}}}$ is non-trivial. In the latter case, say $\smash{N_{\omega,\omega'}:=\pr_{\omega'}F^{(2)}_{T_{\omega}}}$ is non-trivial for some $\omega'\in\Omega$. Then $N_{\omega,\omega'}$ is subnormal in $F_{\omega_{0}}$ as $N_{\omega,\omega'}\unlhd N\unlhd F_{\omega_{0}}$ and therefore has trivial nilpotent radical. The Thompson-Wielandt Theorem \cite{Tho70}, \cite{Wie71} (cf. \cite[Theorem 2.1.1]{BM00a}) now implies that there is no $F^{(k)}\le\Aut(B_{d,k})$ $(k\ge 3)$ which satisfies $\pi_{2}F^{(k)}=F^{(2)}$ and \eqref{eq:CD}. Thus $\dim_{\mathrm{CD}}(F)\le 2$. Equality holds by Proposition \ref{prop:cd_dim_1}.
\end{proof}

Proposition~\ref{prop:cd_dim_2} applies to $\Alt(d)$ and $\Sym(d)$ ($d\ge 6$) whose point stabilizers have non-abelian simple socle $\Alt(d-1)$. It also applies to primitive groups of O'Nan-Scott type (TW) and (HS), whose point stabilizers have trivial solvable radical \cite[Theorem 4.7B]{DM96} and simple non-abelian socle \cite{LPS88} respectively.

\begin{example}
By Example \ref{ex:s3}, we have $\dim_{\mathrm{CD}}(S_{3})\ge 3$. The article \cite{DM80} shows that in fact $\dim_{\mathrm{CD}}(S_{3})=3$.
\end{example}

To contrast the primitive case, we show that imprimitive wreath products have dimension at least $3$, illustrating the use of involutive compatibility cocycles. Recall that for $F\le\mathrm{Sym}(\Omega)$ and $P\le\mathrm{Sym}(\Lambda)$ the wreath product $\smash{F\wr P:=F^{\vert\Lambda\vert}\rtimes P}$ admits a natural imprimitive action on $\Omega\times\Lambda$ with the partition $\bigsqcup_{\lambda\in\Lambda}\Omega\times\{\lambda\}$, namely $((a_{\lambda})_{\lambda},\sigma)\cdot(\omega,\lambda'):=(a_{\sigma(\lambda')}\omega,\sigma\lambda')$.

\begin{proposition}\label{prop:ukf_discrete_dim_wreath}
Let $\Omega$ and $\Lambda$ be finite sets of size at least $2$. 
Furthermore, let $F\le\mathrm{Sym}(\Omega)$ and $P\le\mathrm{Sym}(\Lambda)$ be transitive. Then $\dim_{\mathrm{CD}}(F\wr P)\ge 3$.
\end{proposition}

\begin{proof}
We define a subgroup $W(F,P)\le\Aut(B_{|\Omega\times\Lambda|,2})$ which projects onto $F\wr P$, satisfies \eqref{eq:C}, does not satisfy \eqref{eq:D} but admits an involutive compatibility cocycle. This suffices by Lemma \ref{prop:ukf_gamma_k}. For $\lambda\in\Lambda$, let $\iota_{\lambda}$ denote the $\lambda$-th embedding of $F$ into $\smash{F\wr P=\big(\prod_{\lambda\in\Lambda}F\big)\rtimes P}$. Recall the map $\gamma$ from Section \ref{sec:ukf_examples_k=2} and consider
\begin{displaymath}
 \gamma_{\lambda}:F\to\Aut(B_{|\Omega\times\Lambda|,2}),\ a\mapsto (\iota_{\lambda}(a),((\iota_{\lambda}(a))_{(\omega,\lambda)},(\id)_{(\omega,\lambda'\neq\lambda)})),
\end{displaymath}
\vspace{-0.5cm}
\begin{displaymath}
 \gamma_{\lambda}^{(2)}:F\to\Aut(B_{|\Omega\times\Lambda|,2}),\ a\mapsto (\id,((\id)_{(\omega,\lambda)},(\iota_{\lambda}(a))_{(\omega,\lambda'\neq\lambda)})).
\end{displaymath}

\noindent
Furthermore, let $\iota$ denote the embedding of $P$ into $F\wr P$. We define
\begin{displaymath}
 W(F,P):=\langle \gamma_{\lambda}(a),\gamma_{\lambda}^{(2)}(a),\gamma(\iota(\varrho))\mid \lambda\in\Lambda,\ a\in F,\ \varrho\in P\rangle.
\end{displaymath}
By construction, $W(F,P)$ does not satisfy \eqref{eq:D}. To see that $W(F,P)$ admits an involutive compatibility cocycle, we first determine its group structure. Consider the subgroups $V:=\langle\gamma_{\lambda}(a)\mid \lambda\in\Lambda,\ a\in F\rangle$ and $\smash{\overline{V}:=\langle \gamma_{\lambda}^{(2)}(a)\mid \lambda\in\Lambda,\ a\in F\rangle}$. Then $W(F,P)=\langle V,\overline{V},\Gamma(\iota(P))\rangle$. Observe that $V\cong F^{|\Lambda|}$ and $\overline{V}\cong F^{|\Lambda|}$ commute, intersect trivially and that $\Gamma(\iota(P))$ permutes the factors of each product. Hence
\begin{displaymath}
 W(F,P)\cong(V\times\overline{V})\rtimes P\cong(F^{|\Lambda|}\times F^{|\Lambda|})\rtimes P.
\end{displaymath}
An involutive compatibility cocycle $z$ of $W(F,P)$ may now be defined by setting
\begin{displaymath}
 z(\gamma_{\lambda}(a),(\omega,\lambda')):=\begin{cases} \gamma_{\lambda}(a) & \lambda=\lambda' \\ \gamma_{\lambda}^{(2)}(a) & \lambda\neq \lambda'\end{cases}, \text{ } z(\gamma_{\lambda}^{(2)}(a),(\omega,\lambda')):=\begin{cases} \gamma_{\lambda}^{(2)}(a) & \lambda=\lambda' \\ \gamma_{\lambda}(a) & \lambda\neq\lambda'\end{cases} 
\end{displaymath}
for all $\lambda\in\Lambda$, $a\in F$, and $z(\gamma(\iota(\varrho)),(\omega,\lambda)):=\gamma(\iota(\varrho))$ for all $\varrho\in P$. In fact, the map $z$ extends to an involutive compatibility cocycle of $V\times\overline{V}\le W(F,P)$ which in turn extends to an involutive compatibility cocycle of $W(F,P)$.
\end{proof}

\vspace{0.2cm}
\bibliographystyle{amsalpha}
\bibliography{bm_semi}

\end{document}